\documentclass{siamart190516}
\usepackage[utf8]{inputenc}
\usepackage{geometry}
\usepackage{graphicx,subcaption}
\usepackage{amsmath,amssymb}
\usepackage{mathtools}

\usepackage{enumerate}
\usepackage{cases}

\usepackage{tikz}
\usetikzlibrary{arrows,calc,plotmarks,positioning}
\usetikzlibrary{decorations.pathreplacing}
\tikzset{>=angle 60}

\usepackage{float}

\usepackage{LatexDefinitions}

\title{Linearized optimal transport on manifolds}
\author{Cl\'ement Sarrazin\thanks{RAPSODI team, INRIA de Lille, France}
	\and Bernhard Schmitzer\thanks{Campus Institute Data Science, Göttingen University, Germany}}
\date{\today}

\begin{document}
\maketitle
\begin{abstract}
Optimal transport is a geometrically intuitive, robust and flexible metric for sample comparison in data analysis and machine learning. Its formal Riemannian structure allows for a local linearization via a tangent space approximation. This in turn leads to a reduction of computational complexity and simplifies combination with other methods that require a linear structure. Recently this approach has been extended to the unbalanced Hellinger--Kantorovich (HK) distance. In this article we further extend the framework in various ways, including measures on manifolds, the spherical HK distance, a study of the consistency of discretization via the barycentric projection, and the continuity properties of the logarithmic map for the HK distance.
\end{abstract}
\section{Introduction}
\subsection{Motivation}
\paragraph{Optimal transport in data analysis} Optimal transport (OT) induces a geometrically intuitive and robust metric on the set of probability measures over a metric space, with many applications in geometry, stochastics and PDE analysis. OT is also becoming an increasingly popular tool in data analysis and machine learning, including numerical applications in image classification \cite{feydy2017optimal}, inverse problems \cite{engquist_froese_yang_2016}, and deep learning \cite{pmlr-v70-arjovsky17a,ChizatNIPS2018}.
This is also made possible by the development of increasingly efficient corresponding algorithms.
We refer to \cite{villani2009optimal,santambrogio2015optimal,GalichonOTEconomics} for general introductions to the topic of optimal transport and to \cite{PeyreCuturiCompOT} for an overview on computational techniques.

Consider a collection of probability measures $\{\nu_1,\ldots,\nu_N\} \subset \prob(\R^n)$, each $\nu_i$ representing one data sample, and assume that the 2-Wasserstein distance $\W_2$ is a meaningful metric for their comparison.
Our task is to analyze this dataset. This means we might have to train a classifier, according to some given labels $(l_i,\ldots,l_N)$; to find a meaningful clustering; or to perform a dimensionality reduction and to visualize the dominant types of variation within the samples.

To solve this task, we could, in principle, start by computing the full matrix of pairwise Wasserstein distances $\W_2(\nu_i,\nu_j)$ between all samples.
However, this would require the solution of $N(N-1)/2=O(N^2)$ optimal transport problems. Due to the quadratic growth in $N$ this will quickly become a computational problem, despite recent progress in algorithmic efficiency. Moreover, any of the aforementioned tasks is difficult to perform in a general (non-linear) metric space, based solely on the knowledge of the pairwise distances.

\paragraph{Linearized optimal transport} Instead, in \cite{wang2013linear} it was proposed to leverage the formal Riemannian structure of the 2-Wasserstein distance over $\R^n$ (see for instance \cite{otto2001geometry,ambrosio2005gradient,LottWassersteinRiemannian2008}) and to formally conduct a local linearization by a tangent space approximation.
For this purpose, one chooses a suitable reference measure $\mu \in \prob(\R^n)$ and maps all samples $\nu_i$ to the tangent space at $\mu$ via the logarithmic map, $v_i \assign \Log_\mu(\nu_i)$. The tangent space can be identified with (curl free) vector fields in $\LL^2(\R^n,\R^n;\mu)$. Pairwise distances between the embeddings $v_i$ in the tangent space are then used to approximate distances between samples. This is called the linearized optimal transport distance (LOT).
Beyond mere approximation of pairwise distances, the Hilbertian structure of $\LL^2(\R^n,\R^n;\mu)$ allows the application of a wide variety of standard data analysis techniques to tackle any of the aforementioned potential analysis tasks. For instance, principal component analysis (PCA) can be used for dimensionality reduction and to extract the dominant modes of variation. Interpolation between samples $\nu_i$ is replaced by linear interpolation between the embeddings $v_i$ and these trajectories can be visualized by employing the exponential map (a left-inverse of the logarithmic map). Such curves are called generalized geodesics, introduced in \cite{ambrosio2005gradient}, we also refer the reader to \cite{NennaPass2022} for a recent article on the subject.
In one dimension, $n=1$, this embedding reduces to the cumulative distance transform \cite{Park2018}, which drastically reduces the computational complexity. This can still be applied to higher dimensional datasets via the slicing trick \cite{KolouriSlicedLOT2016}.

The empirical success in applications has also lead to an increased interest in the theoretical properties of the method.
How good is the approximation of the pairwise distances by local linearization? In the conventional, finite-dimensional Riemannian setting this question is tied to the curvature of the manifold. The results do not apply to $\W_2$ as here the Riemannian structure is merely formal and additional assumptions are necessary to establish some approximation guarantees.
For relatively simple families of samples this is addressed in \cite{MoosmuellerLinOTSheared2022}.
More general continuity properties of the logarithmic map are established, for instance, in \cite{gigli11,delalande21}.
Another relevant question is the choice of the reference measure $\mu$. From a Riemannian perspective it should be chosen as the Riemannian center of mass (or barycenter) of the sample set $(\nu_i)_i$, which has been established for the Wasserstein distance in \cite{WassersteinBarycenter}. While this can, in principle, be approximated numerically, for instance via \cite{BenamouIterativeBregman2015}, this is numerically not feasible for large number of samples. In practice an approximation by (a few or even a single iteration of) Lloyd's algorithm works well, as highlighted in \cite{cuturi2014fast} and \cite{alvarez2016fixed}, and even naive linear averaging of the measures (e.g.~as densities in some $\LL^p$ space) tends to work acceptably.
\cite{BigotKleinBarycenter2018} studies the consistency of the barycenter as the number of samples tends to infinity, and \cite{Bigot2017} studies the consistency of tangent space principal component analysis in the one-dimensional, $n=1$, setting.

\paragraph{Unbalanced optimal transport}
Recently, optimal transport has also been extended to comparing measures of varying mass, the most prominent example maybe being the Hellinger--Kantorovich distance ($\HK$) \cite{KMV-OTFisherRao-2015,ChizatOTFR2015,liero2018optimal}. From the perspective of data analysis this can explicitly model growth or shrinkage, but can also provide additional robustness towards general mass fluctuations. The linearized optimal transport framework has been extended to the Hellinger--Kantorovich distance in \cite{cai2022linearized}.

\subsection{Contribution and outline}
In this article we further extend the toolbox of linearized optimal transport and provide some new theoretical results.
In particular we consider the case where the base space $X$ is no longer flat $\R^n$, but itself a non-linear Riemannian manifold.

Sections \ref{sec:WBasic} and \ref{sec:WRiemann} recall basic properties of the $\W_2$ distance and its local linearization. In Section \ref{sec:LinW2} we show how to extend this to the case of manifold base spaces $X$.
In Section \ref{sec:HKBasic} we recall basic properties of the HK metric. In Section \ref{sec:HKMaps} we give a slight extension of a result from \cite{gallouet2021regularity} concerning the existence of optimal HK transport maps on manifold base spaces $X$. In Section \ref{sec:LinHK} we introduce the local linearization for manifold base spaces $X$. Beyond the extension to base manifolds, compared to \cite{cai2022linearized} the formulas have been simplified by using primal-dual optimality conditions for the HK transport problem. In addition we give a formula for the logarithmic map based on a suitable dual variable. The latter will be useful when determining the range of the logarithmic map (see Section \ref{sec:convexity_range}).
In Section \ref{sec:LinSHK} we introduce the local linearization of the spherical Hellinger--Kantorovich (SHK) distance, a projection of the HK metric structure which was introduced in \cite{laschos2019geometric}. Working with SHK might be preferable, when all samples are indeed probability measures, but still present non-local mass fluctuations. The latter make it difficult to work with the standard Wasserstein distance, but working with HK will introduce a bias towards smaller masses. The SHK metric may be a suitable way to attain robustness to local mass fluctuations without the bias induced by the HK metric, see Section \ref{sec:Numerics} for an illustration. Results from \cite{laschos2019geometric} show how SHK can be computed directly from HK and we extend these results to the logarithmic and exponential maps.
In Section \ref{sec:convexity_range} we provide a discussion on the range of the logarithmic map on the (formal) optimal transport manifolds.

In Section 5 we give several approximation results. First (Section \ref{sec:LinHKToLinW}) we show that linearized HK converges to standard linearized $\W_2$, as the length scale parameter in HK tends to infinity, in agreement with the fact that HK converges to $\W_2$ \cite{liero2018optimal}.
In Section \ref{sec:BarycentricW} we show that `barycentric projection', a popular technique for numerical approximation of linearized optimal transport for $\W_2$, does actually converge to the true logarithmic map in the limit of increasing discretization resolution.
In Section \ref{sec:BarycentricHK} we give the corresponding result for the HK distance. This extends beyond discrete approximation and generally provides a continuity (or lack thereof) result for the logarithmic map of the HK metric. We illustrate this result with several examples.

Section \ref{sec:Numerics} provides some numerical experiments, in particular a comparison between linearized HK and SHK, an example of linearized optimal transport over a sphere as base manifold, and some examples on the convexity of the range of the logarithmic map.

\subsection{Notation}
\label{sec:Notation}
\paragraph{Riemannian manifolds}
Throughout the article $\Surf$ will be a connected complete $n$-di\-men\-sional Riemannian manifold. 
For $x \in \Surf$ denote by $T_x \Surf$ the tangent space of $\Surf$ at $x$, and by $TX \assign \bigcup_{x \in \Surf} \{x\} \times T_x \Surf$ the tangent bundle of $\Surf$. For $v,w \in T_x \Surf$ we denote by $\langle v,w\rangle_x$ the corresponding Riemannian inner product and by $\|v\|_x$ the induced norm.

We will denote by $\Log^X_x$ the logarithmic map on $\Surf$ at $x$ and likewise, by $\Exp^X_x$ the exponential map at $x$.
Intuitively, $\Exp^X_x(v)$ yields the point on $\Surf$ that one reaches at time 1 when one embarks on a constant speed geodesic starting at $x$ with velocity vector $v$. Conversely, $\Log^X_x(y)$ yields the smallest such initial velocity vector $v$ such that one ends up at $y$, when starting from $x$ along this geodesic. In particular, $\Exp^X_x(\Log^X_x(y))=y$, i.e.~the exponential map is a left inverse of the logarithmic map (however, the image of $\Log^X_x$ might be strictly smaller than $T_xX$). In addition one has
\begin{equation}
\label{eq:LogMapBasic}
\norm{\Log^\Surf_x(y)}_x =d(x,y)  \qquad \tn{and} \qquad
\nabla_x d^2(x,y)=-2\Log^\Surf_x(y)
\end{equation}
whenever $\Log^\Surf_x$ is defined.
Because we are working on a complete and connected manifold, the Hopf-Rinow theorem guarantees that the exponential map is defined on the entire tangent bundle $TX$ (geodesics can be extended indefinitely before $t=0$ or after $t=1$) and that any two points in $\Surf$ can be connected by such a geodesic (so that it is surjective).

\begin{example}
When $\Surf$ is the `flat space' $\R^n$ (with the Euclidean metric), then $T_x \Surf = \R^n$ and  $\Log^X_x(y)=y-x$ is well defined for any $x$ and $y$ and $\Exp^X_x(v)=x+v$.
\end{example}

Finally, we denote the open ball of $\Surf$ of radius $r\geq 0$ and center $x\in\Surf$ by $B_\Surf(x,r)$.

\paragraph{Measures on metric spaces}
For a Polish metric space $\Surf$, we will denote by $\Borel(\Surf)$ the set of Borel-measurable functions on $\Surf$ (valued in $\bar{\R}=\R\cup\{-\infty,+\infty\}$), $\meas(\Surf)$ the signed Radon measures on $\Surf$, by $\measp(\Surf)$ the non-negative Radon measures and by $\prob(\Surf)$ the Radon probability measures. 

When $\Surf$ is a Riemannian manifold, denote by $\vol$ its volume measure. We will write $\measpl(X)$ for the non-negative measures on $X$ that are dominated by $\vol$ and similarly $\probl(X)$ for the respective probability measures. Intuitively, the volume measure will play the role of the Lebesgue measure on the curved space $X$. In particular, when $X$ is the flat space $\R^d$, $\vol=\Leb^d$ is the Lebesgue measure itself.

For a non-negative measure $\mu$ on $\Surf$ we will denote by $\LL^2(\Surf;\mu)$ the space of real functions on $X$ whose square is $\mu$-integrable (up to equality $\mu$-almost everywhere) and $\LL^2(\Surf,TX;\mu)$ the space of vector fields on $X$ whose squared norm (for the metric on $TX$) is $\mu$-integrable (again, up to equality $\mu$-almost everywhere). We may sometimes merely write $\LL^2(\mu)$ when the context is clear.

Finally, for $\mu\in\meas(\Surf)$ and a measurable map $T:\Surf\to Y$ (where $Y$ is a Polish space equipped with its Borel $\sigma$-algebra), the \emph{push-forward measure} of $\mu$ along $T$ is the measure $T_{\#}\mu$ on $Y$ characterized by
\begin{equation}
	\label{eq:push_forward}
	\int_Y\Phi(x_1)\diff(T_\#\mu)(x_1)=\int_X\Phi(T(x_0))\diff\mu(x_0)
\end{equation}
for all bounded measurable $\Phi : Y \to \R$.

Finally, let $K$ be a compact subset of $\Surf$. For the sake of simplicity, we will frequently assume that the support of measures is contained in $K$ to avoid technical compactness arguments.

\paragraph{Optimal transport}
We refer to \cite{villani2009optimal,santambrogio2015optimal,GalichonOTEconomics,PeyreCuturiCompOT} for detailed monographs on optimal transport, applications, and corresponding computational methods. In the following paragraph we collect some fundamental definitions and concepts that are required for the manuscript.
For two measures $\mu_0, \mu_1 \in \measp(\Surf)$ and a lower-semicontinuous cost function $c:\Surf\times \Surf\to\RCupInf$ the corresponding optimal transport problem is defined as
\begin{align}
	\label{eq:OT_primal}
	\Icost{c}(\mu_0,\mu_1) & \assign \inf \left\{
	\int_{\Surf \times \Surf} c(x_0,x_1)\,\diff \pi(x_0,x_1)~
	\middle|~ \pi \in \Pi(\mu_0,\mu_1) \right\} \\
\intertext{were the infimum is taken among all \emph{transport plans} between $\mu_0$ and $\mu_1$}
	\label{eq:OT_feasible}
	\Pi(\mu_0,\mu_1) & \assign\left\{\pi\in\meas_+(\Surf\times \Surf)~\middle|~\proj_{i\#} \pi = \mu_i \tn{ for } i=0,1\right\}
\end{align}
and $\proj_{i} :  \Surf \times \Surf \to X$ is the projection onto the $i$-th coordinate, $(x_0,x_1) \mapsto x_i$, and therefore $\proj_{i\#} \pi$ yields the corresponding marginal of $\pi$, see \eqref{eq:push_forward}.
When $\mu_0(\Surf) \neq \mu_1(X)$ then $\Pi(\mu_0,\mu_1) = \emptyset$ and therefore by convention $\Icost{c}(\mu_0,\mu_1)=+\infty$. Sometimes the optimal plan $\pi$ is induced by a Borel map $T : \Surf \to \Surf$ via
\begin{equation}
\label{eq:OTMap}
\pi=(\id,T)_{\#} \mu_0, \quad \tn{where} \quad (\id,T): x \mapsto (x,T(x)).
\end{equation}
In this case, $T$ is called an optimal transport map for the associated OT problem.

Since $\Icost{c}(\mu_0,\mu_1)$ is defined via a convex minimization problem of a linear objective over an affine set, one can derive a corresponding concave dual maximization problem:
\begin{align}
	\label{eq:OT_dual}
	\Icost{c}(\mu_0,\mu_1) & = \sup \left\{
	\int_{\Surf} \Phi_0(x_0)\,\diff \mu_0(x_0)+\int_{\Surf} \Phi_1(y)\,\diff \mu_1(y)
	~\middle|~(\Phi_0,\Phi_1)\in\Cont_c(\Surf) \right\}
\intertext{with the admissible set given by}
	\Cont_c(\Surf) & \assign\left\{(\Phi_0,\Phi_1)\in\Cont(\Surf)\times \Cont(\Surf)~\middle|~\forall(x_0,x_1)\in \Surf^2,~\Phi_0(x_0)+\Phi_0(y)\leq c(x_0,x_1)\right\}.
\end{align}
Dual maximizers are sometimes referred to as (Kantorovich) potentials.
Depending on the regularity of $c$ or compactness of the suppport, existence of such Kantorovich potentials may require the space $\Cont(\Surf)$ to be relaxed, e.g.~to $\LL^1(\Surf,\mu_i)$ (see for instance \cite[Theorem 1.40]{santambrogio2015optimal}).
Since $\mu_0$ and $\mu_1$ are non-negative measures, for fixed $\Phi_0$ it is trivial to maximize over $\Phi_1$. One obtains
\begin{align*}
	\Phi_1=\Phi_0^c:=\inf_{x_0\in \Surf}c(x_0,.)-\Phi_0(x_0).
\end{align*}
The function $\Phi_0^c$ is called the $c$-transform of $\Phi_0$ and a function obtained in that way is said to be $c$-concave. Likewise, for fixed $\Phi_1$ we can maximize over $\Phi_0$. When $c$ is symmetric, as all costs considered throughout this article are, both $c$-transforms coincide.
In \eqref{eq:OT_dual} one may therefore impose the additional constraint that the $\Phi_i$ are $c$-transforms of each other and therefore $c$-concave, and we will make that assumption for simplicity in the future.

\section{2-Wasserstein distance}
\subsection{Basic properties}
\label{sec:WBasic}
Setting $c=d^2$ as the squared distance in \eqref{eq:OT_primal}, $\Icost{c}$ becomes the squared 2-Wasserstein distance on suitable subsets of probability measures, e.g.~those with support contained in some compact set $K \subset \Surf$.
\begin{definition}[2-Wasserstein distance]
	\label{def:W2_primal}
	For $\mu_0, \mu_1 \in \prob(K)$, the corresponding 2-Wasserstein distance between $\mu_0$ and $\mu_1$ is defined as
	\begin{align}
		\label{eq:W2_primal}
		\W_2(\mu_0,\mu_1) & \assign \inf \left\{
		\int_{\Surf \times \Surf} d^2(x_0,x_1)\,\diff \pi(x_0,x_1)
		\middle| \pi \in \Pi(\mu_0,\mu_1) \right\}^{1/2}.
	\end{align}
	It is a distance on $\prob(K)$ and the resulting metric space $(\prob(K),W_2)$ is often called the 2-Wasserstein space (over $K$).
\end{definition}
The construction can be generalized to any $p \in [1,\infty)$. For a set $K$ that would be non-compact, an additional bound on the $p$-th moments of the $\mu_i$ needs to be imposed, so as to guarantee a finite transport cost.

Existence of a transport plan $\pi$ minimizing \eqref{eq:W2_primal} is a corollary of the general existence result for \eqref{eq:OT_primal}.
\begin{proposition}
	Minimizing $\pi$ for Problem \eqref{eq:W2_primal} exist.
\end{proposition}
Furthermore, compactness of $K$, together with the continuity of $c=d_X^2$ yields the existence of optimal Kantorovich potentials (which can then be taken to be $c$-concave).
\begin{proposition}[Dual transport problem] For $\mu_0,\mu_1\in\prob(K),$
\label{prop:W2_dual}
\begin{multline}
  \label{eq:W2_dual}
	\W_2^2(\mu_0,\mu_1) = \sup\bigg\{ \int_X \Phi_0\,\diff \mu_0+ \int_X \Phi_1\,\diff \mu_1
	\bigg| \Phi_0, \Phi_1 \in \Cont(\Surf),\\
	\Phi_0(x_0)+\Phi_1(x_1) \leq d^2(x_0,x_1)
	\tn{ for all } (x_0,x_1) \in \Surf \times \Surf \bigg\}
\end{multline}
Maximizing $\Phi_0,\Phi_1$ exist.
\end{proposition}

For $\Surf= \R^n$ and $\mu_0 \in \probl(X)$ a celebrated result by Brenier \cite{brenier1991polar} shows that the minimizing $\pi$ in \eqref{eq:W2_primal} is unique and concentrated on the graph of the gradient of a convex function.
\begin{theorem}[Brenier]
\label{thm:Brenier}
Let $\Surf=\R^n$ endowed with its flat metric, $\mu_0\! \in\! \probl(K)$ and ${\mu_1\in\prob(K)}$. Then,
\begin{enumerate}[(i)]
\item The minimizer $\pi$ of \eqref{eq:W2_primal} is unique and concentrated on the graph of a map $T : K \to K$, i.e.~it is given by $\pi=(\id,T)_\# \mu_0$ where $(\id,T) : x \mapsto (x,T(x))$.
\item The map $T$ satisfies $T=\id-\tfrac12 \nabla \Phi_0$ $\mu_0$-almost everywhere, where $\Phi_0$ is a corresponding dual optimizer of \eqref{eq:W2_dual}. $T$ is the gradient of the function $x \mapsto \tfrac12 (\|x\|^2-\Phi_0(x))$, which is convex.
\item Conversely, if $\Phi_0 \in \Cont(\Surf)$ such that $x \mapsto \tfrac12 (\|x\|^2-\Phi_0(x))$ is convex, then $T=\id-\tfrac12 \nabla \Phi_0$ is the optimal transport map between $\mu_0$ and $T_\# \mu_0$.
\end{enumerate}
\end{theorem}
It is easy to check that convexity of $x \mapsto \tfrac12 (\|x\|^2-\Phi_0(x))$ is equivalent to $c$-concavity of $\Phi_0$ in that case. At this point, the reader's choice between Rademacher's theorem (which still holds for more general costs $c$) or Alexandrov's theorem guarantees the differentiability almost-everywhere allowing to state the theorem. 

This result was later generalized to Riemannian manifolds by McCann \cite{mccann2001polar} (see Section \ref{sec:LinW2}) and to the Hellinger--Kantorovich distances, on $\R^n$ by Liero et al. \cite{liero2018optimal} and on Riemannian manifolds by Gallou\"et et al. \cite{gallouet2021regularity} (we give a further extension of this last result in Section \ref{sec:LinHK}).

Finally, when $\Surf\subset \R^n$ is (geodesically) convex, then $(\prob(\Surf),\W_2)$ is also a geodesic space. Geodesics in the latter space are obtained by lifting geodesics from the base space $X$ as formalized in the next result.
\begin{proposition}
Under the assumptions of Theorem \ref{thm:Brenier}, for $\mu_0 \in \probl(K)$, $\mu_1 \in \prob(\Surf)$, let $T$ be the corresponding optimal transport map. Then the curve
\begin{align}
\label{eq:OTGeodesic}
	[0,1] \ni t \mapsto \mu_t \assign ((1-t) \cdot \id + t \cdot T)_\# \mu_0 \in \prob(\Surf)
\end{align}
is the unique geodesic in $(\prob(\Surf),\W_2)$ connecting $\mu_0$ to $\mu_1$.
\end{proposition}
A detailed exposition of this statement can be found, for instance, in \cite[Section 5]{santambrogio2015optimal}, in particular Theorem 5.27.
Again, this result will readily generalize to Riemannian manifolds (Section \ref{sec:LinW2}).

\subsection{Riemannian structure and linearization}
\label{sec:WRiemann}
The space $(\prob(\Surf\subset \R^n),\W_2)$ exhibits a formal Riemannian structure which can be intuited from the celebrated Benamou--Brenier formula, giving the 2-Wasserstein distance as the length of a path in the Wasserstein space:
\begin{align}
\label{eq:BB}
\W_2^2(\mu_0,\mu_1) = \inf_{(\mu_t,v_t)_{t \in [0,1]}} \int_0^1 \int_X \|v_t(x)\|^2\diff \mu_t(x)\,\diff t
\end{align}
where the infimum is taken over solutions (in the sense of distributions) of the continuity equation $\partial_t \mu_t + \ddiv(v_t \cdot \mu_t)=0$ on $t \in [0,1]$ with $\mu_t \geq 0$, and boundary conditions $\mu_{t=0}=\mu_0$ and $\mu_{t=1}=\mu_1$.

Formula \eqref{eq:BB} may formally be interpreted as energy functional on the Riemannian manifold of probability measures $(\prob(\Surf),\W_2)$, where $v_t$ represents an element of the tangent space at $\mu_t$, the infinitesimal change of $\mu_t$ under $v_t$ being encoded by the continuity equation. The Riemannian inner product at the point $\mu_t$ is given by the $\LL^2(\mu_t)$-inner product, such that $\|v_t\|^2_{\LL^2(\mu_t)}=\int_X \|v_t(x)\|^2 \diff \mu_t$ can be interpreted as squared norm of $v_t$ in the tangent space at $\mu_t$.

Using the vocabulary of Riemannian geometry, the logarithmic map of $(\prob(\Surf),W_2)$ at base point $\mu_0\in \probl(X)$ is then the map that associates to a measure $\mu_1\in\prob(X)$, the tangent vector of the constant speed geodesic from $\mu_0$ to $\mu_1$ at $t=0$, which corresponds to the initial velocity field $v_{t=0}$ of a minimizer in \eqref{eq:BB}.
From \eqref{eq:OTGeodesic}, this velocity field actually turns out to be simply $T-\id$ where $T$ is the corresponding optimal transport map from $\mu_0$ to $\mu_1$.
At an intuitive level, this is since in \eqref{eq:OTGeodesic} mass particles of $\mu_0$ are pushed along constant speed straight lines from their initial positions $x$ at $t=0$ to their final positions $T(x)$ at $t=1$. At a more technical level, we find for $\varphi \in \Cont^1(X)$ that
$$\frac{\diff}{\diff t} \left. \int_X \varphi\,\diff \mu_t \right|_{t=0}
=\frac{\diff}{\diff t} \left. \int_X \varphi((1-t) \cdot x + t \cdot T(x))\,\diff \mu_t(x) \right|_{t=0}
=\int_X \langle \nabla \varphi,T-\id \rangle \diff \mu_0.$$
For a more detailed exposition of geodesics in $\W_2$ and their relation to the continuity equation we refer again to \cite[Section 5.4]{santambrogio2015optimal}.
Therefore, for $\W_2$ we define
\begin{align}
	\label{eq:W2_Log}
	\Log^{\W_2}_{\mu_0} & : \prob(\Surf) \to \LL^2(\Surf,\R^n;\mu_0), &
	\mu_1 & \mapsto T-\id, \\
	\intertext{where $T$ is the optimal transport map from $\mu_0$ to $\mu_1$. We can then define the corresponding \emph{exponential map} as}
	\label{eq:W2_Exp}
	\Exp^{\W_2}_{\mu_0} & : \LL^2(\Surf,TX=\R^n;\mu_0) \to \prob(\R^n), &
	v & \mapsto (\id+v)_\# \mu_0.
\end{align}

We make two observations. First, note that we set the co-domain of $\Exp^{\W_2}_{\mu_0}$ to $\prob(\R^n)$, since an arbitrary vector field $v$ might easily push mass outside of $\Surf$, if $x + v(x) \notin X$.
Second, it is easy to see that $\Exp^{\W_2}_{\mu_0}$ is a left-inverse of $\Log^{\W_2}_{\mu_0}$. It is not a general inverse however, because while any measurable vector field $v$ yields some image measure under $\Exp^{\W_2}_{\mu_0}$, the corresponding transport map $T=\id + v$ is in general not optimal, see Theorem \ref{thm:Brenier}. The logarithmic map will therefore then extract a different velocity field.

Finally, note that by Brenier's theorem (Theorem \ref{thm:Brenier}) the logarithmic map can be rewritten as
\begin{equation}
	\label{eq:W2_LogDual}
	\Log^{\W_2}_{\mu_0}(\mu_1)=-\frac12\nabla \Phi_0
\end{equation}
where $\Phi_0$ is a $c$-concave optimal Kantorovich potential in \eqref{eq:W2_dual}. This implies a better understanding of the range of the logarithmic map (see Section \ref{sec:convexity_range}).
The formal Riemannian structure of $\W_2$ has first been highlighted in \cite{otto2001geometry}, more details are given in \cite{ambrosio2005gradient}. We now give a concrete example that also highlights the limitations of the formal interpretation.

\begin{example}[Formal Riemannian structure of $\W_2$ on the real line]
\label{ex:1dWRiemann}
Let $A\assign[0,1] \cup [1+a,2+a]$ for some $a \geq 0$ and let $\mu$ be the uniform probability measure on $A$. $\mu$ consists of two chunks of mass, separated by $a$. In this case, $X$ can be any compact, connected subset of $\R$, such that its interior contains $A$.
Then $v \in \LL^2(\mu)$, $v(x)\assign-1$ for $x <1+a/2$, $v(x)\assign+1$ for $x >1+a/2$ can be interpreted as tangent vector at $\mu$.
For $t>-a/2$, $\nu_t \assign \Exp^{\W_2}_{\mu}(t \cdot v)=(\id+ t\cdot v)_\# \mu$ is the uniform probability measure on the set $[-t,1-t] \cup [1+a+t,2+a+t]$ and $T_t=\id+t\cdot v= \Log^{\W_2}_\mu(\nu_t)$ is the optimal transport map from $\mu$ to $\nu_t$, since $T_t$ is non-decreasing and thus the gradient of a convex map.
For $t>0$, the two chunks are pushed further apart, for $t<0$ they are pushed towards each other. For $t<-a/2$, the two chunks collide. In this case, the map $\Exp^{\W_2}_{\mu}(t \cdot v)$ can still be evaluated, but $\Log^{\W_2}_\mu$ is no longer a left-inverse and $T_t$ is no longer an optimal transport map (it decreases between $x=1$ and $x=1+a$, and thus is no longer the derivative of a convex potential). The point $t=-a/2$ could be interpreted as running into a cut locus. For the case $a=0$, no negative $t$ are admissible. Velocity fields for which $t<0$ is inadmissible (but small positive $t$ are admissible) exist for every $\mu \in \probl(X)$. Consequently, the logarithmic map is not a local homeomorphism between a small open ball around $\mu$ and open environment of the origin in a subspace of $\LL^2(\mu)$. Hence the Riemannian structure is merely intuitive and formal.
\end{example}

As in conventional Riemannian geometry the logarithmic map can be used for local linearization of the Wasserstein distance. For a reference measure $\mu_0 \in \prob(K)$ we can map a set of samples $(\nu_i)_{i=1}^M$ in $\prob(\Surf)$ to the tangent space of $\mu_0$ via $v_i \assign \Log^{\W_2}_{\mu_0}(\nu_i)$, given by $\LL^2(\Surf,TX;\mu_0)$.
Intuitively, one should have
$$\|v_i-v_j\|_{\LL^2(\mu_0)} \approx \W_2(\nu_i,\nu_j)$$
as long as $\nu_i$ and $\nu_j$ are `close' to $\mu_0$. This approximation was proposed for (medical) image analysis in \cite{wang2013linear}. More applications and details are given, for instance, in \cite{kolouri2016continuous,park2018representing,MoosmuellerLinOTSheared2022}.
Since the Riemannian structure of $\W_2$ is merely formal, it is difficult to quantify what `close' in the above sense means. Important contributions to that question are given in \cite{gigli11,delalande21}.

\subsection{Linearized 2-Wasserstein distance for manifolds}
\label{sec:LinW2}
We now illustrate how the definition of logarithmic and exponential map for $\W_2$ can be extended to the case where $\Surf$ is a connected complete Riemannian manifold.
The key ingredient is a generalization of Brenier's theorem (Theorem \ref{thm:Brenier}) to manifolds due to McCann \cite{mccann2001polar}. We will use the slightly extended version of \cite[Theorem 10.41]{villani2009optimal}.
\begin{theorem}[McCann]
	\label{thm:McCann}
	Let $\Surf$ be a $n$-dimensional Riemannian manifold and $K$ a compact subset of $X$. Let $\mu_0 \in \probl(K)$, $\mu_1 \in \prob(K)$. Furthermore, let $\Phi_0$ be a maximizer of \eqref{eq:W2_dual} for these measures. Then:
	\begin{enumerate}[(i)]
		\item The minimizer $\pi$ of \eqref{eq:W2_primal} is unique and can be written as $\pi=(\id,T)_\# \mu_0$ for a corresponding optimal transport map $T$.
		\item $\Phi_0$ is differentiable $\mu_0$-almost everywhere and the map $T$ satisfies
		\begin{equation}
			\label{eq:Brenier-McCann}
			T(x)=\Exp^X_x\left(-\frac 12\nabla \Phi_0(x)\right) \quad \text{$\mu_0(x)$-almost everywhere.}
		\end{equation}
	\end{enumerate}
\end{theorem}

Logarithmic and exponential maps \eqref{eq:W2_Log} and \eqref{eq:W2_Exp} can then be extended in the natural way:
\begin{definition}[Logarithmic and exponential map for $\W_2$ over manifolds]
	\label{def:LogW2Mfold}
	In the setting of the above Theorem we define the \emph{logarithmic and exponential maps} at base point $\mu_0$ as
	\begin{align}
		\Log^{\W_2}_{\mu_0} : \prob(\Surf) & \to \LL^2(\Surf,TX;\mu_0), &
		\mu_1 & \mapsto (x\mapsto\Log^X_x(T(x))) \\
	\intertext{with $T$ the optimal transport map from $\mu_0$ to $\mu_1$, provided by \Cref{thm:McCann}, and}
		\Exp^{\W_2}_{\mu_0} : \LL^2(\Surf,TX;\mu_0) &\to \prob(\Surf), &
		v & \mapsto [x\mapsto\Exp^X_x(v(x))]_\# \mu_0.
	\end{align}
\end{definition}
For $\Surf \subset \R^n$ these formulas reduce again to \eqref{eq:W2_Log} and \eqref{eq:W2_Exp}.
The following proposition shows that these maps do indeed mimic the behaviour of logarithmic and exponential maps in the Riemannian sense.
\begin{proposition}
	\label{prop:LogW2Mfold}
	For $\Log^{\W_2}_{\mu_0}$ and $\Exp^{\W_2}_{\mu_0}$ as given in Definition \ref{def:LogW2Mfold} one has:
	\begin{enumerate}[(i)]
		\item $\W_2^2(\mu_0,\mu_1) = \|v\|^2_{\LL^2(\mu_0)}$ when $v=\Log^{\W_2}_{\mu_0}(\mu_1)$.
		\label{item:LogW2Mfold_dist}
		\item $\Exp^{\W^2}_{\mu_0}$ is a left-inverse of $\Log^{\W_2}_{\mu_0}$. 
		\label{item:LogW2Mfold_LeftInverse}
		\item For the same $v$ as in \eqref{item:LogW2Mfold_dist}, $[0,1] \ni t \mapsto \Exp^{\W_2}_{\mu_0}(t \cdot v)$ yields the constant speed geodesic from $\mu_0$ to $\mu_1$.
		\label{item:LogW2Mfold_Geodesic}
	\end{enumerate}
\end{proposition}
\begin{proof}
	\eqref{item:LogW2Mfold_dist} is an immediate consequence of the fact that the map $T$ in the definition of $\Log^{\W_2}_{\mu_0}$ is an optimal transport map from $\mu_0$ to $\mu_1$ for the quadratic cost, and the properties of $\Log^X_x$ on $\Surf$. Indeed,
	\begin{align*}
		\|v\|^2_{\LL^2(\mu_0)} & = 
		\int_X \| \Log^X_{x_0}(T({x_0}))\|^2_{x_0} \,\diff \mu_0({x_0})
		= \int_X d({x_0},T({x_0}))^2\,\diff \mu_0({x_0}) = \W_2^2(\mu_0,\mu_1).
	\end{align*}
	where we used \eqref{eq:LogMapBasic} in the middle equality.
	Likewise, for \eqref{item:LogW2Mfold_LeftInverse} the claim follows from
	$$\Exp^X_x(v(x))=\Exp^X_x(\Log^X_x(T(x)))=T(x).$$
	
	\eqref{item:LogW2Mfold_Geodesic}: For $t \in [0,1]$ let $p_t : \Surf \times \Surf \to X$ be a (measurable, also w.r.t.~$t$) map that takes $(x_0,x_1)$ to a point on a constant speed geodesic from $x_0$ to $x_1$ at time $t$.
	Let $\pi$ be a minimizer in \eqref{eq:W2_primal}. Then $t \mapsto p_{t\#}\pi$ yields a geodesic between $\mu_0$ and $\mu_1$. See, for instance, \cite[Section 7]{villani2009optimal} for details.
	In our case, $\pi$ is given by $(\id,T)_{\#} \mu_0$ and by our assumptions a map $p_t$ as above can be given by $(x_0,x_1) \mapsto \Exp^X_{x_0}(t \cdot \Log^X_{x_0}(x_1))$. For these choices of $\pi$ and $p_t$, $t \mapsto p_{t\#}\pi$ yields the expression in \eqref{item:LogW2Mfold_Geodesic}
	Uniqueness follows from uniqueness of $\pi$ (Theorem \ref{thm:McCann}) and essential uniqueness of $p_t$ (Section \ref{sec:Notation}).
\end{proof}

Finally, formula \eqref{eq:Brenier-McCann} of Theorem \ref{thm:McCann} yields again a dual expression for the logarithmic map, which will provide more insight into its range (see Section \ref{sec:convexity_range}).
\begin{proposition}
	\label{prop:W2_Log_dual}
	In the setting of Theorem \ref{thm:McCann} one has $\Log^{W_2}_{\mu_0}(\mu_1)=-\frac{1}{2}\nabla\Phi_0$ where $\Phi_0$ is a $c$-concave optimal Kantorovich potential in \eqref{eq:W2_dual}.
\end{proposition}

\section{Hellinger--Kantorovich distance}
\label{sec:HK}
\subsection{Basic properties}
\label{sec:HKBasic}
For a meaningful comparison of non-negative measures with varying mass various unbalanced transport models have been introduced. A prominent example is the Hellinger--Kantorovich distance. Similar to $\W_2$ it can be formulated as a static and a dynamic optimization problem, and it exhibits a formal Riemannian structure. We refer to \cite{liero2018optimal} and \cite{ChizatDynamicStatic2018} for more details and further formulations.
We start by recalling a formulation that is obtained by a relaxation of classical balanced transport, \eqref{eq:OT_primal}.
\begin{definition}[Hellinger--Kantorovich distance]
	\label{def:HK_soft_primal}
	For a parameter $\kappa>0$ we set
	\begin{align}
		\label{eq:HK_soft_primal}
		\HK_\kappa^2(\mu_0,\mu_1) &\assign
		\inf \bigg\{E_\kappa(\pi|\mu_{0},\mu_{1})
		\bigg|
		\pi \in \measp(\Surf \times \Surf)\bigg\}
		\intertext{where}
		E_\kappa(\pi|\mu_{0},\mu_{1})&\assign\int_{\Surf \times \Surf} c_\kappa^{\HK}(x_0,x_1)\diff\pi(x_0,x_1)
		+\kappa^2 \cdot \KL(\proj_{0\#} \pi|\mu_0)
		+\kappa^2 \cdot \KL(\proj_{1\#} \pi|\mu_1)\\
		\intertext{with the cost}
		c^{\HK}_\kappa(x_0,x_1) & \assign \begin{cases}
			-2\kappa^2\log\left( \cos\left(\frac{\dist(x_0,x_1)}{\kappa}\right)\right)
			& \tn{if } \dist(x_0,x_1)<\kappa\frac{\pi}{2}, \\
			+\infty & \tn{otherwise}
		\end{cases}
		\intertext{and the Kullback--Leibler divergence is defined as}
		\label{eq:KL}
		\KL(\rho|\mu) & \assign \begin{cases}
			\int_{\Surf}\log\left(\RadNik{\rho}{\mu}\right)\diff\rho - \abs{\rho} + \abs{\mu}
			& \tn{if }\rho\ll\mu,\,\rho \geq 0, \mu \geq 0, \\
			+\infty &  \text{otherwise}
		\end{cases}
	\end{align}
	for $\rho, \mu \in \meas(\Surf)$.
\end{definition}
By relaxing the marginal conditions $\proj_{i\#}\pi=\mu_i$ this optimization problem remains meaningful even when measures of unequal mass are compared. It turns out that this particular relaxation in combination with the cost function $c^{\HK}_\kappa$ indeed yields a metric on $\measp(\Surf)$.
We collect several fundamental properties in the following proposition.
\begin{proposition}\hfill
	\label{prop:HKBasic}
	\begin{enumerate}[(i)]
		\item $\HK_\kappa$ is a metric on $\measp(\Surf)$.
		\label{item:HKBasicMetric}
		\item Minimizing $\pi$ in \eqref{eq:HK_soft_primal} exist. All minimizers have the same marginals $\proj_{i\#} \pi$, $i=0,1$.
		\label{item:HKBasicMinimizer}
		\item Any minimizing $\pi$ is an optimal transport plan for the balanced OT problem with cost $c_\kappa^{\HK}$, between its marginals $p_{0\#}\pi,~p_{1\#}\pi$ in the sense of \eqref{eq:OT_primal}.
		\label{item:HKBasicBalanced}
	\end{enumerate}
\end{proposition}
\begin{proof}
	\eqref{item:HKBasicMetric} is one of the key results of \cite{liero2018optimal}, see in particular \cite[Theorem 7.20]{liero2018optimal}.
	Existence in \eqref{item:HKBasicMinimizer} is established in \cite[Theorem 6.2]{liero2018optimal}, uniqueness of the marginals in \cite[Theorem 6.6]{liero2018optimal}.
	\eqref{item:HKBasicBalanced} is a trivial observation: if it were not true, the alternative minimizer $\tilde{\pi}$ for \eqref{eq:OT_primal} (with the same marginals as $\pi$) would also yield a better objective in \eqref{eq:HK_soft_primal}.
\end{proof}
\begin{remark}
	\label{rem:scale}
	Note that for a parameter $\kappa>0$, an optimal $\pi$ for \eqref{eq:HK_soft_primal} will not transport as far or further than distance $\kappa \cdot \pi/2$. The choice of $\kappa$ therefore controls the trade-off between transport and mass creation/annihilation in \eqref{eq:HK_soft_primal}. \cite{liero2018optimal} focuses on the case $\kappa=1$, the other cases can be described by re-scaling the metric on $\Surf$. See also \cite[Remark 3.3]{cai2022linearized}.
\end{remark}
\begin{remark}[Mass transport and generation]
\label{rem:TransportAndCreation}
	Let $\pi$ be an optimal plan for the soft-marginal formulation \eqref{eq:HK_soft_primal} of $\HK_\kappa^2(\mu_{0},\mu_{1})$. We will denote its marginals by $\pi_0\assign p_{0\#}\pi$ and $\pi_1\assign p_{1\#}\pi$.
	$\pi_0$ and $\pi_1$ do not depend on the choice of the minimizer $\pi$ (Proposition \ref{prop:HKBasic} \eqref{item:HKBasicMinimizer}). We can then consider the Lebesgue decompositions of $\mu_i$ with respect to the $\pi_i$, i.e.
	\begin{align*}
	\mu_{0}&=\frac{\diff\mu_{0}}{\diff\pi_0}\pi_0+\mu_{0}^\perp, &
	\mu_1 & =\frac{\diff\mu_1}{\diff\pi_1}\pi_1+\mu_{1}^\perp,
	\end{align*}
	where $\mu_i^\perp$ is the part of $\mu_i$ that is singular with respect to $\pi_i$.
	The parts of $\mu_i$ that are dominated by $\pi_i$ should be interpreted as undergoing transport (and gradual mass change along the transport), whereas the singular parts $\mu_i^\perp$ undergo complete destruction ($i=0$) and creation ($i=1$) respectively, without any movement. This is illustrated in more detail in \cite[Section 3.3]{cai2022linearized}.
	If possible, the $\HK_\kappa$ metric always prefers transport over pure destruction/creation, i.e.~$d(x_0,x_1)\geq\kappa\frac\pi2$ for $\mu_{0}^\perp\times \mu_{1}$-a.e.~$(x_0,x_1)$ (and symmetrically, exchanging the roles of $\mu_{0}$ and $\mu_{1}$), see \cite[Lemma 3.13]{cai2022linearized}.
\end{remark}

Duality for $\HK^2_\kappa$, \eqref{eq:HK_soft_primal}, is slightly more involved than for $\W_2$, since the cost $c^{\HK}_\kappa$ can become $+\infty$ at finite range, and due to the non-linearity of the $\KL$-terms. But \eqref{eq:HK_soft_primal} is still a convex optimization problem and with Fenchel--Rockafellar duality one can obtain, for instance, the following dual problem.
\begin{proposition}
	\label{prop:HK_soft_dual}
	For $\mu_0, \mu_1 \in \measp(\Surf)$,
	\begin{multline}
		\label{eq:HK_soft_dual}
		\HK^2_\kappa(\mu_0,\mu_1)=\kappa^2\sup\bigg\{\int_X 1-e^{-\phi_0(x_0)/\kappa^2}\diff\mu_0(x_0)+\int_X 1-e^{-\phi_1(x_1)/\kappa^2}\diff\mu_1(x_1)~\bigg|\\
		(\phi_0,\phi_1)\in\Cont_{c_\kappa^{\HK}}(\Surf)\bigg\}
	\end{multline}
	with the definition of $\Cont_{c_\kappa^{\HK}}(\Surf)$ from balanced transport, \eqref{eq:OT_feasible}.
\end{proposition}
By a (monotonous) change of variables, $\Phi_i=\kappa^2\left(1-e^{-\phi_i/\kappa^2}\right)$, this can be rewritten as a problem with linear objective, but somewhat more complicated constraints.
\begin{proposition}
	\label{prop:HK_cone_dual}
	For any positive measures $\mu_0$, $\mu_1$ on $\Surf$,
	\begin{equation}
		\label{eq:HK_cone_dual}
		\HK^2_\kappa(\mu_0,\mu_1)=\sup\left\{\int_X \Phi_0(x_0)\diff\mu_0(x_0)+\int_X\Phi_1(x_1)\diff\mu_1(x_1)~\middle|~(\Phi_0,\Phi_1)\in\QD_\kappa(\Surf)\right\}
	\end{equation}
	where $\QD_\kappa(\Surf)$ is now the (convex) set
	\begin{multline}
		\label{eq:cone_dual_set}
		\QD_\kappa(\Surf):=\bigg\{\Phi_0,\Phi_1\in\Cont(\Surf)~\bigg|~\left(1-\frac{\Phi_0(x_0)}{\kappa^2}\right)\left(1-\frac{\Phi_1(x_1)}{\kappa^2}\right)\geq\Cos^2\left(\frac{d(x_0,x_1)}{\kappa}\right),\\
		\Phi_i\leq \kappa^2,\, i=0,1 \bigg\}.
	\end{multline}
	where we use $\Cos$ to denote the `truncated' $\cos$ function: $\Cos(s)\assign\cos(\min\{|s|,\pi/2\})$.
\end{proposition}
Both duality formulas are covered by \cite[Theorem 6.3]{liero2018optimal}.
For $\Surf\subset \R^n$ the latter is also given in \cite[Corollary 5.8]{ChizatDynamicStatic2018}.
We will focus on the formulation \eqref{eq:HK_cone_dual} in the following, as it has a more explicit connection with the logarithmic map for $\HK_\kappa$ (Proposition \ref{prop:LogHK_dual}).

\begin{remark}
\label{rem:HKAdmissibility}
If $\spt \mu_i \subset K$ for some compact $K \subset \Surf$ with $\diam K < \kappa \cdot \pi/2$, then $c^{\HK}$ is continuous on the compact set $K \times K$ and continuous maximizers of \eqref{eq:HK_soft_dual} can be shown to exist with the same arguments as for $\W_2$, \eqref{eq:W2_dual}. (This does not hold for \eqref{eq:HK_cone_dual} if one of the two measures is zero.)

This can be generalized to the case where the transport distance of minimal $\pi$ in \eqref{eq:HK_soft_primal} is bounded strictly away from $\kappa \pi/2$. A sufficient condition for this is that the distance between two points, in $\spt \mu_0$ and $\spt \mu_1$ respectively, has to be strictly less than $\kappa \pi/2$. This `admissibility condition' was introduced in \cite{gallouet2021regularity} where further regularity properties are deduced from this assumption. We do not use this assumption in our article.
\end{remark}

In general, existence of dual maximizers can be established by relaxing the maximization space $\Cont(\Surf)$. The following proposition is a subset of \cite[Theorem 6.3]{liero2018optimal} and will be usefull in order to state the existence of a transport map in our fairly general setting:
\begin{proposition}[Existence of weak dual potentials for $\HK$]
	\label{prop:HK_L1_existence}
For $\mu_0, \mu_1 \in \prob(\Surf)$,
\begin{multline}
\label{eq:HK_L1_dual}
\HK_\kappa^2(\mu_0,\mu_1) = \sup \bigg\{
\int_X \Phi_0\,\diff \mu_0 + \int_X \Phi_1 \,\diff \mu_1 \bigg|
\Phi_0,\Phi_1 \in \Borel(\Surf),\,
\Phi_i\leq \kappa^2,\, i=0,1 \text{ and}\\
\forall(x_0,x_1)\in\Surf^2,\,\left(1-\frac{\Phi_0(x_0)}{\kappa^2}\right)\odot\left(1-\frac{\Phi_1(x_1)}{\kappa^2}\right)\geq
\Cos^2\left(\frac{d(x_0,x_1)}{\kappa}\right),
\bigg\}
\end{multline}
\noindent where the extended product $\odot$ is the standard product on $\R\times\R$ and $0\odot+\infty=+\infty\odot0=+\infty$.

Maximizing $(\Phi_0,\Phi_1)$ in \eqref{eq:HK_L1_dual} exist and are all in $\LL^1(\Surf;\mu_{0})\times\LL^1(\Surf;\mu_{1})$. A plan $\pi \in \measp(\Surf \times X)$ and a pair $(\Phi_0,\Phi_1)$ are optimal in \eqref{eq:HK_soft_primal} and \eqref{eq:HK_L1_dual} respectively if and only if
\begin{enumerate}[(i)]
\item for $\pi$-almost every $(x_0,x_1)\in \Surf^2$, $d(x_0,x_1)<\kappa\frac{\pi}{2}$ and  $$\left(1-\frac{\Phi_0(x_0)}{\kappa^2}\right)\left(1-\frac{\Phi_1(x_1)}{\kappa^2}\right)=\cos^2\left(\frac{d(x_0,x_1)}{\kappa}\right)\quad\mathrm{and}\quad\Phi_i(x_i)< \kappa^2,~i=0,1,$$
\item for $i=0,1$ and $(\mu_i^\perp \otimes \mu_{1-i})$-almost every $(x_i,x_{1-i})$,
$$d(x_i,x_{1-i})\geq\kappa\frac{\pi}{2},\quad\mathrm{and}\quad\Phi_i(x_i)=\kappa^2,$$
\label{item:HK_L1_existence_perpfar}%
\item the marginal densities of $\pi$ verify
\label{item:HK_L1_marginals}
\begin{equation}
  \label{eq:marg_densities}
  \RadNik{\pi_i}{\mu_i}(x_i)=1-\frac{\Phi_i(x_i)}{\kappa^2}
  \qquad \text{for $\mu_i$-almost every $x_i$, $i=0,1$.}
\end{equation}
\end{enumerate}
\end{proposition}

\begin{remark}[Log-$c$-concave functions]
\label{rem:logcconcave}
Let $\Phi_0,\Phi_1$ be admissible in \eqref{eq:HK_L1_dual} for $\HK_\kappa^2(\mu_0,\mu_1)$. Introduce now the function
\begin{equation}
	\label{eq:log_c_transform}
	\Phi_0^{\tn{cone}} : X \ni x_1 \mapsto 
	\begin{cases}
		-\infty & \text{if $\exists\,x_0\in B_\Surf(x_1,\kappa\pi/2):\,\Phi_0(x_0)=\kappa^2$},\\
		\kappa^2\inf_{
				x_0 \in B_\Surf(x_1,\kappa\pi/2)}\left(1-\frac{\Cos^2(d(x_0,x_1)/\kappa)}{1-\frac{\Phi_0(x_0)}{\kappa^2}}\right) & \text{otherwise.}
	\end{cases}
\end{equation}
	Then $(\Phi_0,\Phi_0^{\tn{cone}})$ is a competing pair of functions for \eqref{eq:HK_L1_dual}. 
%
	Indeed, by definition, $\Phi_0^{\tn{cone}}$ is Borel measurable and for  any ~$(x_0,x_1)\in\Surf^2$,
	$$\left(1-\frac{\Phi_0(x_0)}{\kappa^2}\right)\odot\left(1-\frac{\Phi_0^{\tn{cone}}(x_1)}{\kappa^2}\right)\geq
	\Cos^2\left(\frac{d(x_0,x_1)}{\kappa}\right).$$
	Also, $\Phi_0^{\tn{cone}}\leq\kappa^2$ since $\Phi_0\leq\kappa^2$. By admissibility one must have that $\Phi_1$ is less than the above infimum, i.e.~$\Phi_1 \leq \Phi_0^{\tn{cone}}$ and therefore, the pair $(\Phi_0,\Phi_0^{\tn{cone}})$ yields a better objective than $(\Phi_0,\Phi_1)$. This allows us to assume that an optimal pair is of this form $(\Phi_0,\Phi_0^{\tn{cone}})$ (and in turn get $\Phi_0^{\tn{cone}}\in\LL^1(\Surf;\mu_{1})$) and we can rewrite
	\begin{equation}
		\label{eq:HK_logct_dual}
		\HK_\kappa^2(\mu_0,\mu_1) = \max \bigg\{
		\int_X \Phi_0\,\diff \mu_0 + \int_X \Phi_0^{\tn{cone}} \,\diff \mu_1 \bigg|
		\Phi_0 \in \Borel(\Surf),~
		\Phi_0\leq \kappa^2
		\bigg\}.
	\end{equation}
	
	Intuitively, $\Phi_0^{\tn{cone}}$ plays the role of the $c$-transform of $\Phi_0$ for balanced optimal transport, see \cite[Definition 1.10, Section 1.6.1]{santambrogio2015optimal} and in fact, the operation $\Phi_0\mapsto\Phi_0^{\tn{cone}}$ mostly corresponds to $\phi_0\mapsto\phi_0^{c^{\HK}_\kappa}$ up to the change of variable yielding \eqref{eq:HK_cone_dual} (exactly for $\Phi_0<\kappa^2$, otherwise, some additional conventions must be taken). Following this analogy we will call a function `log-$c$-concave' if it can be written as $\Phi_1=\Phi_0^{\tn{cone}}$ for some $\Phi_0\leq\kappa^2$. We revisit these canonical dual solutions in Sections \ref{sec:convexity_range} and \ref{sec:convexity_range_num}.
	
	Finally, similar to the balanced scenario, one can further restrict candidates $\Phi_0$ in \eqref{eq:HK_logct_dual} to the set log-$c$-concave functions by repeating the above argument. $\Phi_0\mapsto\Phi_0^{cone}$ turns out to also be an involution on this set.
\end{remark}

\subsection{Transport maps for HK}
\label{sec:HKMaps}
Similar as for $\W_2$, to state exponential and logarithmic maps for $\HK_\kappa$, it will be helpful to establish the existence of optimal transport maps first.
As in \eqref{eq:OTMap}, by optimal transport map it is meant here that the optimal unbalanced coupling $\pi$ in \eqref{eq:HK_soft_primal} is concentrated on the graph of a map $T: X \to X$. Note however, that for the HK distance, the relation between $T$ and shortest paths is more involved than \eqref{eq:OTGeodesic} for $\W_2$.
For HK, geodesics can be reconstructed from an optimal plan via the logarithmic and exponential map given in Section \ref{sec:LinHK}, see also \cite[Section 3.3]{cai2022linearized}.
While existence of optimal transport maps for a particular class of convex cost functions has been established by \cite{McCannGangboOTGeometry1996} this does not apply to $\HK_\kappa$ directly, as $c^{\HK}_\kappa$ can take value $+\infty$. For $\Surf \subset \R^n$ existence of optimal transport maps for $\HK$ is shown in \cite[Theorem 6.6]{liero2018optimal} and this was generalized to Riemannian manifolds in \cite[Proposition 16]{gallouet2021regularity} under the admissibility assumption mentioned in Remark \ref{rem:HKAdmissibility}.
Dropping this assumption requires to relax the notion of a gradient for the corresponding dual potentials, in order to obtain an expression coherent with the result by McCann. This is accomplished by the notion of \emph{approximate gradient}, the definition of which we recall in \Cref{ap:approx_grad}, in our context of Riemannian manifolds. The following theorem is then a generalization of \cite[Proposition 16]{gallouet2021regularity} with some adaptations for our purposes: First, we will write the transport map corresponding to an optimal plan in the sense of \Cref{def:HK_soft_primal} using the dual variables of \eqref{eq:HK_cone_dual} instead of these of \eqref{eq:HK_soft_dual} as they will be more natural in the context of the logarithmic map. Furthermore, and most importantly, we drop the admissibility assumption of \cite{{gallouet2021regularity}} (see Remark \ref{rem:HKAdmissibility}), at the price that the transport map is only defined $\pi_0$-almost everywhere. However, since we are not concerned with regularity here, this is not an actual restriction, as we know that no transport can happen for the mass of $\mu_0$ outside of $\spt(\pi_0)$ (only destruction or creation, see Remark \ref{rem:TransportAndCreation}).

\begin{theorem}
  \label{thm:McCann_HK}
  Let $\mu_0 \in \measpl(\Surf)$, $\mu_1 \in \measp(X)$ and $(\Phi_0,\Phi_1)$ be a maximizing pair for \eqref{eq:HK_L1_dual} with these measures. Then:
	\begin{enumerate}[(i)]
	\item The minimizer $\pi$ of \eqref{eq:HK_soft_primal} is unique and can be written as $\pi=(\id,T)_\# \pi_0$ for a corresponding optimal transport map $T$ and $\pi_0$ is the first marginal of $\pi$.
	\item The function $\Phi_0$ is $\mu_0$-almost everywhere approximately differentiable and the map $T$ satisfies
		\begin{equation}
			\label{eq:Brenier-McCann-HK}
			T(x)=\Exp^X_x\left(-\kappa\arctan\left(\frac{1}{2\kappa}\frac{\norm{\tilde{\nabla}\Phi_0(x)}}{1-\Phi_0(x)/\kappa^2}\right)\frac{\tilde{\nabla}\Phi_0(x)}{\norm{\tilde{\nabla}\Phi_0(x)}}\right)
		 \quad \text{for $\pi_0-$almost every $x$,}
		\end{equation} 
	while $\tilde{\nabla}\Phi_0(x_0)=0$ for $\mu_0^\perp$-a.e. $x_0\in \Surf$, where $\mu_0^\perp$ is the part of $\mu_0$ that is singular with respect to $\pi_0$.
	\end{enumerate}	
\end{theorem}

The proof of this result is given in \Cref{ap:approx_grad}.

\begin{remark}
	Under the admissibility assumption of \cite{gallouet2021regularity}, the function $\Phi_0$ is actually Lipschitz-continuous and one can then replace the approximate gradients by regular gradients (still almost everywhere) and the map $T$ is then defined $\mu_0$-almost everywhere, see also Remark \ref{rem:HKAdmissibility}.
\end{remark}

\subsection{Riemannian structure and linearization}
\label{sec:LinHK}
Similar to $\W_2$ the formal Riemannian structure of $(\meas_+(\Surf),\HK_\kappa)$ can be intuited from a dynamical formulation of the distance in the spirit of the Benamou--Brenier formulation,
\begin{align}
\label{eq:HK_BB}
\HK^2_\kappa(\mu_0,\mu_1) =\inf_{(\mu_t,v_t,\alpha_t)_{t \in [0,1]}} \int_0^1 \int_X \left[ \|v_t(x)\|^2 + \frac{\kappa^2}{4} \alpha_t(x)^2 \right] \diff \mu_t(x)\,\diff t
\end{align}
where this time the infimum is taken over (distributional) solutions of the continuity equation \emph{with source term},
\begin{align}
\label{eq:HK_CE}
\partial_t \mu_t + \ddiv(v_t \cdot \mu_t)= \alpha_t \cdot \mu_t
\end{align}
on $t \in [0,1]$ with $\mu_t \geq 0$, again with boundary conditions $\mu_{t=0}=\mu_0$ and $\mu_{t=1}=\mu_1$. We refer to \cite{KMV-OTFisherRao-2015,ChizatOTFR2015,liero2018optimal} for more details on this formulation.

Similar to the 2-Wasserstein distance, \eqref{eq:HK_BB} can be thought to induce a formal Riemannian structure on $(\measp(\Surf),\HK_\kappa)$, where tangent vectors now have an additional mass component $\alpha$ describing the rate of mass destroyed/created in order to go from $\mu_0$ to $\mu_1$. In the Euclidean case $\Surf\subset\R^n$, a logarithmic and exponential map for this $\HK_\kappa$ metric space have been introduced in \cite{cai2022linearized} and therefore, we will only give the extensions and basic properties in the case were $X$ is a general Riemannian manifold to minimize redundancy.

\begin{definition}[Logarithmic map]
	\label{def:HKLog_map}
	Let $\mu_0 \in \measpl(\Surf)$ and $\mu_1 \in \measp(X)$, and the optimal transport plan minimizing \eqref{eq:HK_soft_primal}, $\pi=(\id,T)_\#\pi_0$ (uniqueness and existence of $T$ provided by Theorem \ref{thm:McCann_HK}), let $\pi_1 \assign T_\# \pi_0=\proj_{1\#} \pi$. Consider the following Lebesgue decompositions of $\mu_0$ and $\mu_1$,
	\begin{align*}
		\mu_0 & =\RadNik{\mu_{0}}{\pi_0}\pi_0+\mu_0^\perp, & 
		\mu_1 & =\RadNik{\mu_1}{\pi_1} \pi_1+\mu_1^\perp.
	\end{align*}
	
	Then we define the logarithmic map of $\mu_1$ at $\mu_0$ as $\Log^{\HK_\kappa}_{\mu_0}(\mu_1)=(v_0,\alpha_0,\mu_1^{\perp})$, where
	\begin{subequations}
	\label{eq:def_HK_log}
	\begin{align}
	v_0({x_0}) & =\begin{cases}
			\kappa \RadNik{\pi_0}{\mu_{0}}({x_0})\tan\left(\frac{\norm{\Log_{x_0}(T({x_0}))}}{\kappa}\right)\frac{\Log_{x_0}(T({x_0}))}{\norm{\Log_{x_0}(T({x_0}))}} & \text{for } \pi_0\tn{-a.e. }{x_0}\\
			0 & \text{for } \mu_0^\perp\tn{-a.e. }{x_0}
		\end{cases} \\
		\alpha_0({x_0}) & =\begin{cases}
			2\left(\RadNik{\pi_0}{\mu_{0}}({x_0})-1\right)& \text{for } \pi_0\text{-a.e. }{x_0}\\
			-2& \text{for } \mu_0^\perp\text{-a.e. }{x_0}
		\end{cases}
	\end{align}
	\end{subequations}
	These tangent components live in the set of \emph{feasible tangent vectors for $\mu_0$}:
	\begin{multline}
		\label{eq:def_TFeas}
		\TFeas_\kappa(\mu_0) \assign \Big\{(v,\alpha,\mu^\perp)\in\LL^2(\Surf,TX;\mu_0)\times\LL^2(\Surf;\mu_0)\times\measp(\Surf)~\Big|\\
		\alpha\geq -2~\mu_0\tn{-a.e.},~d(x_0,x_1)\geq\kappa\frac{\pi}{2} \text{ for }\mu_0\otimes\mu^\perp\tn{-a.e.~}(x_0,x_1)\Big\}
		\end{multline}
\end{definition}
\begin{remark}
\label{rem:CaiComparison}
Relative to \cite[Proposition 4.1, Definition 4.5]{cai2022linearized} the above expressions have been adjusted to the Riemannian context, simplified by using optimality conditions from Proposition \ref{prop:HK_L1_existence} (which will be more useful in the present article), and the formal square root for the $\mu_1^\perp$ component has been omitted for simplicity as this component will only play a minor role in this article (see \cite[Remark 4.6]{cai2022linearized}).
\end{remark}
The corresponding exponential map can once again be defined as a left-inverse of this logarithmic map. The following generalizes \cite[Proposition 4.8]{cai2022linearized}.
\begin{proposition}
\label{prop:HKExp}
	For $\kappa>0$, $\mu_0 \in \measp(\Surf)$, define the \emph{exponential map} at $\mu_0$ for $\HK_\kappa$ as
	\begin{equation}
		\label{eq:def_HK_exp}
		\Exp^{\HK_\kappa}_{\mu_0}: \TFeas_\kappa(\mu_0) \ni (v,\alpha,\mu^\perp)\mapsto T_\#u^2\mu_0+\mu^\perp
	\end{equation}
	where
	\begin{align}
		\label{eq:transmap_HK}
		T & :x\mapsto\exp_{x}\left(\kappa\arctan\left(\frac{1}{\kappa}\frac{\norm{v(x)}}{1+\alpha(x)/2}\right)\frac{v(x)}{\norm{v(x)}}\right) \\
		\label{eq:massratio_HK}
		u & :x\mapsto\sqrt{\left(1+\frac{\alpha(x)}{2}\right)^2+\frac{1}{\kappa^2}\norm{v(x)}^2}
	\end{align}
	Then $\Exp^{\HK_\kappa}_{\mu_0}$ is a left inverse of $\Log^{\HK_\kappa}_{\mu_0}$,
	$$\Exp^{\HK_\kappa}_{\mu_0}\circ\Log^{\HK_\kappa}_{\mu_0}=\id.$$
\end{proposition}

Next, we give the expression for $\HK_\kappa^2(\mu_0,\mu_1)$ in terms of $\Log^{\HK_\kappa}_{\mu_0}(\mu_1)$, generalizing \cite[Proposition 4.1, (4.4)]{cai2022linearized}. As expected in Riemannian geometry, this is a quadratic form in $v$ and $\alpha$, but since we discarded the formal square root for the $\mu_1^\perp$ component (see Remark \ref{rem:CaiComparison}) it is 1-homogeneous in this component.
\begin{proposition}
	\label{prop:HK_as_L2}
	Let $\mu_0 \in \measpl(\Surf)$, $\mu_1 \in \measp(X)$ and let $(v_0,\alpha_0, \mu_1^\perp)=\Log^{\HK_\kappa}_{\mu_0}(\mu_1)$. Then,
	\begin{equation}
		\label{eq:HK_as_L2}
		\HK^2_\kappa(\mu_0,\mu_1)=\norm{v_0}^2_{\LL^2(\mu_0)}+\frac{\kappa^2}{4}\norm{\alpha_0}^2_{\LL^2(\mu_0)}+\kappa^2\abs{\mu_1^\perp}.
	\end{equation}
\end{proposition}
\begin{proof}
	For the optimal transport plan $\pi$ for the soft-marginal formulation, one has the following expression for $\HK^2$ as a cumulated mass discrepancy (see for instance \cite[Theorem 6.3, eq. (6.20)]{liero2018optimal}):
	\begin{equation*}
		\HK_\kappa^2(\mu_0,\mu_1)= \kappa^2 \left(\abs{\mu_0} + \abs{\mu_1} - 2\abs{\pi}\right)
	\end{equation*}
	Now we evaluate the right-hand side of \eqref{eq:HK_as_L2} (using the optimality conditions from \Cref{prop:HK_L1_existence})
	\begin{align*}
	\int_X\norm{v_0({x_0})}^2\diff\mu_0({x_0}) & =\kappa^2\int_{\Surf }\RadNik{\pi_0}{\mu_{0}}({x_0})
		\frac{\sin^2\left(\frac{d({x_0},T({x_0}))}{\kappa^2}\right)}{\cos^2\left(\frac{\dist({x_0},T({x_0}))}{\kappa^2}\right)}
		\diff\pi_0({x_0})\\
	&=\kappa^2\int_{\Surf\times \Surf}\RadNik{\mu_1}{\pi_{1}}(x_1)\sin^2\left(\frac{\dist(x_0,x_1)}{\kappa}\right)\diff\pi(x_0,x_1) \\
	\tfrac14 \int_X\alpha_0({x_0})^2\diff\mu_0({x_0}) &
	= \int_X \left(\RadNik{\pi_0}{\mu_0}-1\right)^2 \diff \mu_0
	= \int_X \left(1-2 \RadNik{\pi_0}{\mu_0} + \left(\RadNik{\pi_0}{\mu_0}\right)^2 \right) \diff \mu_0 \\
	& = \abs{\mu_0} - 2 \abs{\pi} + \int_X \RadNik{\mu_1}{\pi_{1}}(x_1) \cos^2\left(\frac{\dist(x_0,x_1)}{\kappa}\right)\diff \pi(x_0,x_1)
	\end{align*}
	and therefore,
\begin{equation*}
	\begin{split}
		\norm{v_0}^2_{\LL^2(\mu_0)}+\frac{\kappa^2}{4}\norm{\alpha_0}^2_{\LL^2(\mu_0)}+\kappa^2\abs{\mu_1^\perp}
		=& \kappa^2 \left(\abs{\mu_0}-2\abs{\pi}+\int_X \RadNik{\mu_1}{\pi_{1}}(x_1) \diff \pi(x_0,x_1) + \abs{\mu_1^\perp} \right)
		\\=&\HK_\kappa^2(\mu_{0},\mu_{1})
	\end{split}
\end{equation*}
	and we obtain the claim.
\end{proof}

Finally, the expression for the optimal transport map $T$, in the style of McCann using the dual potential, \eqref{eq:Brenier-McCann-HK} allows to give simpler expressions for the logarithmic map in terms of the dual potential. This form of the logarithmic map is closely related to the polar factorization result for $\HK$ given in \cite[Theorem 18]{gallouet2021regularity}.
\begin{proposition}
	\label{prop:LogHK_dual}
	For $\mu_0 \in \measpl(\Surf)$ and $\mu_1 \in \measp(X)$, let $(\Phi_0,\Phi_1)$ be optimal dual solutions of  \eqref{eq:HK_L1_dual}. Then the components $(v_0, \alpha_0,\mu_1^\perp)=\Log^{\HK_\kappa}_{\mu_0}(\mu_1)$ of the logarithmic map (\cref{def:HKLog_map}) can be written as
		\begin{align}
			\label{eq:LogHK_dual}
			v_0 & =-\frac{\tilde{\nabla}\Phi_0}{2}, &
			\alpha_0 & =-\frac{2}{\kappa^2} \Phi_0, &
			\mu_1^\perp & =\mu_1\mres\{\Phi_1=\kappa^2\}.
		\end{align}
\end{proposition}

\begin{proof}
	We plug the expression for the optimal transport map $T$ from \cref{thm:McCann_HK} into the definition of $v_0$ and $\alpha_0$ of Definition \ref{def:HKLog_map} and find for $\pi_0$-almost every $x$,
	$$\tan\left(\frac{\norm{\Log_x(T(x))}}{\kappa}\right)\frac{\Log_x(T(x))}{\norm{\Log_x(T(x))}}=-\frac{1}{2\kappa}\frac{\tilde{\nabla}\Phi_0(x)}{1-\Phi_0(x)/\kappa^2},$$
	and therefore for $\pi_0$-a.e.~$x$,
	\begin{align*}
			v_0(x) & =\kappa\frac{\Log_x(T(x))}{\norm{\Log_x(T(x))}}\RadNik{\pi_0}{\mu_{0}}(x)\tan\left(\frac{\norm{\Log_x(T(x))}}{\kappa}\right)\\
			& =\kappa\frac{\Log_x(T(x))}{\norm{\Log_x(T(x))}}(1-\Phi_0(x)/\kappa^2)\tan\left(\frac{\norm{\Log_x(T(x))}}{\kappa}\right)\\
			&=\kappa(1-\Phi_0(x)/\kappa^2)\left(-\frac{1}{2\kappa}\frac{\tilde{\nabla}\Phi_0(x)}{1-\Phi_0(x)/\kappa^2}\right)=-\frac{1}{2}\tilde{\nabla}\Phi_0(x)\\
	\intertext{and}
		\alpha_0(x)&=2(1-\frac{\Phi_0(x)}{\kappa^2}-1)	=-\frac{2}{\kappa^2}\Phi_0(x).
	\end{align*}

On the other hand, for $\mu_0^\perp$-a.e.~$x\in \Surf$, $\Phi_0(x)=\kappa^2$ so that $\tilde{\nabla}\Phi_0(x)=0$ (up to a $\mu_0^\perp$-negligible set of points) and so the expressions for $v_0(x)$ and $\alpha_0(x)$ are still valid.

Finally, by the optimality conditions of \cref{prop:HK_L1_existence} \eqref{item:HK_L1_existence_perpfar}, $\Phi_1<\kappa^2$ $\pi_1$-a.e., whereas $\Phi_1=\kappa^2$ $\mu_1^\perp$-a.e.~which yields the expression for $\mu_1^\perp$.
\end{proof}

\begin{remark}[Relation to Hamilton--Jacobi equation]
Formally the dual problem to \eqref{eq:HK_BB} is given by
\begin{align*}
\sup_{\phi : [0,1] \times \Surf \to \R} \int_X \phi(1,\cdot)\diff \mu_1 - \int_X \phi(0,\cdot)\diff \mu_0
\end{align*}
over sufficiently regular functions $\phi$ subject to the constraint
$$\partial_t \phi(t,x) + \frac{\|\nabla \phi(t,x)\|^2}{4} + \frac{\phi(t,x)^2}{\kappa^2} \leq 0$$
and with formal primal-dual optimality conditions
\begin{equation}
\label{eq:FormalPDDynamicHK}
v_t=\nabla \phi(t,\cdot)/2 \qquad \tn{and} \qquad \alpha_t=2\phi(t,\cdot)/\kappa^2.
\end{equation}
This dynamic dual problem can subsequently be identified with the static dual \eqref{eq:HK_L1_dual} with the correspondence $(\Phi_0,\Phi_1)=(-\phi(0,\cdot),\phi(1,\cdot))$.
Formulas \eqref{eq:LogHK_dual} and \eqref{eq:FormalPDDynamicHK} can then seen to be consistent. We refer to \cite[Section 8.4]{liero2018optimal} for more details on the dynamic dual perspective.
\end{remark}

\subsection{Linearized Spherical Hellinger--Kantorovich distance}
\label{sec:LinSHK}
A motivation for the introduction of the Hellinger--Kantorovich distance is its ability to compare measures of different total mass. For one, this allows the description of growth and destruction processes, but it also makes the metric more robust with respect to small mass fluctuations and measurement errors when comparing measures of (almost) identical mass.
It turns out that the $\HK$ distance between measures of different mass can be reduced to the comparison of probability measures by a simple scaling formula.
Let $\mu_0,\mu_1 \in \prob(X)$, $m_0,m_1 \geq 0$, then \cite[Theorem 3.3]{laschos2019geometric}
$$\HK_1(m_0 \cdot \mu_0,m_1 \cdot \mu_1)^2 = \sqrt{m_0 \cdot m_1} \HK_1(\mu_0,\mu_1)^2 + (\sqrt{m_0}-\sqrt{m_1})^2.$$
So the expected key advantage of $\HK$ over $W_2$ in data analysis applications may not be so much the ability to deal with global mass differences, but rather the robustness with respect to local fluctuations, i.e.~being able to slightly increase mass in one area of the sample while decreasing it in another to allow for a better matching.

Of course, by restriction $\HK$ still induces a metric on the set of probability measures, but it will no longer be geodesic, since shortest paths in the full space $(\measp(X),\HK)$ between $\mu_0, \mu_1 \in \prob(X)$ will have mass less than 1 for times $t \in (0;1)$. In fact, one will have \cite[Equation (1.1)]{laschos2019geometric}
\begin{equation}
\label{eq:HKMass}
M(t) \assign \abs{\mu_t}=1-t(1-t) \HK_\kappa^2(\mu_0,\mu_1)/\kappa^2.
\end{equation}
Thus, when working with $\HK_\kappa$ on probability measures, we will always see this bias in the proposed interpolations and it will also affect the linearized analysis.

To obtain a geodesic distance within $\prob(X)$, one can restrict admissible paths in \eqref{eq:HK_BB} to those with constant unit mass. That is, one still admits a growth field $\alpha_t$ that can locally create or destroy mass, but on average the mass must be preserved, i.e.~one adds the constraint to the optimization problem that $\int_X \alpha_t \diff \mu_t=0$ for all $t$.
The resulting metric is called the spherical Hellinger--Kantorovich distance ($\SHK_1$), introduced in \cite{laschos2019geometric}. A key observation of \cite{laschos2019geometric} is that $\SHK_1$ and $\HK_1$ are even more intimately related than the preceding discussion might suggest. Indeed, $(\measp(X),\HK_1)$ turns out to be a cone space over $(\prob(X),\SHK_1)$. This means that $\SHK(\mu_0,\mu_1)$ and $\HK(\mu_0,\mu_1)$ can be directly computed from each other, as can be their geodesics, by a simple geometric intuition, sketched in Figure \ref{fig:SHKIntuition}.
\begin{figure}[hbt]
	\centering
	\includegraphics{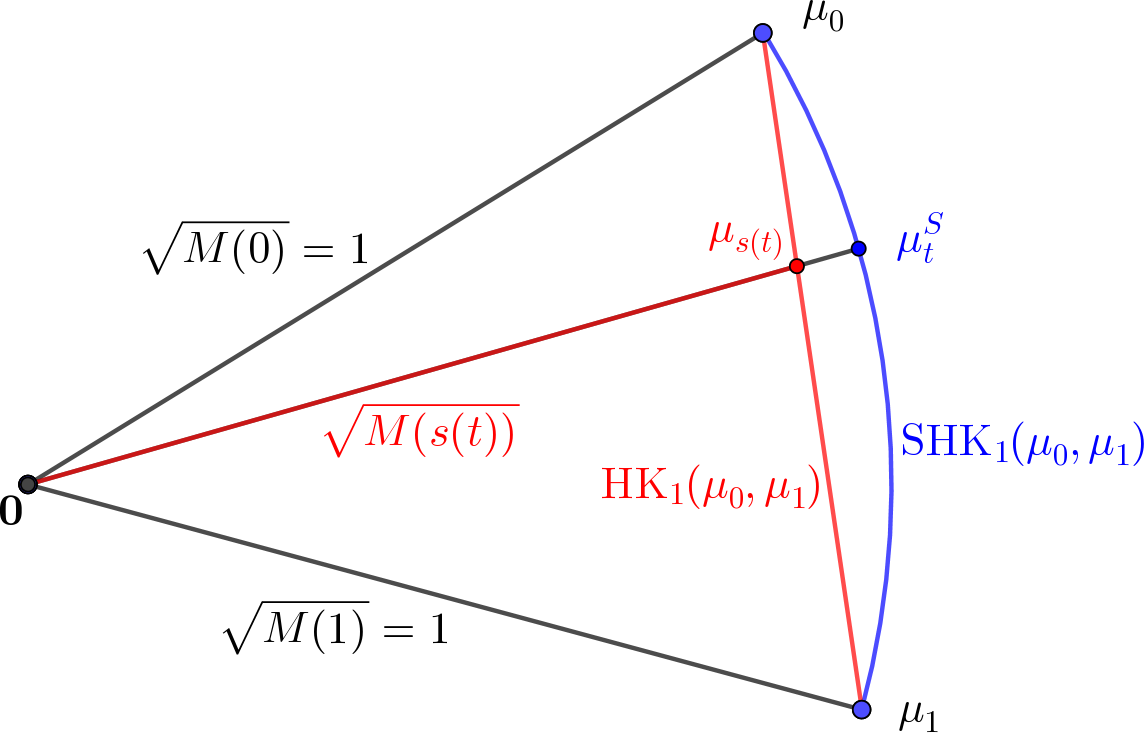}
	\caption{Geometric intuition for the cone structure and relation between $\HK_1$ and $\SHK_1$. See text for definition of symbols.}
	\label{fig:SHKIntuition}
\end{figure}

\begin{definition}[Spherical Hellinger-Kantorovich distance]
For $\mu_0, \mu_1 \in \prob(X)$, $\kappa>0$, the spherical Hellinger--Kantorovich distance between $\mu_0$ and $\mu_1$ (with scale $\kappa$) is defined as
\begin{equation}
\label{eq:SHK}
\SHK_\kappa(\mu_0,\mu_1) \assign \kappa \cdot \arccos\left(1-\frac{\HK_\kappa(\mu,\nu)^2}{2\kappa^2}\right).
\end{equation}
\end{definition}
For $\kappa=1$ this formula is given by \cite[Equation (1.1)]{laschos2019geometric} for $\alpha=1$, $\beta=4$.
The proper generalization to $\kappa \neq 1$ can be deduced from a simple scaling argument, see \cite[Remark 3.3]{cai2022linearized}.

Moreover, if $[0,1] \ni t \mapsto \mu_t \in \measp(X)$ is a constant speed geodesic between $\mu_0$ and $\mu_1$ for $\HK_\kappa$, then a constant speed geodesic $(\mu_t^S)_{t \in [0,1]}$ in $(\prob(X),\SHK_\kappa)$ can be obtained by re-normalizing to unit mass and a reparametrization to account for the discrepancy between the arc and the chord in Figure \ref{fig:SHKIntuition}. One finds \cite[Theorem 2.7]{laschos2019geometric},
\begin{align*}
 \mu^S_t & \assign \mu_{s(t)}/M(s(t))
\end{align*}
with $M(s) \assign\abs{\mu_s}=1-s(1-s) \HK_\kappa^2(\mu_0,\mu_1)/\kappa^2$ as in \eqref{eq:HKMass} and
\begin{align}
\label{eq:SHKReparam}
s(t) & \assign \frac{\sin(t \SHK_\kappa(\mu_0,\mu_1)/\kappa)}{\sin((1-t)\SHK_\kappa(\mu_0,\mu_1)/\kappa)+\sin(t\SHK_\kappa(\mu_0,\mu_1)/\kappa)} 
\end{align}
The formula for $s$ can be deduced from Figure \ref{fig:SHKIntuition}, it is also a special case of \cite[Theorem 2.7]{laschos2019geometric}.

The reparametrization and rescaling of $\mu_t$ to $\mu^S_t$ imply corresponding transformations on the velocity and mass change fields. Let $(v_t,\alpha_t)_{t \in [0,1]}$ be the fields for the $\HK_\kappa$-geodesic. Then the ones for the $\SHK_\kappa$-geodesic are given by
\begin{align}
\label{eq:SHKFieldTransform}
	v^S_t & \assign v_{s(t)} \cdot s'(t), &
	\alpha^S_t & \assign (\alpha_{s(t)}-\ol{\alpha}_{s(t)}) \cdot s'(t), &
	\ol{\alpha}_{s} & \assign \tfrac{1}{M(s)} \int_X \alpha_s \diff \mu_s = \tfrac{M'(s)}{M(s)}
\end{align}
where for $t\in[0;1]$, $\ol{\alpha}_{s(t)}$ is a constant shift in the growth field to maintain the total mass of $\mu_t^S$ constant.
One can show that $(\mu^S_t,v^S_t,\alpha^S_t)_{t \in [0,1]}$ do indeed solve the continuity equation \eqref{eq:HK_CE}, and that plugging them into \eqref{eq:HK_BB} yields the expected value $\SHK_\kappa^2(\mu_0,\mu_1)$, as the integral of a constant squared speed.

One can then introduce the corresponding notion of logarithmic map for $\SHK_\kappa$, similar to $\HK_\kappa$.
\begin{definition}[Logarithmic map for $\SHK_\kappa$]
	\label{def:LogSHK_map}
	Let $\mu_0\in\probl(X)$ and $\mu_1\in\prob(X)$ and the components of the logarithmic map $(v_0,\alpha_0,\mu_1^\perp)=\Log_{\mu_0}^{\HK_\kappa}(\mu_1)$. We define the Logarithmic map for $\SHK_\kappa$ at base point $\mu_0$ as
	$$\Log_{\mu_0}^{\SHK_\kappa}(\mu_1) \assign (v^S_0,\alpha^S_0,(\mu^S_1)^\perp)$$
	with
	\begin{align}
		\label{eq:SHKLog}
		v^S_0 & \assign v_0 \cdot s'(0), &
		\alpha^S_0 & \assign (\alpha_0+\HK_\kappa^2(\mu_0,\mu_1)/\kappa^2) \cdot s'(0), &
		(\mu^S_1)^\perp & = \mu_1^\perp \cdot s'(0)^2
	\end{align}
	and
	\begin{align}	
	\label{eq:SHKLogSPrime}
	s'(0)=\frac{\SHK_\kappa(\mu_0,\mu_1)/\kappa}{\sin(\SHK_\kappa(\mu_0,\mu_1)/\kappa)}.
	\end{align}
\end{definition}
Expressions for $v_0^S$ and $\alpha_0^S$ follow directly from \eqref{eq:SHKFieldTransform}.
Evaluating $s'$, \eqref{eq:SHKReparam} at $t=0$ is a simple algebraic exercise. The value for $\ol{\alpha}_0$ is computed in the following \Cref{lem:mean_alpha}.
The appropriate rescaling of $\mu^S_1$ then becomes apparent, for example by requiring that the norm of 
$\Log_{\mu_0}^{\SHK_\kappa}(\mu_1)$ equals $\SHK_\kappa(\mu_0,\mu_1)$, as shown in \Cref{prop:SHK_as_L2} below.

\begin{lemma}
	\label{lem:mean_alpha}
	Let $\mu_0\in\probl(\Surf)$, $\mu_1\in\prob(\Surf)$ and $v_0,\alpha_0,\mu_1^\perp=\Log^{\HK_\kappa}_{\mu_0}(\mu_1)$.
	
	Then, 
	$$\bar{\alpha}_0=\int_{\Surf}\alpha_0\diff\mu_0=-\frac{1}{\kappa^2}\HK_\kappa^2(\mu_0,\mu_1)=2\left[\cos\left(\frac{\SHK_\kappa(\mu_0,\mu_1)}{\kappa}\right)-1\right]$$
\end{lemma}
\begin{proof}
	This is straightforward from the expressions in \cref{prop:LogHK_dual} and the fact that $\mu_0$ and $\mu_1$ have same mass. Indeed, recall that the marginals of the optimal transport plan $\pi$ defining the $\HK_\kappa$ distance are given by the dual potentials $\Phi_0$, $\Phi_1$ and also have same mass (that of $\pi$):
	$$\int_{\Surf}\diff \pi_0=\int_\Surf\left(1-\frac{\Phi_0}{\kappa^2}\right)\diff\mu_0=\int_\Surf\left(1-\frac{\Phi_1}{\kappa^2}\right)\diff\mu_1=\int_{\Surf}\diff \pi_1$$
	Simplifying the middle equality (and recalling $\abs{\mu_{0}}=\abs{\mu_1}=1$), we get
	$$\HK_\kappa^2(\mu_0,\mu_1)=\int_{\Surf}\Phi_0\diff\mu_0+\int_{\Surf}\Phi_1\diff\mu_1=2\int_{\Surf}\Phi_0\diff\mu_0$$
	and one can conclude by replacing $\Phi_0$ by $-\kappa^2\alpha_0/2$.
\end{proof}

\begin{proposition}
\label{prop:SHK_as_L2}
In the setting of \Cref{def:LogSHK_map} one has
	\begin{equation}
		\label{eq:BB_SHK}
		\SHK_\kappa^2(\mu_0,\mu_1) = \norm{v_0^S}^2_{\LL^2(\mu_0)}
		+\frac{\kappa^2}{4}\norm{\alpha_0^S}^2_{\LL^2(\mu_0)}+\kappa^2\abs{(\mu^S_1)^\perp}.
	\end{equation}
\end{proposition}
\begin{proof}
First observe that using $\ol{\alpha}_0=\int_X \alpha_0\diff \mu_0$, one finds
\begin{align*}
	\|\alpha_0^S\|^2_{\LL^2(\mu_0)} & = (s'(0))^2
	\int_X (\alpha_0-\ol{\alpha}_0)^2\diff \mu_0=
	(s'(0))^2 \left( \|\alpha_0\|^2_{\LL^2(\mu_0)}-\ol{\alpha}_0^2 \right)
\end{align*}
where $\ol{\alpha}_0^2=\HK_\kappa(\mu_0,\mu_1)^4/\kappa^4$, and consequently with \Cref{prop:HK_as_L2}, \eqref{eq:SHKLogSPrime} and finally \eqref{eq:SHK},
\begin{align*}
& \norm{v_0^S}^2_{\LL^2(\mu_0)} +\frac{\kappa^2}{4}\norm{\alpha_0^S}^2_{\LL^2(\mu_0)}+\kappa^2\abs{(\mu^S_1)^\perp} \\
= {} & (s'(0))^2 \cdot \left(
\norm{v_0}^2_{\LL^2(\mu_0)} +\frac{\kappa^2}{4}\norm{\alpha_0}^2_{\LL^2(\mu_0)}- \frac{\HK^4_\kappa(\mu_0,\mu_1)}{4\kappa^2} +\kappa^2\abs{\mu^\perp_1} \right) \\
= {} & \frac{\SHK_\kappa^2(\mu_0,\mu_1)}{\kappa^2\sin^2\left(\SHK_\kappa(\mu_0,\mu_1)/\kappa\right)}
\left( \HK_\kappa^2(\mu_0,\mu_1)- \frac{\HK^4_\kappa(\mu_0,\mu_1)}{4\kappa^2}\right) \\
= {} & \frac{\SHK_\kappa^2(\mu_0,\mu_1)}{\kappa^2-\kappa^2\left(1-\frac{\HK_\kappa^2(\mu_0,\mu_1)}{2\kappa^2}\right)^2}
\left( \HK_\kappa^2(\mu_0,\mu_1)- \frac{\HK^4_\kappa(\mu_0,\mu_1)}{4\kappa^2}\right) \\
= {} & \SHK_\kappa^2(\mu_0,\mu_1). 
\end{align*}
\end{proof}

\begin{remark}[Exponential map for $\SHK_\kappa$]
The formulas that provide the logarithmic map for $\SHK_\kappa$, based on that of $\HK_\kappa$ are invertible (using that \eqref{eq:SHK} is also invertible). Consequently, using this inversion and then using the exponential map for $\HK_\kappa$, one can define the exponential map for $\SHK_\kappa$. By construction it is, once again, a left-inverse of the logarithmic map, yielding
$$\Exp^{\SHK_\kappa}_{\mu_0}\circ\Log^{\SHK_\kappa}_{\mu_0}=\id.$$
We point out here that the slope $s'(0)$ did depend on the target measure $\mu_{1}$ and therefore, should be computed from the arguments of the exponential, using \eqref{eq:BB_SHK}.
\end{remark}

\subsection{Convexity of the range of the logarithmic map}
\label{sec:convexity_range}
%
A natural question in the context of linearized optimal transport is whether or not taking averages of embeddings (or interpolating between them) in the tangent space yields another `valid' tangent vector, i.e.~a vector that corresponds to an embedding of a measure via the logarithmic map. In other words: Is the range of the logarithmic map convex?

For $\W_2$ the logarithmic map can be expressed as gradient of a $c$-concave potential, \Cref{prop:W2_Log_dual}, and therefore convexity of the range of the logarithmic map is equivalent to convexity of the set of $c$-concave functions
(here we must carefully distinguish between \emph{convexity of a set} in a vector space of functions, and \emph{convexity of a function}, which means that its epigraph is a convex set).
For $X \subset \R^n$ with $c(x,y)=\|x-y\|^2$, a function $\phi$ is $c$-concave if and only if $x \mapsto \tfrac12(\|x\|^2-\phi(x))$ is convex, see \Cref{thm:Brenier}, and thus the answer is positive: the set of $c$-convex functions is a convex set.
More generally, convexity of the set of $c$-concave functions in turn is equivalent to the strong Ma-Trudinger-Wang condition MTW(0,0), see \cite[Definition 12.27]{villani2009optimal} for several formulations of this condition, such as the definite positivity of the Ma-Trudinger-Wang tensor, and \cite[Theorem 3.2]{figalli2011multidimensional} and \cite[Theorem 1.5.4]{gallouet2012optimal} for the fact that this condition yields the convexity of the set of $c$-concave functions.

Consider now the $\HK$ distance with the logarithmic map $(v_0,\alpha_0,\mu_1^\perp)=\Log_{\mu_0}^{\HK_\kappa}(\mu_1)$ as given in \Cref{def:HKLog_map}. For simplicity we restrict this discussion to the case where $\mu_1^\perp$ is zero. \eqref{eq:LogHK_dual} then suggests that the convexity of the range of the logarithmic map is equivalent to the convexity of the set of optimal potentials $\Phi_0$ for \eqref{eq:HK_logct_dual} (for fixed $\mu_0$ and arbitrary $\mu_1$).
By analogy with the balanced case, it is likely that this is equivalent to convexity of the set of log-$c$-concave functions, as introduced in \Cref{rem:logcconcave}.
Gallouët et al.~\cite{gallouet2021regularity} give some results on the MTW(0) condition (which is weaker than MTW(0,0)) for the cost $c_\kappa^{\HK}$. However, due to the non-linearity of the change of variable from dual problem \eqref{eq:HK_soft_dual} to \eqref{eq:HK_cone_dual},  at this point it is unclear whether this allows statements on the convexity of the set of log-$c$-concave functions and we are not aware of other relevant results in this direction.
For $\SHK$ the situation is even less clear due to the presence of an additional non-linear transformation by the prefactor $s'(0)$, see \eqref{eq:SHKLog}.

Some numerical experiments related to this are presented in Section \ref{sec:convexity_range_num}.
\section{Asymptotic stability of the logarithmic maps}
\subsection[Convergence of the HK tangent structure to the W2 one for large scales]{Convergence of the $\HK$ tangent structure to the $W_2$ one as $\kappa\to\infty$}
\label{sec:LinHKToLinW}
It is well-known that, as $\kappa\to\infty$, the distance $\HK_\kappa$ converges to the 2-Wasserstein distance, see for instance \cite[Theorem 7.24]{liero2018optimal} and \cite[Theorem 5.10]{ChizatDynamicStatic2018}. A very natural `sanity check' for our definition of their tangent structure is to verify whether or not the logarithmic map of $\HK_\kappa$ does also converge (in a suitable sense) to that of $\W_2$.

To establish this result it is convenient to exploit the convergence of the optimal plans $\pi_\kappa$ in \eqref{eq:HK_soft_primal} to a corresponding optimizer of \eqref{eq:W2_primal} as $\kappa \to \infty$. To the best of our knowledge, this result is only implicit in previous statements of the convergence of the distances that are references above. We will therefore give a brief sketch here.
\begin{lemma}
	\label{lem:CV_plans_delta}
	Let $\mu_0, \mu_1 \in \measp(K)$ for some compact $K \subset X$ be probability measures and $(\kappa_n)_{n \in \N}$ be a sequence with $\kappa_n\rightarrow+\infty$ as $n \to \infty$. For $n\in\N$, let $\pi_n$ be an optimal transport plan for $\HK_{\kappa_n}^2(\mu_0,\mu_1)$, minimizing \eqref{eq:HK_soft_primal}.
	
	Then (up to subsequences) $(\pi_n)_{n\in\N}$ converges weakly-$\star$ to some $\pi$ which is an optimal transport plan for $\W_2^2(\mu_0,\mu_1)$, minimizing \eqref{eq:W2_primal}. The optimal values also converge, i.e.
	$$\lim_{n\to \infty}\HK^2_{\kappa_n}(\mu_0,\mu_1)=\W_2^2(\mu_0,\mu_1).$$
\end{lemma}
\begin{proof}
	This is a simple Gamma-convergence result on the primal formulation of $\HK^2_{\kappa_n}(\mu_0,\mu_1)$ as $n\to\infty$ and we merely give a sketch. 
	
	Note that by compactness of $K$ and for $\kappa_n$ large enough, $c^{\HK}_{\kappa_n}$ is uniformly continuous on $K^2$ (since eventually $\diam(K)<\kappa_n\frac{\pi}{2}$) and converges uniformly to the quadratic cost $c=d^2$ of $\W_2$. Hence, for any converging sequence $(\pi_n)_n$ with limit $\pi$, one has $\int_{K \times K} c^{\HK}_{\kappa_n} \diff \pi_n \to \int_{K \times K} c \diff \pi$.

	For the liminf condition, in addition to the above convergence of the cost we need to consider the marginal constraints. When $\kappa_n^2 \KL(\proj_{i\#}\pi_n)$ remains bounded as $n \to \infty$, this implies that the limit $\pi$ lies in $\Pi(\mu_0,\mu_1)$ as required.
	The limsup condition is readily verified by taking the constant sequence $\pi_n=\pi$ as recovery sequence. In this case the Kullback--Leibler terms always yield zero and the cost term converges, as discussed above.

	By boundedness of $\HK_{\kappa_n}^2(\mu_0,\mu_1)$ we obtain that the sequence $(\pi_n)_n$ of optimal plans has bounded mass and tight marginals and is therefore weak* precompact. The previous Gamma-convergence then establishes that any cluster point $\pi$ solves \eqref{eq:W2_primal}.
\end{proof}

\begin{theorem}
	\label{prop:CV_tanvecs_kappa}	
	Let $K \subset X$ be compact, let $\mu_0 \in \probl(K)$, $\mu_1 \in \prob(K)$ and $(\kappa_n)_{n \in \N}$ be a positive sequence with $\kappa_n\rightarrow\infty$ as $n \to \infty$.
	Let $(v_0^N, \alpha_0^N, (\mu_1^N)^\perp)\assign\Log^{HK_{\kappa_n}}_{\mu_0}(\mu_1)$.
	
	Then, there will be some $N \in \N$ such that $(\mu_1^N)^\perp = 0$ for $n \geq N$. Furthermore, as $n \to \infty$, $v_0^N$ converges strongly in $\LL^2(\Surf,\R^n,\mu_0)$ to $v_0\assign\Log^{\W_2}_{\mu_0}(\mu_1)$ and $\alpha_0^N$ converges strongly in $\LL^2(\Surf; \mu_0)$ to 0.
\end{theorem}
\begin{proof}
	Let us take $N \in \N$ large enough, such that $\kappa_n \pi/2 > \diam(K)$ for $n \geq N$.
	By \Cref{prop:HK_L1_existence} \eqref{item:HK_L1_existence_perpfar} one has $d(x_0,x_1) \geq \kappa_N \pi/2$ for $\mu_0 \otimes (\mu_1^N)^\perp$ almost all $(x_0,x_1)$. Therefore, for measurable $A \subset K$ one has
	\begin{align*}
		(\mu_1^N)^\perp(A) & = \mu_0 \otimes (\mu_1^N)^\perp(K \times A)
		=\mu_0 \otimes (\mu_1^N)^\perp\left(\{(x_0,x_1) \in K \times A\,:\,d(x_0,x_1) \geq \kappa_n \pi/2\}\right)=0
	\end{align*}
	where the last equality follows because the considered subset of $K \times A$ is empty.
	
	For $N$ large enough, by \Cref{prop:LogW2Mfold} \eqref{item:LogW2Mfold_dist} and \Cref{prop:HK_as_L2} (where we use $(\mu_1^N)^\perp=0$) one has
	\begin{equation*}
		\begin{array}{ccc}
			\W_2^2(\mu_0,\mu_1) = \|v_0\|_{\LL^2(\Surf;\mu_0)}^2 &\text{ and }&
			\HK_{\kappa_N}^2(\mu_0,\mu_1) = \|v_0^N\|_{\LL^2(\Surf;\mu_0)}^2 + \tfrac{\kappa_N^2}{4} \|\alpha_0^N\|_{\LL^2(\Surf;\mu_0)}^2
		\end{array}
	\end{equation*}
	and since $\HK_{\kappa_N}^2(\mu_0,\mu_1) \uparrow \W_2^2(\mu_0,\mu_1)$ as $N\to\infty$ \cite[Theorem 7.24]{liero2018optimal}, we deduce that $\alpha_{N} \to 0$ strongly in $\LL^2(\Surf;\mu_0)$ and that
	\begin{align}
		\label{eq:LinHKConv_vnnorm}
		\limsup_{N \to \infty} \|v_0^N\|_{\LL^2(\Surf;\mu_0)}^2 \leq \|v_0\|_{\LL^2(\Surf;\mu_0)}^2.
	\end{align}
	We now show the strong convergence of $v_0^N$ to $v_0$.
	For $n\geq N$, let $\pi^n$ be optimal for the primal soft-marginal formulation \eqref{eq:HK_soft_primal} of $\HK_{\kappa_N}$.
	By \Cref{thm:McCann_HK} there is a map $T^N$ such that $\pi^N=(\id,T^N)_\# \pi_0^N$ where $\pi_0^N$ is the first marginal of $\pi^N$.
	Analogously let $\pi=(\id,T)_\# \mu_0$ be the optimal transport plan for $\W_2^2(\mu_0,\mu_1)$ in \eqref{eq:W2_primal}, uniqueness and existence of the map $T$ provided by \Cref{thm:McCann}. By \Cref{lem:CV_plans_delta}, $\pi^N\narrc{N\to\infty}\pi$ . 
	By the bound \eqref{eq:LinHKConv_vnnorm} the sequence $(v_0^N)_{N\in\N}$ must (up to selection of a subsequence) converge weakly in $\LL^2(\Surf;\mu_0)$ to some $\tilde{v}$. Below we will show that indeed $\tilde{v}=v_0$ (i.e.~the whole sequence converges weakly to this unique cluster point). By weak lower-semicontinuity of the norm and the bound \eqref{eq:LinHKConv_vnnorm}, we then get that $\lim_{n \to \infty} \|v_0^N\|_{\LL^2(\Surf;\mu_0)} = \|v_0\|_{\LL^2(\Surf;\mu_0)}$, which in combination with weak convergence will imply strong convergence of $v_0^N \to v_0$.
	
	To show $\tilde{v}=v_0$, let $\xi \in \Cont(\Surf)$. Then, recalling the expression for $v_0^N$ from \Cref{def:HKLog_map},
 \begin{align*}
 	\int_X \xi(x_0)\cdot& \tilde{v}(x_0)\diff\mu_0(x_0)\\
 	& = \lim_{N\to\infty} \int_X \xi(x_0) \cdot v_0^N(x_0)\,\diff\mu_0(x_0) \\
	&= \lim_{N\to\infty} \int_{\Surf\times\Surf}
	\kappa_N\xi(x_0)\cdot\frac{\Log^X_{x_0}(T^N(x_0))}{\norm{\Log^X_{x_0}(T^N(x_0))}}
	\tan\left(\frac{\norm{\Log^X_{x_0}(T^N(x_0))}}{\kappa_N}\right) \RadNik{\pi_0^N}{\mu_0}(x_0)\diff\mu_{0}(x_0)\\
	&= \lim_{N\to\infty} \int_{\Surf\times\Surf}
	\xi(x_0)\cdot \Log^X_{x_0}(x_1)
	\frac{\tan\left(\frac{\norm{\Log^X_{x_0}(x_1)}}{\kappa_N}\right)}{\norm{\Log^X_{x_0}(x_1)}/\kappa_N} \diff\pi^N(x_0,x_1)\\
	&=\int_{\Surf\times\Surf}\xi(x_0)\cdot\Log^X_{x_0}(x_1)\,\diff\pi(x_0,x_1)\\
 	&=\int_{\Surf\times\Surf}\xi(x_0)\cdot v_0 (x_0)\,\diff\mu_0(x_0)
 \end{align*}
The second to last equality here would simply be the narrow convergence of $\pi^N$ and uniform convergence of the integrated functions as $N\to 0$ if not for the fact that the logarithmic map $(x_0,x_1)\mapsto\Log_{x_0}^\Surf(x_1)$ might not be continuous everywhere on $K\times K$. Fortunately, a byproduct of the proof of \Cref{thm:McCann} is that for $\pi$-almost any $(x_0,x_1)$, the squared distance function $d^2(.,x_1)$ is differentiable at $x_0$. This is actually equivalent to the function $d^2$ being smooth at $(x_0,x_1)$, and in particular, the logarithmic map being continuous at $(x_0,x_1)$ (this is due to the link between the differentiability of the squared distance and the cut locus at $x_1$, see for instance \cite[Proposition 4.8]{sakai1996riemannian}). Furthermore, we have a trivial global (essential) bound on the norm of the $\Log^X_x$ map, namely $\diam(K)$. This is sufficient to apply the continuous mapping theorem (see for instance \cite[Theorem 2.3]{vaart_1998}) to be able to replace the limit integral with the integral of the limit. Finally, by the density of continuous functions in $\LL^2(\Surf,T\Surf;\mu_0)$, we have $\tilde{v}=v_0$ is the strong $\LL^2(\mu_{0})$ limit of $(v_0^N)_{N\in\N}$.
\end{proof}

\subsection{Convergence of barycentric projection for Wasserstein distance}
\label{sec:BarycentricW}
For the theoretical introduction of the logarithmic map for $\W_2$, $\HK_\kappa$ and $\SHK_\kappa$ it was necessary to assume that the reference measure $\mu_0$ had a Lebesgue density such that the optimal transport from $\mu_0$ to the sample $\mu_1$ is deterministic, i.e.~induced by an optimal map (in fact, slightly weaker assumptions on $\mu_0$ suffice, see e.g.~\cite{brenier1991polar}). Only then can the information about the transport be fully encoded in a velocity field (and an additional growth field for the unbalanced metrics).
In numerical practice however measures are often approximated by discretization and will then be concentrated on a finite number of points. The optimal transport plans are then in general not deterministic, i.e.~no map $T$ exists. On top of this, in practice one often uses entropic regularization for numerical approximation, which introduces additional non-determinism.
A popular remedy for the linearized $\W_2$ metric is `barycentric projection' where the mean velocity of all particles leaving a point $x_0$ is used as approximation \cite{wang2013linear}. A similar heuristic has been proposed in \cite{cai2022linearized} for the linearized $\HK_\kappa$ metric.
In this section we show convergence of the approximation of \cite{wang2013linear} as the approximation of the marginal measures and transport plans improves, thus providing a rigorous foundation for their use.
In the next section we show the analogous result for a slightly adjusted variant of \cite{cai2022linearized}.
As a corollary we gain some insight on the continuity (or lack thereof) of the logarithmic map for $\HK$.
Our results are not restricted to $\Surf \subset \R^n$, but also apply in the manifold setting.

We now introduce a generalized notion of the logarithmic maps for $\W_2$ and $\HK_\kappa$ that corresponds to the barycentric projection used as a stand-in for a transport map in \cite{cai2022linearized}, but can also be used to construct `tangent components' based on slightly sub-optimal transport plans (see \Cref{rem:HKentropicPiN}). As a consequence, we define them directly as a function of the transport plan instead of second measure $\mu_1$. Since such a plan does not have to be deterministic, the expressions for the Logarithmic components will be obtained as averages, according to a disintegration of the sub-optimal plan w.r.t. its first marginal (the one corresponding to the base measure $\mu_{0}$):
\begin{definition}
	for any measure $\pi \in \measp(\Surf \times \Surf)$ we denote by $(\pi(\cdot|x_0))_{x_0 \in X}$ its disintegration with respect to its first marginal $\pi_0=\proj_{0\#}\pi$, i.e.~the $\pi_0$-a.e.~unique family of measures satisfying
	\begin{equation*}
		\int_{\Surf \times \Surf} \phi(x_0,x_1)\,\diff \pi(x_0,x_1)
		= \int_X \left[ \int_X \phi(x_0,x_1)\,\diff \pi(x_1|x_0) \right] \diff \pi_{0}(x_0)
	\end{equation*}
	for any measurable test functions $\phi$.
\end{definition}
This is a specific case to a theorem generalizing the notion of Radon-Nikodim derivatives, see for instance \cite[Theorem 5.3.1]{ambrosio2005gradient} for the general statement. We will simply point out that, for a deterministic plan $\pi=(\id,T)_{\#} \sigma$, one has $\pi(\cdot|x)=\delta_{T(x)}$ for $\sigma$-almost all $x \in \Surf$. In the non-deterministic case $\pi(\cdot|x)$ is ($\sigma$-almost everywhere) a probability measure on $X$ which gives us the distribution of targets of particles starting at $x$.

\begin{definition}[Barycentric logarithmic map for $\W_2$]
\label{def:average_logmap_W}
For $\mu_0 \in \prob(K)$ and $\pi \in \prob(K \times K)$ with $\proj_{0\#}\pi=\mu_0$ and $(x,y)\mapsto\Log_x(y)$ defined $\pi-a.e.$, we set
	$$\Log_{\mu_0}^{W_2}(\pi)(x)\assign\int_X\Log^X_x(y)\diff\pi(y|x).$$
\end{definition}
When $\mu_0 \in \probl(K)$ and $\pi$ is an optimal plan for $W_2^2(\mu_0,\mu_1)$ to its second marginal $\mu_1 \assign \proj_{1\#}$ then this definition reduces to \Cref{def:LogW2Mfold}.

The next result then provides convergence of the barycentric projection approximation for the Wasserstein-2 case, when the approximated measure is absolutely continuous and as the discretization or approximation becomes increasingly refined.
\begin{theorem}
	\label{prop:CV_tanvecs_W2}
	Let $(\mu_{0}^N)_N$, $(\mu_{1}^N)_N$ be two sequences in $\prob(K)$, converging weakly* to $\mu_0 \in \probl(K)$ and $\mu_1 \in \prob(K)$ respectively.
	Let $(\pi^N)_N$ be a sequence with $\pi^N \in \Pi(\mu_0^N,\mu_1^N)$ such that $\pi^N \narrc{N\to\infty} \pi$ where $\pi$ is the optimal plan for $\W_2(\mu_0,\mu_1)$.
	Set $v_0^N \assign \Log^{\W_2}_{\mu_{0}^N}(\pi^N)$ as given by \Cref{def:average_logmap_W} and $v_0 \assign \Log^{\W_2}_{\mu_0}(\pi)(=\Log_{\mu_0}^{\W_2}(\mu_1))$, using either \Cref{def:average_logmap_W} or \Cref{def:LogW2Mfold}. Then the momentum measures $v_0^N\mu_{0}^N$ converge weakly* to $v_0\mu_{0}$ as $N \to \infty$.
\end{theorem}
Above the $\pi^N$ could, for instance, be approximately optimal plans for $W_2(\mu_0^N,\mu_1^N)$ with vanishing sub-optimality as $N \to \infty$, e.g.~computed with entropic regularization with vanishing entropy as $N \to \infty$.
This narrow convergence of momentum measures then implies strong $\LL^2(\mu_0)$ convergence of a suitable small transformation of the $v_0^N$ to the limit $v_0$, as the next result shows. In particular, for large $N$, values of $v_0^N$ on $\mu_0^N$ are on average good approximations of values of $v_0$ on nearby points of $\mu_0$ and thus the barycentric projection works as intended.
\begin{corollary}
\label{cor:strong_after_translation}
For the setup of \Cref{prop:CV_tanvecs_W2}, let $S^N$ be the optimal transport map for $W_2(\mu_0,\mu_0^N)$.
Then $v_0^N \circ S^N$ converges to $v_0$ strongly in $\LL^2(\mu_0)$.
\end{corollary}
\begin{proof}[Proof of \Cref{prop:CV_tanvecs_W2}]
The result is closely related to the stability of optimal transport plans under perturbations of the marginals \cite[Theorem 5.20, Corollary 5.23]{villani2009optimal}. We give it mainly as preparation for the more involved proof for the $\HK$-metric.

For a test function $\phi\in\Cont(K)$ we find that
\begin{align*}
		\int_K\phi(x)v_0^N(x)\diff\mu_0^N(x)
		&=\int_K\int_K\phi(x)\Log^X_x(y)\diff\pi^N(y|x)\diff\mu_0^N(x)\\
		&=\int_{K\times K}\phi(x)\Log^X_x(y)\diff\pi^N(x_0,x_1)
		\xrightarrow[N\to\infty]{} \int_{K\times K}\phi(x)\Log^X_x(y)\diff\pi(x_0,x_1)
\end{align*}
as $N \to \infty$ where we have once more applied the continuous mapping theorem (\cite[Theorem 2.3]{vaart_1998}, see proof of \Cref{lem:CV_plans_delta} for the same argument and more details) for the transition to the limit. Therefore, we can conclude that $v_0^N\mu_0^N$ converges narrowly towards $v_0 \mu_{0}$.
\end{proof}

\begin{proof}[Proof of \Cref{cor:strong_after_translation}]
First observe that,
\begin{align}
\norm{v_0^N \circ S^N}^2_{\LL^2(\mu_0)} & = \norm{v_0^N}^2_{\LL^2(\mu_0^N)}
	= \int_K\norm{v_0^N(x_0)}^2\diff\mu_0^N(x_0) \nonumber \\
	&\leq\int_K\int_K\norm{\Log_{x_0}(x_1)}^2\diff\pi^N(x_1|x_0)\diff\mu_0^N(x_0) \nonumber \\
	&=\int_{K\times K}d(x_0,x_1)^2\diff\pi^N(x_0,x_1)
	\to \W_2^2(\mu_0,\mu_1)	\label{eq:strong_proof_normbound}
\end{align}
where the inequality is due to Jensen for the probability measures $\pi^N(\cdot|x_0)$ and the convergence is due to $\pi^N \narrc{N\to\infty} \pi$.

Now set $\gamma^N\assign(\id,S^N)_\#\mu_{0}$ the optimal transport plan associated with $S^N$.
We first show that $v_0^N\circ S^N$ converges weakly to $v_0$ in $\LL^2(\mu_{0})$.
Indeed for any continuous vector field supported on $K$, $\xi:K\to TX$,
\begin{align*}
	& \left\lvert\int_K\xi(x)\cdot v_0^N(S^N(x))\diff\mu_{0}(x)-\int_K\xi(S^N(x))\cdot v_0^N(S^N(x))\diff\mu_{0}(x)\right\lvert\\
	= {} & \int_{K\times K}\abs{(\xi(x)-\xi(y))\cdot v_0^N(y)}\diff\gamma^N(x,y)\\
	\leq {} & \sqrt{\int_{K\times K}\norm{\xi(x)-\xi(y)}^2\diff\gamma^N(x,y)}
	\sqrt{\int_{K\times K}\norm{v_0^N(y)}^2\diff\gamma^N(x,y)}
\end{align*}
by the Cauchy--Schwarz inequality. The second factor equals $\|v_0^N\|_{\LL^2(\mu_0^N)}$, which is bounded due to \eqref{eq:strong_proof_normbound}. The first term tends to zero since $\gamma^N$ has to converge narrowly to the unique optimal plan from $\mu_{0}$ to $\mu_{0}$ for the quadratic cost, i.e.~$(\id,\id)_\#\mu_{0}$, and so
\begin{align*}
		\lim_{N\to\infty}\int_K\xi(x)\cdot v_0^N(S^N(x))\diff\mu_{0}(x)=&\lim_{N \to \infty}\int_K\xi(S^N(x))\cdot v_0^N(S^N(x))\diff\mu_{0}(x)\\
		=&\lim_{N \to \infty}\int_K\xi(y)\cdot v_0^N(y)\diff\mu_0^N(y)=\int_K\xi(y)\cdot v_0(y)\diff\mu_{0}(y)
\end{align*}
and we have the weak convergence $v_0^N \circ S^N \weakc{N\to\infty}{\LL^{2}(\mu_{0})}v_0$.
 
We augment this to strong convergence by showing that $\|v_0^N\circ S^N\|_{\LL^2(\mu_0)}$ also converges to $\norm{v}_{\LL^2(\mu_{0})}$.
From \eqref{eq:strong_proof_normbound} one has
$$\limsup_{n \to \infty}\norm{v_0^N \circ S^N}^2_{\LL^2(\mu_0)}\leq W_2^2(\mu_{0},\mu_1)=\norm{v}^2_{\LL^2(\mu_{0})}$$
where the last equality is due to \Cref{prop:LogW2Mfold} \eqref{item:LogW2Mfold_dist}.
By weak lower-semicontinuity of the norm this implies convergence of the norm and thus strong convergence.
\end{proof}

\subsection{Convergence of barycentric projection for Hellinger--Kantorovich distance}
\label{sec:BarycentricHK}
Now we show the corresponding results for the Hellinger--Kantorovich distance.
Here the situation is more involved due to the cut-off of transport at $\kappa \pi/2$.
We will find that the singular part of the logarithmic map $\mu_1^\perp$ will not be continuous (\Cref{ex:HKdiscontMu1Perp}), that mere asymptotic optimality of the plans $\pi^N$ is not sufficient for convergence of the logarithmic map (\Cref{ex:HKdiscontPiCos}) and additional conditions must therefore be imposed (\Cref{prop:HKexactPiN} and \Cref{rem:HKentropicPiN}), and that for singular $\mu_0$ the logarithmic map may not converge even when optimal transport maps would exist (\Cref{ex:HKdiscontVelSingular}).
The following is the analogue of \Cref{def:average_logmap_W}.
\begin{definition}[Barycentric logarithmic map for $\HK_\kappa$]
\label{def:average_logmap_HK}
Let $\mu_0 \in \measp(K)$, $\pi \in \measp(K \times K)$ with $\KL(\proj_{0\#}\pi|\mu_0)<\infty$ and $\int_{X \times X} c^{\HK}_\kappa \diff \pi < \infty$. Then we set
$$\Log_{\mu_0}^{\HK_\kappa}(\pi) \assign (v_0,\alpha_0)$$
where the components are defined for $\mu_{0}-$almost any $x\in X$ by
\begin{align*}
	v_0(x) & \assign \kappa \RadNik{\pi_0}{\mu_{0}}(x)\int_X \frac{\Log^X_x(y)}{\norm{\Log_x(y)}}\tan\left(\frac{\norm{\Log_x(y)}}{\kappa}\right)\diff\pi(y|x), &
	\alpha_0(x) &  \assign 2\left(\RadNik{\pi_0}{\mu_{0}}(x)-1\right).
\end{align*}
\end{definition}
When $\mu_0 \in \measpl(K)$ and $\pi$ is an optimal plan for $\HK_\kappa^2(\mu_0,\mu_1)$ for the formulation \eqref{eq:HK_soft_primal} then this definition reduces to \Cref{def:HKLog_map}, for $v_0$ and $\alpha_0$. We drop the singular part $\mu_1^\perp$ from the generalized definition since it does not exhibit meaningful continuity properties (see for instance \Cref{ex:HKdiscontMu1Perp}).
Note that the expression for $\alpha_0(x)$ does not depend on the transport target(s) $y$ associated to $x$ by $\pi$ (since $\pi_0$ is uniquely defined by $\mu_{0}$ and $\mu_{1}$) and therefore, averaging over the possible targets is unnecessary.
We then have the following stability result.
\begin{theorem}
	\label{prop:CV_tanvecs_HK}
	Let $(\mu_{0}^N)_N$, $(\mu_{1}^N)_N$ be two sequences in $\measp(K)$, narrowly converging to measures $\mu_0 \in \measpl(K)$ and $\mu_1 \in \measp(K)$ respectively.
	Let $(\pi^N)_N$ be a sequence in $\measp(K \times K)$ such that
	$\pi^N/\cos(d/\kappa) \narrc{N\to\infty} \pi/\cos(d/\kappa)$
	where $\pi$ is the optimal plan for $\HK_\kappa^2(\mu_0,\mu_1)$ in \eqref{eq:HK_soft_primal}.
	Set $(v_0^N,\alpha_0^N) \assign \Log^{\HK_\kappa}_{\mu_{0}^N}(\pi^N)$, and $(v_0,\alpha_0) \assign \Log^{\HK_\kappa}_{\mu_0}(\pi)$, as given by \Cref{def:average_logmap_HK}. Then the momentum measures $v_0^N\mu_{0}^N$ and $\alpha_0^N \mu_0^N$ narrowly converge to $v_0\mu_{0}$ and $\alpha_0 \mu_0$ as $N \to \infty$.
\end{theorem}

Necessity of the assumption about the convergence of $\pi^N/\cos(d/\kappa)$ is demonstrated in \Cref{ex:HKdiscontPiCos}. We show in \Cref{prop:HKexactPiN} that it is satisfied when the $\pi^N$ are actually optimal, and discuss the case of entropic approximation in \Cref{rem:HKentropicPiN}.
\begin{remark}
	We quickly mention here for clarity that, provided $\pi$ is the optimal plan for for $\HK_\kappa^2(\mu_0,\mu_1)$, the measure $\pi/\cos(d/\kappa)$ is finite. Indeed, due to optimality conditions \Cref{prop:HK_L1_existence}, \cref{item:HK_L1_marginals}, one has
	\begin{equation*}
			\int_{K \times K}\frac{\diff\pi(x,y)}{\cos \left(\frac{d(x,y)}{\kappa}\right)}=\int_{K \times K}\frac{\diff\pi(x,y)}{\sqrt{\RadNik{\pi_0}{\mu_{0}}(x)\RadNik{\pi_1}{\mu_{1}}(y)}}\leq\sqrt{\abs{\mu_0}\abs{\mu_1}}.
	\end{equation*}
\end{remark}
\begin{proof}[Proof of \Cref{prop:CV_tanvecs_HK}]
We start with the convergence of $\alpha^N \cdot \mu_0^N$. For a test function $\phi \in \Cont(X)$ one finds that
\begin{equation*}
		\int_X\phi(x)\alpha_0^N(x)\diff\mu_{0}^N(x)=2\left(\int_{\Surf\times \Surf}\phi(x)\diff\pi^N(x,y)-\int_{\Surf}\phi(x)\diff\mu_0^N(x)\right)
\end{equation*}
and the same expression holds for the limit. As by assumption, $\pi^N/\cos(\dist/\kappa) \narrc{N\to\infty} \pi/\cos(\dist/\kappa)$, one also has $\pi^N \narrc{N\to\infty} \pi$, from which we deduce that $\alpha_0^N \mu_0^N \narrc{N \to \infty} \alpha_0 \mu_0$.
As for the convergence of $v_0^N \cdot \mu_0^N$, consider a test function $\phi \in \Cont(X,TX)$. One has
\begin{align*}
& \int_{\Surf}\phi(x_0)\cdot v_0^N(x_0)\diff\mu_0^N(x_0) \\
={} &\kappa \int_{\Surf}\int_X \tan\left(\frac{\dist(x_0,x_1)}{\kappa}\right)\phi(x_0)\cdot
	\frac{\Log_{x_0}^\Surf(x_1)}{\norm{\Log_{x_0}^\Surf(x_1)}}
	\,\diff\pi^N(x_1|x_0)\RadNik{\pi_0^N}{\mu_{0}^N}(x_0)\,\diff\mu_0^N(x_0)\\
={} &\int_{\Surf\times \Surf}F(x_0,x_1)\frac{1}{\cos\left(\frac{\dist(x_0,x_1)}{\kappa}\right)}
	\,\diff\pi^N(x_0,x_1)
\end{align*}
where we grouped some terms into the function $F$, which is independent of $N$, bounded, and continuous except on a $\pi$-negligible set (see related discussion in the proof of \Cref{prop:CV_tanvecs_kappa}).
By the assumption that $\pi^N/\cos(d/\kappa) \narrc{N\to\infty} \pi/\cos(d/\kappa)$ and the continuous mapping theorem (again, see proof of \Cref{prop:CV_tanvecs_kappa}) this last integral converges to
\begin{equation*}
\int_{\Surf\times \Surf}F(x_0,x_1)\frac{1}{\cos\left(\frac{\dist(x_0,x_1)}{\kappa}\right)}\diff\pi(x_0,x_1)=\int_{\Surf}\phi(x)\cdot v_0(x)\diff\mu_0(x).
\end{equation*}
\end{proof}

The rest of this section deals mainly with the investigation of the strong convergence assumption for the plans from \Cref{prop:CV_tanvecs_HK}, namely $\pi^N/\cos(d/\kappa) \narrc{N\to\infty} \pi/\cos(d/\kappa)$. First, we show that it is satisfied when $\mu_0 \ll \vol$ and the $\pi^N$ are optimal plans for each $N$.
\begin{proposition}
\label{prop:HKexactPiN}
Let $(\mu_{0}^N)_N$, $(\mu_{1}^N)_N$ be two sequences in $\measp(K)$, converging narrowly to $\mu_0 \in \measpl(K)$ and $\mu_1 \in \measp(K)$ respectively.
Let $(\pi^N)_N$ be an optimal unbalanced plan for $\HK_\kappa^2(\mu_0^N,\mu_1^N)$ in \eqref{eq:HK_soft_primal} and likewise let $\pi$ be optimal for $\HK_\kappa^2(\mu_0,\mu_1)$.
Then $\pi^N/\cos(d/\kappa)$ narrowly converges to $ \pi/\cos(d/\kappa)$ as $ N\to\infty$.
\end{proposition}
The proof of \Cref{prop:HKexactPiN} uses the following \emph{semi-coupling} formulation of $\HK_\kappa$, introduced in \cite{ChizatDynamicStatic2018}.
\begin{remark}[Square-root measure]
\label{rem:SquareRootMeasure}
For $\rho, \sigma \in \measp(K)$ (and similarly for general compact metric spaces) the square root measure $\sqrt{\rho \sigma}$ is characterized by
$$\int_K \phi \diff\sqrt{\rho\sigma} \assign \int_K \phi \sqrt{\RadNik{\rho}{\tau}\RadNik{\sigma}{\tau}} \diff \tau$$
for all $\phi \in \Cont(K)$, where $\tau \in \measp(K)$ is such that $\rho,\sigma \ll \tau$. Due to the joint 1-homogeneity of the function $(r,s) \mapsto \sqrt{rs}$ the right integral does not depend on the choice of $\tau$.

Using that the function $(s,t) \mapsto s+t-2\sqrt{rs}$ is non-negative, convex, lower-semicontinuous and positively 1-homogeneous on $\R_+^2$, by integrating against continuous test functions, and with \cite[Theorem 2.38]{ambrosio2000functions}, we obtain that if $(\rho_n,\sigma_n) \narrc{N\to\infty} (\rho,\tau)$ (in the sense of narrow convergence of vector measures), then
$$\int_K \phi \diff \sqrt{\rho \sigma} \geq \limsup_n \int_K \phi \diff \sqrt{\rho_n \sigma_n}.$$
In particular $\sqrt{\rho \sigma} \geq \widehat{\sqrt{\rho \sigma}}$ holds for any narrow cluster point $\widehat{\sqrt{\rho \sigma}}$ of $\sqrt{\rho_n \sigma_n}$.
\end{remark}

\begin{proposition}[Semi-coupling formulation]
	\label{prop:HK_semi_couplings}
	For $\mu_{0}, \mu_1 \in \measp(K)$ one has
	\begin{multline}
		\label{eq:HK_semi_couplings}
		\HK_\kappa^2(\mu_{0},\mu_1)= \kappa^2 \inf\Bigg\{
			\|\gamma_{0}\|+\|\gamma_{1}\|
			-2\int_{\Surf \times \Surf}\Cos\left(\frac{\dist(x_0,x_1)}{\kappa}\right)
			\diff\sqrt{\gamma_0\gamma_1}(x_0,x_1)\\
		\Bigg|\gamma_{0},\gamma_1\in\measp(K\times K),\,p_{i\#}\gamma_i=\mu_i \tn{ for } i=0,1
		\Bigg\}
	\end{multline}
	and minimizers for the right-hand side exist.
	Furthermore, with $\pi$ denoting an optimal transport plan for the soft-marginal formulation \eqref{eq:HK_soft_primal}, optimal semi-couplings $(\gamma_0,\gamma_1)$ for \eqref{eq:HK_semi_couplings} are given by:
	\begin{align}
	\label{eq:GammaFromPi}
	\gamma_{0} & \assign\RadNik{\mu_{0}}{\pi_0}\pi+(\id,\id)_\#\mu_{0}^\perp, &
	\gamma_1 & \assign\RadNik{\mu_{1}}{\pi_1}\pi+(\id,\id)_\#\mu_1^\perp	
	\end{align}
	with the notations of \Cref{rem:scale}.
	Conversely, if $(\gamma_{0},\gamma_1)$ are optimal for \eqref{eq:HK_semi_couplings}, then
	\begin{align}
	\label{eq:PiFromGamma}
	\pi \assign \Cos\left(\frac{\dist(x_0,x_1)}{\kappa}\right)\sqrt{\gamma_{0}\gamma_1}
	\end{align}
	is a minimizer of the soft-marginal formulation \eqref{eq:HK_soft_primal}.
\end{proposition}
\begin{proof}
	Existence of minimizers and equivalence with \eqref{eq:HK_soft_primal} was established in \cite[Corollary 5.9]{ChizatDynamicStatic2018}.

	Formula \eqref{eq:GammaFromPi} follows from the following inequality (see \cite[Lemma 3.16]{cai2022linearized}) between the costs featured in the two problems: for $x_0,x_1\in \Surf$ and $u_0, u_1>0$,
		\begin{equation}
				\label{eq:soft_geq_cone}
				c_\kappa(x_0,x_1)+\kappa^2 \cdot \sum_{i \in \{0,1\}}\varphi\left(\frac{1}{u_i}\right)u_i\geq \kappa^2\left(u_0+u_1-2\sqrt{u_0u_1}\Cos\left(\frac{d(x_0,x_1)}{\kappa}\right)\right)
			\end{equation}
		where $\varphi(s)=s \log(s)-s+1$ is the integrand of the KL-divergence, \eqref{eq:KL}. Furthermore, equality happens if and only if
		$$\sqrt{u_0u_1}\Cos\left(\frac{d(x_0,x_1)}{\kappa}\right)=1.$$
	Integrating \eqref{eq:soft_geq_cone} with $u_i \assign \RadNik{\mu_i}{\pi_i}(x_i)$ against the optimal plan $\pi$ yields optimality of the specific semi-couplings in \eqref{eq:GammaFromPi}. 
	
	Formula \eqref{eq:PiFromGamma} is established for the constructed $\pi$ as follows:
	\begin{align*}
		E_\kappa(\pi|\mu_0,\mu_1) & = \int_{K \times K} c^{\HK}_\kappa \diff \pi
			+ \kappa^2 \sum_{i \in \{0,1\}}
		  \int_K \varphi(\RadNik{\pi_i}{\mu_i}) \diff \mu_i \\
		& \leq \int_{K \times K} c^{\HK}_\kappa \diff \pi + \kappa^2 \sum_{i \in \{0,1\}}
		  \int_{K \times K} \varphi(\RadNik{\pi}{\gamma_i}) \diff \gamma_i
		\intertext{where we used \cite[Lemma 3.15]{cai2022linearized} with $T=\proj_i$ and then finally expanding the formulae for $\varphi$ and $c^{\HK}_\kappa$ (or equivalently leveraging the equality case in \eqref{eq:soft_geq_cone}), the reader can check that the expression simplifies into:}
		& = \kappa^2 \left[
		\abs{\gamma_{0}}+\abs{\gamma_{1}}
			-2\int_{\Surf \times \Surf}\Cos\left(\frac{\dist(x_0,x_1)}{\kappa}\right)
			\diff\sqrt{\gamma_0\gamma_1}(x_0,x_1)\right]
	\end{align*}
	which equals $\HK^2_\kappa(\mu_{0},\mu_{1})$ and thus completes the proof.
\end{proof}

\begin{proof}[Proof of \Cref{prop:HKexactPiN}]
The narrow convergence of $\pi^N$ to $\pi$ follows directly from the joint lower-semicontinuity of $E_\kappa(\pi|\mu_{0},\mu_{1})$, \eqref{eq:HK_soft_primal}, with respect to all three measures and the triangle inequality for $\HK_\kappa$. The convergence of $\pi^N/\cos(d/\kappa)$ is however more involved.

In the following for $i \in \{0,1\}$, let
\begin{align}
\label{eq:HKBarycenterGammaI}
	\gamma_i^N & \assign \RadNik{\mu_i^N}{\pi_i^N} \cdot \pi^N + (\id,\id)_{\#} (\mu_i^N)^\perp, &
	\gamma_i & \assign \RadNik{\mu_i}{\pi_i} \cdot \pi + (\id,\id)_{\#} (\mu_i)^\perp,
\end{align}
where $\pi_i^{N}=\proj_{i\#}\pi^{N}$ (in the spirit of Remark \ref{rem:TransportAndCreation}).
By \Cref{prop:HK_L1_existence} one has
$$\frac{\pi^{N}}{\cos(d/\kappa)}=\sqrt{\gamma_0^{N} \gamma_1^{(N)}},$$
and by \Cref{prop:HK_semi_couplings} $(\gamma^{(N)}_0,\gamma^{(N)}_1)$ are optimal semi-couplings for $\HK_\kappa^2(\mu_0^{(N)},\mu_1^{(N)})$.
Finally, by narrow compactness we select a common subsequence and respective cluster points such that narrowly as $N \to \infty$ (on this subsequence)
\begin{align*}
	\frac{\pi^N}{\cos(d/\kappa)} & = \sqrt{\gamma_0^{N} \gamma_1^{N}} \to \rho, &
	\gamma_0^{N} & \to \tilde{\gamma}_0, &
	\gamma_1^{N} & \to \tilde{\gamma}_1.
\end{align*}
By standard lower-semicontinuity and the triangle inequality the cluster point $(\tilde{\gamma}_0,\tilde{\gamma}_1)$ are also optimal semi-couplings for $\HK^2_\kappa(\mu_0,\mu_1)$. In addition, we set
$$\tilde{\rho} \assign \sqrt{\tilde{\gamma}_0 \tilde{\gamma}_1},$$
and by \Cref{prop:HK_semi_couplings} and uniqueness of the optimal $\pi$ in \eqref{eq:HK_soft_primal} one has $\pi=\cos(d/\kappa) \cdot \tilde{\rho}$.

Let now
\begin{align*}
	A_= & \assign \{(x_0,x_1) \in X^2 | d(x_0,x_1)=\kappa \pi/2\}
\end{align*}
and in the same way define $A_<$ and $A_>$ for the point pairs closer and further than $\kappa \pi/2$ respectively.
From the above considerations we conclude
\begin{align*}
	\pi/\cos(d/\kappa) \restr_{A_<} & = \rho \restr_{A_<} = \tilde{\rho} \restr_{A_<}.
\end{align*}
This can be verified by integrating against continuous test functions $\phi$ with compact support in $A_<$. These do not see the singularity of $1/\cos(d/\kappa)$, and thus $\phi/\cos(d/\kappa)$ is still continuous test function and the above convergence is reduced to the already established convergence of $\pi^N$ to $\pi$.
By optimality of $\pi$ for \eqref{eq:HK_soft_primal} one has
\begin{align*}
	\pi/\cos(d/\kappa) \restr_{A_=} = \pi/\cos(d/\kappa) \restr_{A_>} = 0,
\end{align*}
and by the Portmanteau theorem one has
\begin{align*}
	\rho \restr_{A_>} = \tilde{\rho} \restr_{A_>} = 0.
\end{align*}
By \Cref{rem:SquareRootMeasure} one has furthermore that $\tilde{\rho} \geq \rho$. So in conclusion one has $\tilde{\rho} \geq \rho \geq \pi/\cos(d/\kappa)$ and the differences between the measures (if any) are concentrated on $A_=$.

Next, we prove that the components of $\gamma_i$ and $\tilde{\gamma}_i$ that are dominated by $\pi$ are identical, for $i=\{0,1\}$. We have already established that both $(\gamma_0,\gamma_1)$ and $(\tilde{\gamma}_0,\tilde{\gamma}_1)$ are optimal semi-couplings for $\HK^2_\kappa(\mu_0,\mu_1)$.
By convexity, any convex combination between the two, denoted by $(\gamma^t_0,\gamma^t_1)$ for $t \in [0,1]$, must also be optimal, and once more by \Cref{prop:HK_semi_couplings} and uniqueness of $\pi$ one must have $\sqrt{\gamma_0^t \gamma_1^t} = \pi$ is constant on $t \in [0,1]$.
The function $\R_+^2 \ni (a,b) \mapsto \sqrt{a b}$ is constant along line segments (of non-zero length) if and only if $a=0$ or $b=0$ along the whole segment. This implies that the $\gamma_i^t$ must be constant in $t$, $\pi$-almost everywhere. Recalling the Lebesgue decomposition of $\gamma_i$ with respect to $\pi$ from \eqref{eq:HKBarycenterGammaI} we can therefore now write for $i \in \{0,1\}$,
\begin{align*}
	\gamma_i & \assign \RadNik{\mu_i}{\pi_i} \cdot \pi + (\id,\id)_{\#} (\mu_i)^\perp, &
	\tilde{\gamma}_i & \assign \RadNik{\mu_i}{\pi_i} \cdot \pi + \tilde{\gamma}_i^\perp.
\end{align*}
We recall from \Cref{prop:HK_L1_existence} that $\mu_0^\perp$ and $\mu_1^\perp$ are mutually singular and in fact even
$\dist(\spt \mu_0,\allowbreak\spt \mu_1) \geq \kappa \pi/2$.
From this decomposition (and $\proj_{i\#} \tilde{\gamma}_i=\mu_i$) we also deduce that $\proj_{i\#} \tilde{\gamma}_i^\perp = \mu_i^\perp$.
We then observe that
$$\Delta \rho \assign \tilde{\rho}-\pi/\cos(d/\kappa)=\sqrt{\tilde{\gamma}_0^\perp \tilde{\gamma}_1^\perp}$$
from which we also infer that $\proj_{i\#} \Delta \rho \ll \mu_i^\perp$, which implies
\begin{align*}
\dist(\spt \proj_{0\#} \Delta \rho,\spt \mu_1^\perp) & \geq \kappa \pi/2, \\
\intertext{Also, we have shown earlier that $\Delta \rho$ is concentrated on $A_=$, and therefore}
\dist(\spt \proj_{0\#} \Delta \rho,\spt \mu_1^\perp) & \leq \dist(\spt \proj_{0\#} \Delta \rho,\spt \proj_{1\#} \Delta \rho) \leq \kappa \pi/2 \\
\intertext{and in combination}
\spt \proj_{0\#} \Delta \rho & \subset B=\{ x \in X\,|\, \dist(x,\spt \mu_1^\perp)=\kappa \pi/2 \}.
\end{align*}
Since
$$\spt \proj_{0\#} \Delta \rho \ll \mu_0^\perp \ll \mu_0 \ll \vol$$
and $\vol(B)=0$ we conclude that $\Delta \rho=0$ and thus $\tilde{\rho}=\rho=\pi/\cos(d/\kappa)$.
\end{proof}
Similarly to the balanced case, one can augment the narrow convergence of the momentum measures from \Cref{prop:CV_tanvecs_HK} to a strong $\LL^2(\mu_{0})$ convergence, up to composition with a transport map from $\mu_{0}/\abs{\mu_0}$ to $\mu_{0}^N/\abs{\mu_0^N}$. However, in the unbalanced case an additional regularity assumption is required. A counter-example, in the absence of the assumption, is given in \Cref{ex:HKdiscontMu1Perp}. 
\begin{corollary}
	\label{cor:HK_strong_L2}
	With the notations and under the assumptions of \Cref{prop:CV_tanvecs_HK}, assume furthermore that $\mu_0 \neq 0$ and
	\begin{equation}
	\label{eq:HK_strong_limsup}
	\limsup_{N \to \infty}\int_{\Surf\times \Surf}\RadNik{\pi_0^N}{\mu_{0}^N}(x_0)\frac{\diff\pi^N}{\cos\left(\frac{d(x_0,x_1)}{\kappa}\right)}\leq\int_{\Surf}\RadNik{\mu_1}{\pi_{1}}(x_1)\diff\pi_1(x_1).
	\end{equation}
	Then, denoting by $S^N$ the optimal transport map for $W_2$ from $\mu_{0}/\abs{\mu_{0}}$ to $\mu_{0}^N/\abs{\mu_{0}^N}$, $(v_0^N\circ S^N)_{N\in\N}$ and $(\alpha_0^N\circ S^N)_{N\in\N}$ converge strongly in $\LL^2(\mu_{0})$ towards respectively $v_0$ and $\alpha_0$.
\end{corollary}
\begin{remark}
The additional condition for \Cref{cor:HK_strong_L2} might seem unexpected at first. However, in the case were, for any $N$, $\pi^N$ is an optimal plan for the soft-marginal formulation of $\HK^2_\kappa(\mu_{0}^N,\mu_{1}^N)$, inequality \eqref{eq:HK_strong_limsup} simplifies to $\liminf_{N\to\infty}\abs{(\mu_1^N)^\perp}\geq \abs{\mu_{1}^\perp}$. Indeed, using the optimality conditions \cref{item:HK_L1_marginals} of \Cref{prop:HK_L1_existence}, \cref{eq:HK_strong_limsup} simplifies to 
		$$\limsup_{N\to\infty}\int_{X\times X}\RadNik{\mu_{1}N}{\pi_1^N}(x_1)\diff\pi^N(x_0,x_1)=\limsup_{N\to\infty}\int_{X}\RadNik{\mu_{1}N}{\pi_1^N}(x_1)\diff\pi_1^N(x_1)\leq\int_X\RadNik{\mu_{1}}{\pi_1}\diff\pi_1(x)$$ 
		and using the fact that $\lim_{N\to\infty}\abs{\mu_{1}^N}=\abs{\mu_{1}}$ and the definition of $(\mu_{1}^N)^\perp,\mu_{1}^\perp$, the former limsup inequality is equivalent to the latter liminf one.
		
		Note that we always have $\limsup_{N\to\infty}\abs{(\mu_1^N)^\perp}\leq \abs{\mu_{1}^\perp}$ from the narrow continuity of the squared $\HK_\kappa$ distance, $\lim_{N \to \infty}\HK^2_\kappa(\mu_{0}^N,\mu_1^N)=\HK_\kappa^2(\mu_{0},\mu_{1})$: Indeed, leveraging the weak $\LL^2(\mu_{0})$ convergence of $v_0^N\circ S_N$ and $\alpha_0^N\circ S_N$ to $v_0$ and $\alpha_0$ (which does not require to assume \eqref{eq:HK_strong_limsup}),
		\begin{equation*}
			\begin{split}
				\HK^2(\mu_{0},\mu_{1})=&\norm{v_0}_{\LL^2(\mu_{0})}^2+\frac{\kappa^2}{4}\norm{\alpha_0}_{\LL^2(\mu_{0})}^2+\kappa^2\abs{\mu_{1}^\perp}\\
				\leq&\liminf_{N\to\infty}\norm{v_0^N\circ S_N}_{\LL^2(\mu_{0})}^2+\frac{\kappa^2}{4}\norm{\alpha_0\circ S_N}_{\LL^2(\mu_{0})}^2+\kappa^2\abs{\mu_{1}^\perp}
			\end{split}
		\end{equation*}
		and rewritting 
		\begin{align*}
			\HK^2(\mu_{0},\mu_{1})=&\lim_{N \to \infty}\HK^2_\kappa(\mu_{0}^N,\mu_1^N)\\
			\geq& \limsup_{N\to\infty}\norm{v_0^N\circ S_N}_{\LL^2(\mu_{0})}^2+\frac{\kappa^2}{4}\norm{\alpha_0\circ S_N}_{\LL^2(\mu_{0})}^2+\kappa^2\abs{(\mu_{1}^N)^\perp},
		\end{align*}
		we have, \emph{a priori}, $\limsup_{N\to\infty}\abs{(\mu_{1}^N)^\perp}\leq\abs{\mu_{1}^\perp}$.	Therefore, the assumption of \Cref{cor:HK_strong_L2} is equivalent to $\lim_{N \to \infty}\abs{(\mu_{1}^N)^\perp}=\abs{\mu_{1}^\perp}$ (which may not always be verified, see \Cref{ex:HKdiscontMu1Perp}).
\end{remark}
\begin{remark}[Extension to SHK]
We remark without proof that by leveraging the relationship between the logarithmic maps for the $\HK$ and $\SHK$ metrics, one can obtain the same stability results, under the same assumptions as \Cref{prop:CV_tanvecs_HK} (resp.~\Cref{cor:HK_strong_L2}) for the $\SHK$ logarithmic map. For instance, the continuity of the scaling factors
$$s'(0)=\frac{\SHK_\kappa(\mu_{0}^N,\mu_{1}^N)/\kappa}{\sin(\SHK_\kappa(\mu_{0}^N,\mu_{1}^N)/\kappa)}$$ 
is a consequence of the continuity of the $\SHK$ distance with respect to the narrow convergence of measures.
\end{remark}

\begin{remark}[Re-scaling of HK tangent vectors for SHK conversion]
The barycentric logarithmic map of \Cref{def:average_logmap_HK} is not exactly invertible due to the averaging of the velocity field. As this section shows, in the limit the true logarithmic map is recovered.
However, when working numerically with SHK, small deviations at finite $N$ can already lead to undesirable errors.
Let $\mu_0,\mu_1 \in \prob(X)$, let $\pi$ be an (approximately) optimal non-deterministic plan. Computing $(v_0,\alpha_0)$ as in \Cref{def:average_logmap_HK} one finds in general that $\tilde{\mu}_1 \assign \Exp_{\mu_0}^{\HK_\kappa}(v_0,\alpha_0) \neq \mu_1$ and $\tilde{\mu}_1 \notin \prob(X)$. In this case, the formulas of Section \ref{sec:LinSHK} for the conversion to SHK tangent vectors become inconsistent as they rely on the assumption $\tilde{\mu}_1 \in \prob(X)$. This can be remedied, by re-scaling the vectors $(v_0,\alpha_0)$ first. Indeed, using \Cref{prop:HKExp} one can show for re-scaled tangent vector $(\tilde{v}_0,\tilde{\alpha}_0) \assign (q \cdot v_0,q \cdot \alpha_0+2(q-1))$ for $q=1/\sqrt{|\tilde{\mu}_1|}$ that $\Exp_{\mu_0}^{\HK_\kappa}(\tilde{v}_0,\tilde{\alpha}_0)=\tilde{\mu}_1/|\tilde{\mu}_1| \in \prob(X)$.
\end{remark}

\begin{example}[Discontinuity of $\mu_1^\perp$]
\label{ex:HKdiscontMu1Perp}
The singular component $\mu_1^\perp$ of the logarithmic map does not necessarily converge, even under the assumptions of \Cref{prop:CV_tanvecs_HK} (despite it being uniquely defined for any $N$, unlike the velocity $v_0^N$ which potentially depends on the choice the optimal plan $\pi^N$ solving \eqref{def:HK_soft_primal}).
Indeed, let $X=\R$, $\kappa=1$ and consider for some $L > \pi/2+1$,
\begin{align*}
	\mu_0^N & \assign (1-1/N) \Leb\restr_{[0,1]} + 1/N \Leb\restr_{[L,L+1]}
	\rightweaks \mu_0 \assign \Leb\restr{[0,1]}, \\
	\mu_1^N & \assign 1/N \Leb\restr_{[0,1]} + (1-1/N) \Leb\restr_{[L,L+1]}
	\rightweaks \mu_1 \assign \Leb\restr{[L,L+1]}.
\end{align*}
One can easily see that for every $N$ one has
$$\pi^N=\sqrt{(1-1/N)1/N} (\id,\id)_{\#} [\Leb\restr_{[0,1]} + \Leb\restr_{[L,L+1]}]$$
and so $(\mu_i^N)^\perp=0$, whereas for the limit one has $\pi=0$ and thus $\mu_i^\perp=\mu_i \neq 0$.
	
This also readily gives an example where the strong $\LL^2$-convergence is not verified (although we make no claims as to the necessity of the assumption $\abs{(\mu_{1}^N)^\perp}\xrightarrow[N\to\infty]{}\abs{\mu_{1}^\perp}$ in \Cref{cor:HK_strong_L2}). Indeed, since $(\mu_i^N)^\perp=0$, and $\mu_{0}^N \ll \Leb$, one has (after some straightforward calculations),
	$$\HK_1^2(\mu_{0}^N,\mu_{1}^N)=\norm{v_0^N}^2_{\LL^2(\mu_{0}^N)}+\frac{1}{4}\norm{\alpha_0^N}^2_{\LL^2(\mu_{0}^N)} \xrightarrow[N\to\infty]{}\HK_1^2(\mu_{0},\mu_{1})$$
	whereas
	$$\norm{v_0}^2_{\LL^2(\mu_{0})}+\frac{\kappa^2}{4}\norm{\alpha_0}^2_{\LL^2(\mu_{0})}=\abs{\mu_{0}}<\HK_1^2(\mu_{0},\mu_{1})=\abs{\mu_0}+\abs{\mu_1},$$
	forbidding the strong $\LL^2$ convergence. In fact, one can even check that $v^N_0=v_0=0$ as neither $\pi^N$ nor $\pi$ represent any movement of mass particles but
	$$\alpha^N_0=\begin{cases}
		2 \left(\sqrt{\frac{1/N}{1-1/N}}-1\right) &\text{on $[0;1]$}\\
		2 \left(\sqrt{\frac{1-1/N}{1/N}}-1\right)&\text{on $[L;L+1]$}
	\end{cases}$$
	and the expression on $[L;L+1]$ diverges to $+\infty$ as $N\to\infty$ (however, this diverging expression is exactly compensated by the density of $\mu_{0}^N$ in the moment measure $\alpha_0^N\mu_{0}^N$, allowing us to still have the claim of \Cref{prop:CV_tanvecs_HK}).
\end{example}

\begin{example}[Non-convergence of $\pi^N/\cos(d/\kappa)$ due to sub-optimality of $\pi^N$]
\label{ex:HKdiscontPiCos}
For $\mu_1 \ll \vol$ and $\pi^N$ optimal, \Cref{prop:HKexactPiN} implies convergence of $\pi^N/\cos(d/\kappa)$ to $\pi/\cos(d/\kappa)$. We show that slight sub-optimality of $\pi^N$, such that still $\pi^N \to \pi$, can already hinder the convergence $\pi^N/\cos(d/\kappa)$ to $\pi/\cos(d/\kappa)$, even when $\mu_1 \ll \vol$.
Let $X=\R$, for simplicity let $\kappa=1$, and consider the identical sequences of measures
$$\mu_{0}^N \assign \mu_1^N\assign\underbrace{\Leb\restr_{[0;1]}}_{\assignRe \mu_{00}}
+\Leb\restr_{[\frac{\pi}{2}-\frac{1}{N};1+\frac{\pi}{2}-\frac{1}{N}]}$$
narrowly converging to the limit measures $\mu_{0} \assign \mu_{1}\assign\Leb\restr_{[0;1]}+\Leb\restr_{[\frac{\pi}{2};1+\frac{\pi}{2}]}$. The obvious optimal plans $\pi^N$ for the soft marginal formulation of $\HK^2(\mu_{0}^N,\mu_{1}^N)$ are induced by the identity map, $\pi^N\assign(\id,\id)_\#\mu_{0}^N$, and similarly at the limit, $\pi\assign(\id,\id)_\#\mu_{0}$.
Consider now the sequence of perturbed plans $\tilde{\pi}^N\assign\left(1-\frac{1}{N}\right)\pi^N+\frac{1}{N}(\id,T^N)_\#\mu_{00}$ where $T^N:x\in\R\mapsto x+\frac{\pi}{2}-\frac1N$. We notice that both $\pi^N$ and $\tilde{\pi}^N$ narrowly converge to $\pi$ and in fact, one even has
\begin{equation*}
	\begin{split}
		E(\tilde{\pi}^N|\mu_{0}^N,\mu_{1}^N)=&\left(1-\frac{1}{N}\right)\times 0-\frac{2}{N}\log\left(\cos\left(\frac{\pi}{2}-\frac{1}{N}\right)\right)+\KL\left(\left(1-\frac1N\right)\mu_{0}^N+\frac1N\mu_{00}~\middle|~\mu_{0}^N\right)\\
		&+\KL\left(\left(1-\frac1N\right)\mu_{0}^N+\frac1N (T^N_\#\mu_{00})~\middle|~\mu_{0}^N\right)\\
		&\xrightarrow[N\to\infty]{}0=\HK^2(\mu_{0},\mu_{1}).
	\end{split}
\end{equation*}
However, 
\begin{equation*}
		\frac{\tilde{\pi}^N}{\cos(d)}=\left(1-\frac{1}{N}\right)\pi^N+\frac{1}{N}\frac{(\id,T^N)_\#\mu_{00}}{\cos\left(\frac\pi 2-\frac1N\right)} \\
		\narrc{N\to\infty}\pi+(\id,T)_\#\mu_{00}\neq\frac{\pi}{\cos\left(d\right)}=\pi
\end{equation*}
where $T:x\in\R\mapsto x+\frac{\pi}{2}$.
And indeed, for the logarithmic maps, $v_0^N,\alpha_0^N\assign\Log_{\mu_0^N}(\tilde{\pi}^N)$ and $v_0,\alpha_0\assign\Log_{\mu_{0}}(\mu_{1})=(0,0)$, we obtain that
$$v_0^N\mu_{0}^N\narrc{N\to\infty}0+\frac{T-\id}{\norm{T-\id}}\mu_{00}\neq v_0\mu_{0}.$$
\end{example}

\begin{example}[Non-convergence of $\pi^N/\cos(d/\kappa)$ due to singular $\mu_1$]
\label{ex:HKdiscontVelSingular}
In this example, the $\pi^N$ are optimal, but assumption $\mu_0 \ll \vol$ is violated, resulting in non-convergence of $\pi^N/\cos(d/\kappa)$.
Let $X=\R$, $\kappa=1$, and set $\mu_{0}^N \assign \mu_{0} \assign \delta_0$, as well as $\mu_{1}^N \assign \delta_{\pi/2-1/N}$, converging narrowly to $\mu_1 \assign \delta_{\pi/2}$.
One quickly finds that the unique optimal plan for the soft-marginal formulation of $\HK_1^2(\mu_{0}^N,\mu_{1}^N)$ is given by $\pi^N\assign\cos(\pi/2-1/N)\cdot \delta_0\otimes\delta_{\pi/2-1/N}$. Between $\mu_0$ and $\mu_1$ the optimal plan is $\pi=0$. Therefore, we have $\pi^N \to \pi$, but $\pi^N/\cos(d) \to \delta_0 \otimes \delta_{\pi/2} \neq \pi$.
At the level of logarithmic maps one obtains for $(v_0^N,\alpha_0^N)\assign\Log_{\mu_0}^{\HK_1}(\pi^N)$ (defined only at $x_0=0$) that
\begin{align*}
v_0^N & =\sin\left(\frac\pi2-\frac1N\right), &
\alpha_0^N & =2\left[\Cos\left(\frac\pi2-\frac1N\right)-1\right]
\end{align*}
for which $\lim_{N \to \infty} (v_0^N,\alpha_0^N)=(1,-2)$. On the other hand, $v_0,\alpha_0\assign\Log^{\HK_1}_{\mu_{0}}(\pi)=(0,-2)$ and we have neither the narrow convergence of $v_0^N\mu_{0}^N$ to $v_0\mu_{0}$ nor the strong $\LL^2(\mu_{0})$-convergence $v_0^N\circ S_N\xrightarrow[N\to\infty]{}v_0$, with the notations of \Cref{cor:HK_strong_L2}.
\end{example}

\begin{remark}[Entropic regularization]
\label{rem:HKentropicPiN}
A relevant sequence of plans $\pi^N$ is given by the solution of the entropic regularization of the soft-marginal formulation,
\begin{equation*}
	\HK^2_{\kappa,\epsilon}(\mu_{0},\mu_{1})\assign \min_{\pi}E_\kappa(\pi|\mu_{0},\mu_{1})+\epsilon \KL(\pi|\mu_{0}\otimes\mu_{1}).
\end{equation*}
We refer the reader to \cite{chizat2018scaling} for a more extensive introduction to this regularized unbalanced transport problem and its numerical resolution. This problem admits a dual formulation,
\begin{align}
	\label{eq:HK_soft_entropic_dual}
	\HK^2_{\kappa,\epsilon}(\mu_{0},\mu_{1})=\sup_{u_0,u_1\in\Cont(K)}&\kappa^2\left(\int_\Surf 1-e^{-u_0(x_0)/\kappa^2}\diff\mu_{0}(x_0)+\int_\Surf 1-e^{-u_1(x_1)/\kappa^2}\diff\mu_{1}(x_1)\right) \nonumber\\
	&+\epsilon\int_{\Surf}1-e^{\frac{u_0(x_0)+u_1(x_1)-c^{\HK}_{\kappa}(x_0,x_1)}{\epsilon}}\diff\mu_{0}(x_0)\otimes\mu_{1}(x_1)
\end{align}
and a corresponding primal-dual optimality condition for $u_{\veps,0}, u_{1,\veps}\in\Cont(K)$ and $\pi_\veps\in\meas(K\times K)$,
\begin{equation}
	\label{eq:ent_UOT_opti}
	\pi_\veps=e^{\frac{u_{\veps,0} \oplus u_{\veps,1}-c^{\HK}_{\kappa}}{\epsilon}} \cdot \mu_{0}\otimes\mu_{1}.
\end{equation}
For simplicity, assume that optimal $u_{\veps,i} \in \Cont(K)$ exist for all $\veps \geq 0$ and that one has uniform convergence of $u_{\veps,i}$ to unregularized limit solutions $u_{0,i}$. This implies that the entropic plans $\pi_\epsilon$ narrowly converge to some $\pi$ that minimizes the unregularized problem. Indeed in that case, the entropic term in the dual formulation \eqref{eq:HK_soft_entropic_dual} vanishes as $\epsilon\to0$ and therefore, any cluster point of the (tight) family $(\pi_\epsilon)_\epsilon$ has to be optimal. Furthermore, using \eqref{eq:ent_UOT_opti} and the definition of $c^{\HK}_\kappa$, one finds
\begin{align*}
	\frac{\pi_\veps}{\cos(d/\kappa)} = \exp\left(
	\left[2\kappa^2\log \Cos(d/\kappa)-\veps \log \Cos(d/\kappa)+u_{\veps,0} \oplus u_{\veps,1}\right]
	/\veps\right) \cdot \mu_0 \otimes \mu_1
\end{align*}
with the convention that $\frac{\pi_\veps}{\cos(d/\kappa)}$ assigns zero mass to points at distance $\kappa \pi/2$ or larger (as does \eqref{eq:ent_UOT_opti}), and that the exponential function on the right side evaluates to zero for $d(x,y)>\kappa \pi/2$.

Then, using the uniform boundedness of the sequences $u_{\veps,i}$, $i=0,1$, one gets that there exist $\eta_0>0$ and $C>0$ such that for any $\eta<\eta_0$, $0<\veps<2\kappa^2$ and $\pi/2-\eta<d(x,y)/\kappa<\pi/2$,
$$2\kappa^2\log \cos(d(x,y)/\kappa)-\veps \log \cos(d(x,y)/\kappa)+u_{\veps,0}(x) + u_{\veps,1}(y)\leq C\log \cos(d(x,y)/\kappa).$$
Consequently, for any $\phi\in\Cont(K^2)$ and $\eta<\eta_0$,
\begin{align*}
  & \limsup_{\epsilon\to 0} \abs{
    \int_{K^2}\phi(x,y)\diff\frac{\pi_\epsilon(x,y)}{\cos\left(d(x,y)/\kappa\right)}
    -\int_{K^2}\phi(x,y)\diff\frac{\pi(x,y)}{\cos\left(d(x,y)/\kappa\right)}} \\
  \leq & \limsup_{\epsilon\to 0} 
    \int_{d(x,y)/\kappa>\pi/2-\eta}\abs{\phi(x,y)}e^{C\log \cos(d(x,y)/\kappa)/\epsilon}\diff\mu_{0}(x)\diff\mu_{1}(y) \\
  & \qquad + \int_{d(x,y)/\kappa>\pi/2-\eta}\abs{\phi(x,y)}\diff\frac{\pi(x,y)}{\cos\left(d(x,y)/\kappa\right)}\\
  & \qquad + \abs{
    \int_{d(x,y)/\kappa\leq\pi/2-\eta}\phi(x,y)\diff\frac{\pi_\epsilon(x,y)}{\cos\left(d(x,y)/\kappa\right)}
    -\int_{d(x,y)/\kappa\leq\pi/2-\eta}\phi(x,y)\diff\frac{\pi(x,y)}{\cos\left(d(x,y)/\kappa\right)}}\\
		\leq&\int_{d(x,y)/\kappa>\pi/2-\eta}\abs{\phi(x,y)}\diff\frac{\pi(x,y)}{\cos\left(d(x,y)/\kappa\right)}
\end{align*}
and the desired convergence is obtained by taking $\eta\to 0$.
\end{remark}

\section{Numerical examples}
\label{sec:Numerics}

\subsection{Preliminaries}
\label{sec:NumericsPrel}
\paragraph{Scope} The numerical experiments in \cite[Section 5.2]{cai2022linearized} focused on the comparison between the linearization of $W_2$ and $\HK$. In particular they demonstrated the susceptibility of balanced transport to misinterpret small mass fluctuations as transport, even when they happen parts of the measures that are far from each other. The Hellinger--Kantorovich can explain such mass discrepancies purely as mass vanishing/creation thus better capturing the local transport variations in the data set.
In this section, we focus on a comparison between $\HK$ and $\SHK$. We will demonstrate that the flexibility of $\HK$ to change the total mass of measures introduces a systematic bias towards lower masses when averaging between samples.
The $\SHK$ metric retains the ability to create and destroy mass locally, but is subjected to the constraint that the total mass must be preserved. It is therefore still able to deal gracefully with long range mass fluctuations but avoids the bias of $\HK$.
In addition, we give an example for linearized optimal transport analysis of measures living on a sphere.
Finally, we give some numerical illustrations for the (non-)convexity of the range of the logarithmic map in balanced and unbalanced linearized optimal transport.

\paragraph{Code} As usual, for our numerical experiments we discretize measures as weighted point clouds.
The discrete (unbalanced) optimal transport problems are solved via entropic regularization and a generalized Sinkhorn algorithm \cite{chizat2018scaling} (see \cite{SchmitzerScaling2019} for some algorithmic details). The length scale of the entropic blur is set to be comparable to the nearest-neighbour distance of the discrete point clouds to somewhat dampen discretization artefacts.
The code (Python/Numpy) for the discrete linearized optimal transport analysis is an updated version of that used in \cite{cai2022linearized}, accounting for some extensions introduced in the present article. This includes support for the SHK metric, Section \ref{sec:LinSHK}; an improved implementation of the barycentric projection for the $\HK$ metric, as analyzed in Section \ref{sec:BarycentricHK}; and support for manifolds as base spaces. The updated code is available at \url{https://gitlab.gwdg.de/bernhard.schmitzer/linot}.

\paragraph{Experiment protocol} In each of the following experiments, we choose a reference measure $\mu_0 \in \prob(X)$ and a set of samples $\{\nu_1,\ldots,\nu_N\} \subset \prob(X)$.
We then compute the embeddings $w_i=(v_i,\alpha_i) \assign \Log_{\mu_0}(\nu_i) \in H$ where $\Log_{\mu_0}$ is the logarithmic map for $\HK_\kappa$ or $\SHK_\kappa$ and $H$ is the corresponding tangent space of velocity and growth fields, as given in \Cref{prop:HK_as_L2}, and \Cref{prop:SHK_as_L2}.
For simplicity, we pick examples where the singular component of unbalanced transport is zero (cf.~$\mu_i^\perp$ in \Cref{rem:TransportAndCreation}). So $H$ is isometric to $(\LL^2(\mu_0))^{n+1}$. When $\mu_0$ is a discrete measure on $k$ points, then $H$ can be embedded isometrically into $\R^{k \cdot (n+1)}$ (one velocity vector and a growth value at each discrete point).
We then perform principal component analysis (PCA) on the collection of embeddings $\{w_i\}_{i=1}^N$, i.e.~we diagonalize the empirical covariance matrix of the $\{w_i\}_{i=1}^N$. The vectors $w_i$ can then be projected to the subspace spanned by the first $m$ dominant eigenvectors (corresponding to the largest eigenvalues) to obtain low-dimensional embeddings of the samples.
An eigenvalue/eigenvector pair $(\sigma^2,u)$ of the PCA indicates that there is variance $\sigma^2$ along the direction $u$ in the set $\{w_i\}_{i=1}^N$. To interpret the direction $u$, one can use the exponential map and look at the path of measures $[-\sigma,\sigma] \ni t \mapsto \assign \Exp_{\mu_0}(\ol{w}+t \cdot u)$ where $\ol{w}=\sum_{i=1}^N w_i/N$ is the empirical mean of the embedded samples and the range of $t$ is chosen to reflect one standard deviation of variation in both directions.

Since we consider discrete measures of the form $\mu_0=\sum_{i=1}^k m_i^0\,\delta_{x_i^0}$, the measure $\nu^w\assign\Exp_{\mu_0}(w)$ for some $w \in H$ will be of the form $\nu^w=\sum_{i=1}^k m_i^w\,\delta_{x_i^w}$, i.e.~we change positions and weights of the original point cloud.
If the reference point cloud $(x_i^0)_{i=1}^k$ is taken from a regular Cartesian grid with constant point density (or an equivalent structure on a manifold), after applying the exponential map, this is in general no longer true.
Therefore, to take these point density fluctuations into account during visualization we re-rasterize the deformed point cloud $(x_i^w)_{i=1}^k$ (and its weights $m_i^w$) back to $(x_i^0)_{i=1}^k$. To dampen Moir\'e artefacts, a small additional blur kernel is applied to the rasterized images.

\subsection{Linearized HK and SHK in the plane}
\paragraph{Dirac measures along a line}
The aforementioned bias of $\HK$ towards smaller masses can be studied analytically in the case of single Dirac measures on the interval $[0;L]$, for $0 < L < \kappa \pi/2$ (assuming the upper bound to avoid the pure Hellinger regime for simplicity).
In this case one has that
\begin{align}
\HK_\kappa^2(m_0\delta_{x_0},m_1 \delta_{x_1}) & =\kappa^2 \left(m_0+m_1-2\sqrt{m_0m_1} \Cos(|x_0-x_1|/\kappa)\right) \nonumber \\
\label{eq:HKDirac}
& =\kappa^2 \left\|\sqrt{m_0} \exp(ix_0/\kappa)-\sqrt{m_1} \exp(ix_1/\kappa)\right\|^2_{\C}.
\end{align}
This can be obtained, for instance, by explicitly solving \eqref{eq:HK_soft_primal} for this problem. The $\KL$-terms acting on the marginals imply that any admissible $\pi$ must satisfy $\proj_{i\#} \pi \ll m_i\delta_{x_i}$ and thus the minimizer must satisfy $\pi=m \delta_{x_0} \otimes \delta_{x_1}$ for some $m \geq 0$. By explicitly solving the remaining one-dimensional optimization over $m$ one obtains
\begin{equation}
\label{eq:PiDiracDirac}
\pi=\Cos\left(\abs{x_1-x_0}/\kappa\right)\sqrt{m_0 m_1} \delta_{x_0}\otimes\delta_{x_1}.
\end{equation}
In particular, in this case an optimal transport map exists, even though $m_0 \delta_{x_0}$ is not dominated by the Lebesgue measure. So the linearized optimal transport framework is applicable.

We pick $\mu_0 \assign \delta_{x_0}$, $x_0 \assign L/2$ as reference measure.
As can be  seen from \eqref{eq:HKDirac}, on the subset of Dirac measures $\{m \delta_x | m \geq 0, x \in [0;L]\}$ the $\HK_\kappa$ metric is flat, since it can be isometrically embedded into $\C$.
Consequently, we expect that the tangent space approximation of such measures around $\mu_0$ is exact.
As sample measures we pick the one-dimensional parametric family of probability Dirac measures $\{\delta_x| x\in [0;L]\}$.
According to \eqref{eq:HKDirac} this is isometric to a circle segment in $\C$.
For $\mu_1\assign\delta_{x_1}$, $x_1 \in [0;L]$, \eqref{eq:PiDiracDirac} yields that $T(x_0)=x_1$ and $\RadNik{\pi_0}{\mu_0}=\Cos\left(\abs{x_1-x_0}/\kappa\right)$ (recall that we chose $m_0=m_1=1$ here). So with \eqref{eq:def_HK_log} one finds for $(v_0,\alpha_0,\mu_1^\perp)=\Log_{\mu_0}^{\HK_\kappa}(\mu_1)$ the components
\begin{align}
\label{eq:HKDiracLog}
v_0(x_0)&=\kappa\sin\left(\frac{x_1-x_0}{\kappa}\right),
&
\alpha_0(x_0)&=2\left(\cos\left(\frac{\abs{x_1-x_0}}{\kappa}\right)-1\right), &
\mu_1^\perp & = 0.
\end{align}
Indeed, one sees that the embedding describes a circle segment as $x_1$ is varied.
Therefore, the mean of the embedded measures will then lie strictly within the circle (which corresponds to mass less than 1 by \eqref{eq:HKDirac}), a PCA of the embedded samples will yield two non-zero eigenvalues, and the corresponding PCA projection will show the circle shape. 

In the $\SHK$ metric, the creation and destruction of mass is still allowed, but the total mass must be preserved. In the case of single Dirac measures, this means that the metric structure reduces to the standard $\W_2$ distance. Indeed, for $m_0=m_1=1$ \eqref{eq:HKDirac} simplifies to $\HK_\kappa^2(\delta_{x_0},\delta_{x_1})=2\kappa^2(1-\Cos(|x_0-x_1|/\kappa))$, and plugging this into \eqref{eq:SHK} yields $\SHK_\kappa(\delta_{x_0},\delta_{x_1})=|x_0-x_1|$.
With this and \eqref{eq:HKDiracLog} one then obtains from \eqref{eq:SHKLog}
\begin{align*}
v^S_0(x_0)&=x_1-x_0,
&
\alpha^S_0(x_0)&=0,
&
(\mu_1^S)^\perp & =0,
\end{align*}
where we used that $s'(0)=|x_0-x_1|/(\kappa \sin(|x_0-x_1|/\kappa))$ in \eqref{eq:SHKLogSPrime}.
So in this case, the $SHK$ and $W_2$ metrics yield the same embedding of the samples onto a straight line in tangent space. Consequently, the mean of the embeddings will also lie on this line and there will only be one non-zero eigenvalue in the PCA.

Comparing the $\HK$ and $\SHK$ embeddings, we see that the mass changes in $\HK$ lead to an embedding as a circle segment, such that the mean of the embeddings in tangent space corresponds to a measure with mass less than one, and to the appearance of an extra non-zero eigenvalue in PCA.
By analysis of \eqref{eq:HKDirac}, \eqref{eq:HKDiracLog} one can verify that the corresponding additional eigenvector corresponds to pure mass changes that are not present in the samples.
In this sense, $\SHK$ yields a more faithful embedding.
We will observe the appearance of this additional eigenvalue numerically in the following experiments.

\paragraph{Small disks (with fixed radius) along a line}
\begin{figure}[hbt]
	\centering
	\subcaptionbox{Generated data,\newline reference measure supported on black disk\label{fig:fixRad1D_Samples}}%
	{\includegraphics[scale=.5]{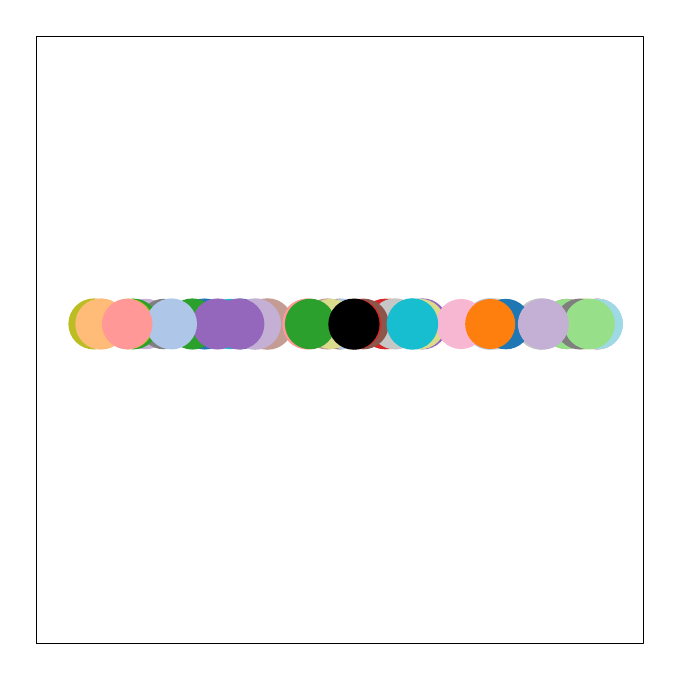}}%
	~\vline~
	\subcaptionbox{Ratios of explained variance for first 4 PCA modes (log scale) and 2-dimensional PCA projection for $\HK$ and $\SHK$ embeddings.\label{fig:fixRad1D_HKPCA_spec_1v1}}%
	{\includegraphics[scale=.3]{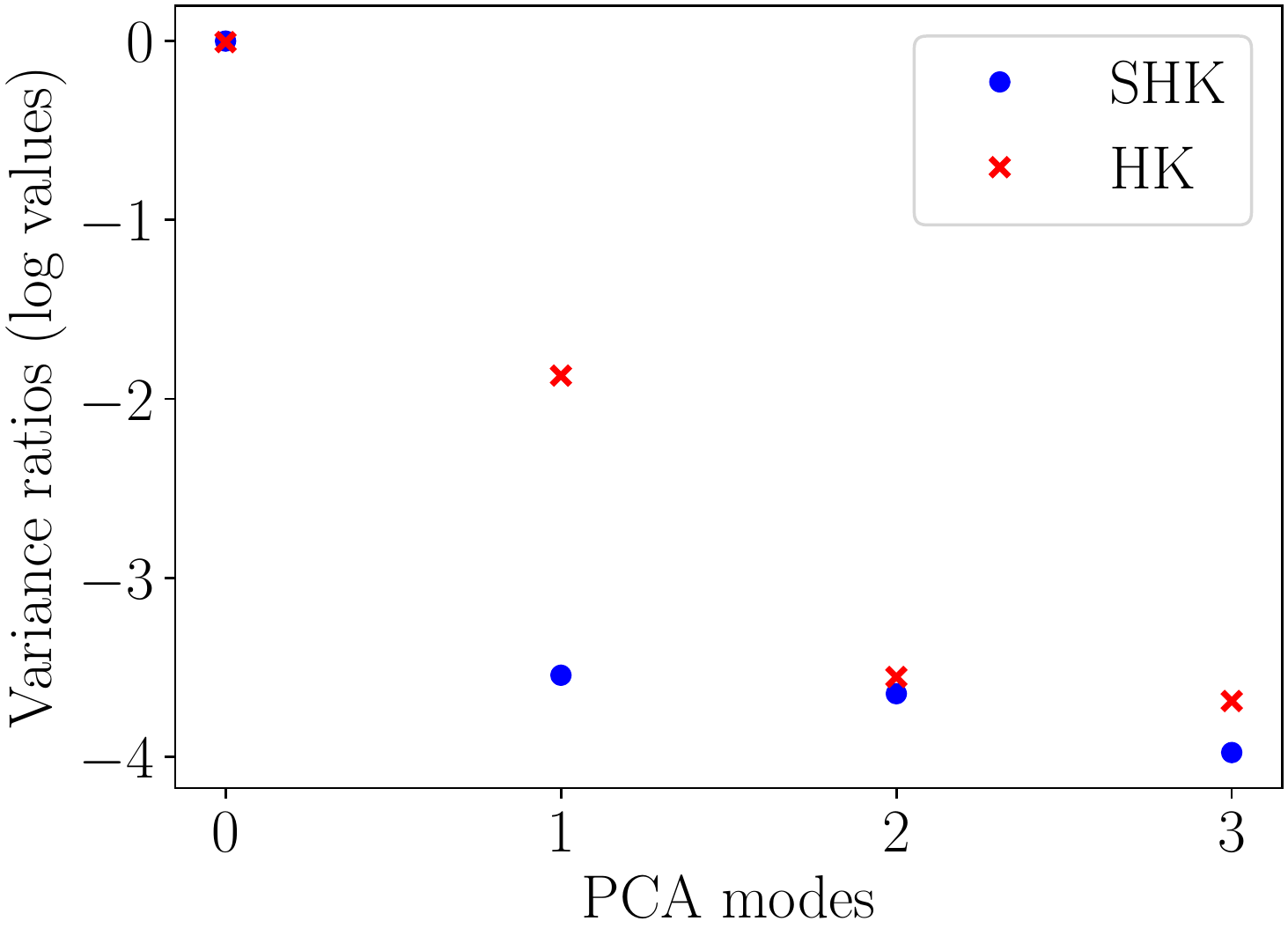}\quad\includegraphics[scale=.3]{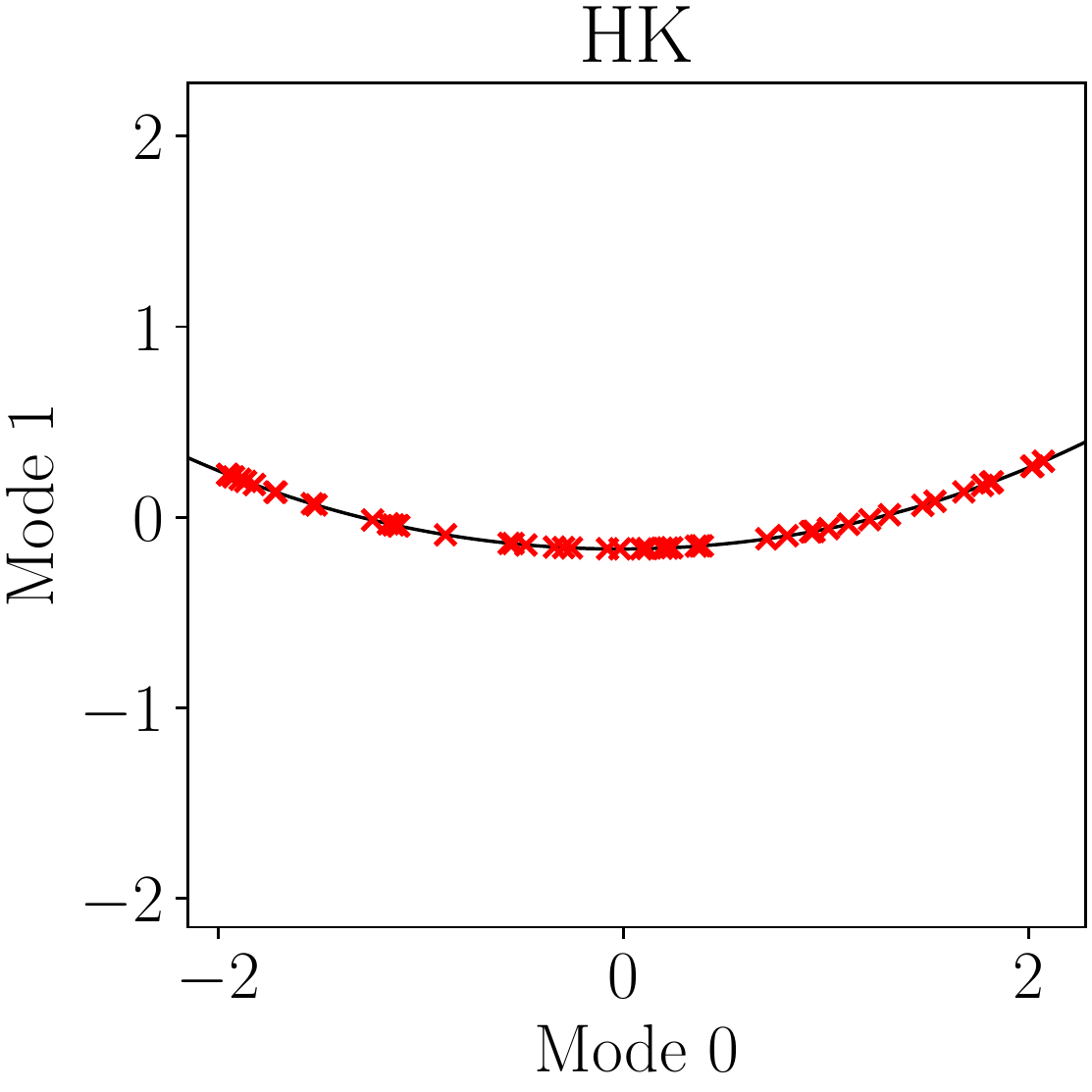}\quad\includegraphics[scale=.3]{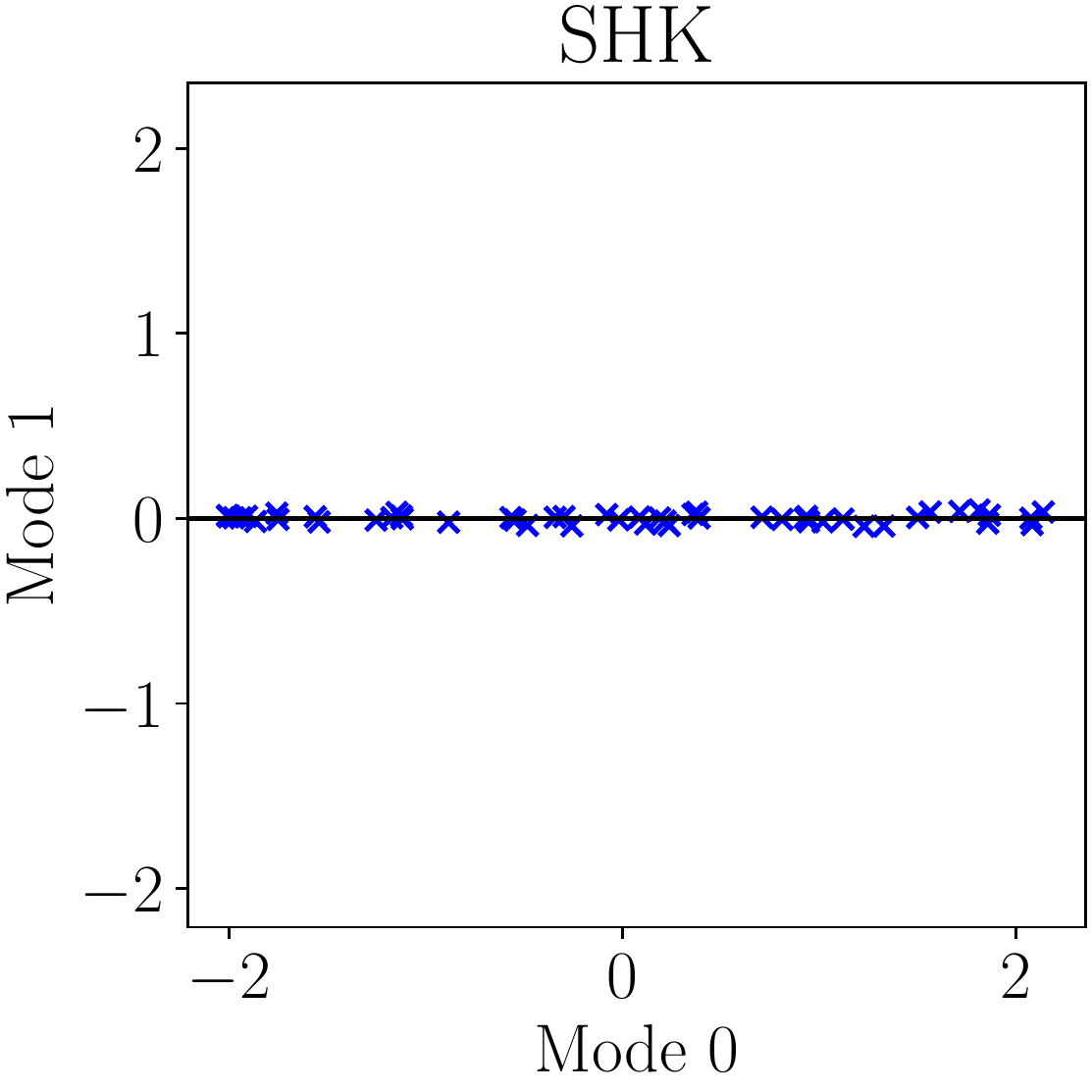}}
	
	\caption{Principal component analysis for the $\HK$ and $\SHK$ embeddings of the 1-dimensional disk data. In \subref{fig:fixRad1D_HKPCA_spec_1v1} one can observe the circular structure for $\HK$ and the appearance of a corresponding additional `large' eigenvalue.}
	\label{fig:fixRad1D_PCA}
\end{figure}
\begin{figure}[hbt]
	\centering
	\subcaptionbox{Exponential map along first two PCA modes for $\HK$.\label{fig:fixRad1D_HKPCA_shooting}}%
	{\includegraphics[scale=.3]{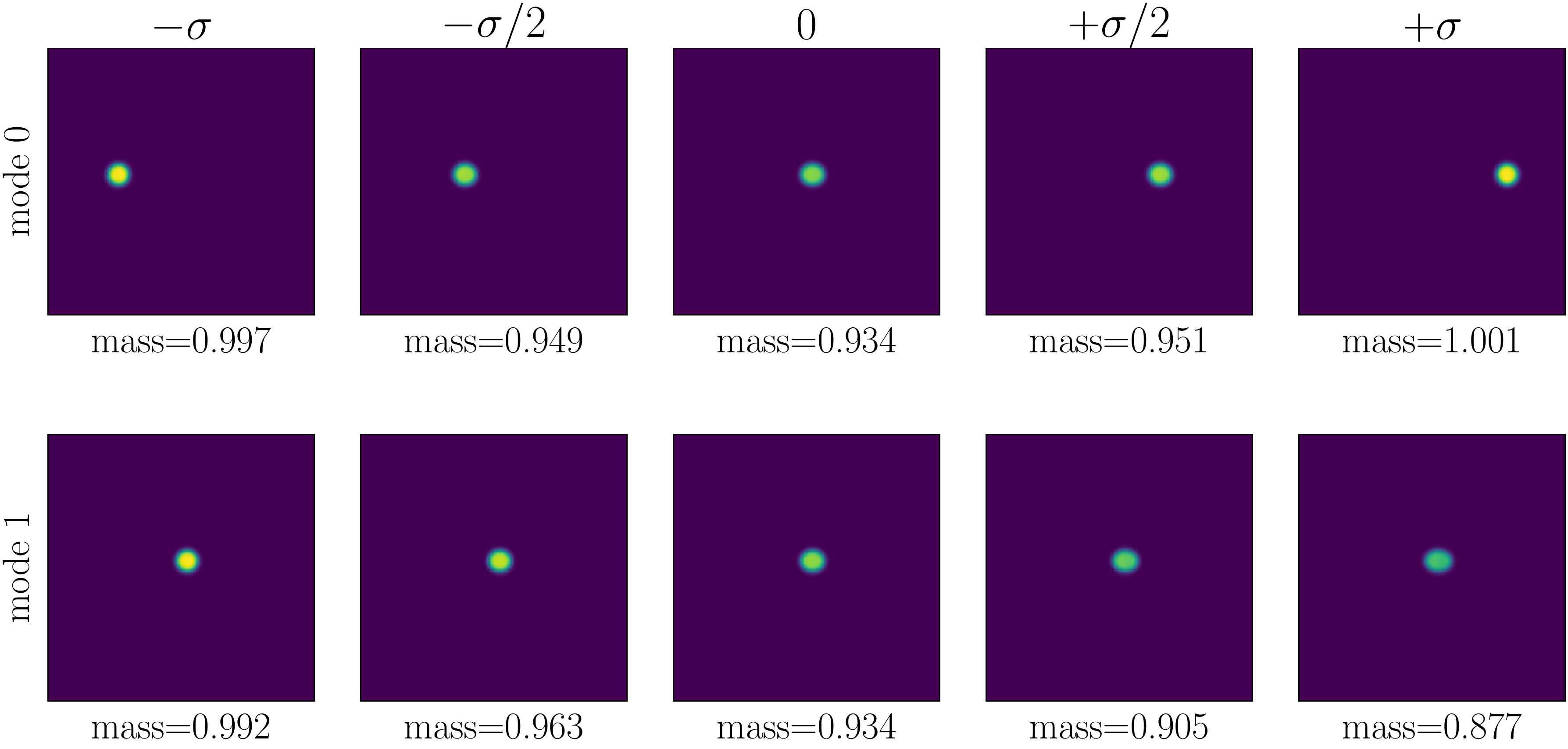}}
	
	\subcaptionbox{Exponential map along first PCA mode for SHK.\label{fig:fixRad1D_SHKPCA_shooting}}%
	{\includegraphics[scale=.3]{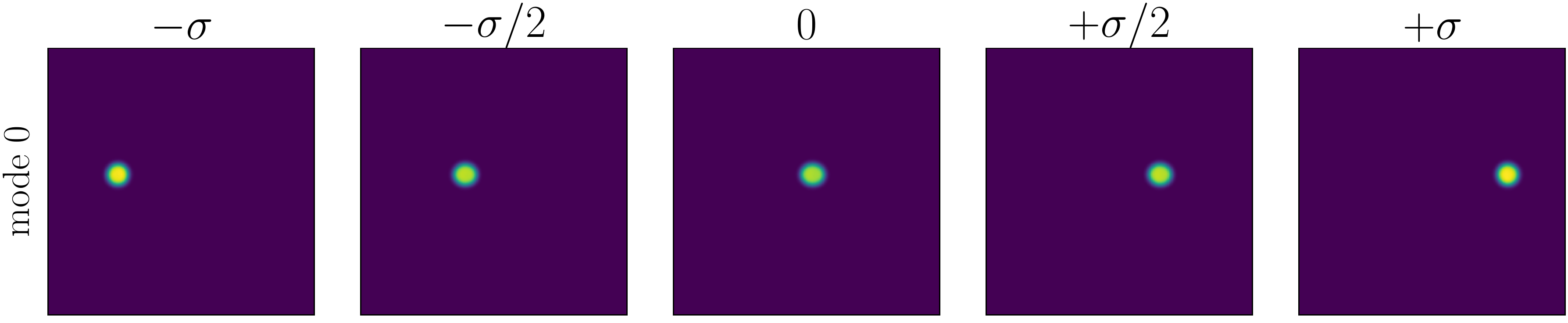}}
	\caption{Visualizing the dominant PCA eigenvectors via the exponential map for $\HK$ and $\SHK$ on the 1-dimensional disk data. In both cases the dominant eigenvector corresponds to translation of the disk, albeit with an additional mass change for $\HK$. To compensate for this, $\HK$ has a second large eigenvalue, corresponding to pure mass changes (eigenvector shown in second row).}
	\label{fig:fixRad1D_shooting}
\end{figure}
Next, we consider uniform probability measures on disks of radius $R\assign0.2$ with centers sampled at random uniformly from the segment $[R; L-R] \times \{0\}$ for $L\assign5$ and set $\kappa\assign6$. The reference measure $\mu_0$ is taken to be the disk centered in the middle of the segment (see \Cref{fig:fixRad1D_Samples}). For these choices of $R,L$ and $\kappa$, the pure Hellinger parts in both $\HK$ and $\SHK$ metrics are zero.

The stability results from \Cref{prop:CV_tanvecs_HK} and \Cref{cor:HK_strong_L2} suggest that the embedded data will be close to the Dirac example from the previous paragraph.
Indeed, the PCA spectra and two-dimensional embeddings for $\HK$ and $\SHK$, shown in \Cref{fig:fixRad1D_HKPCA_spec_1v1}, confirm this. The two-dimensional PCA embedding for $\HK$ shows a circle segment, for $\SHK$ we see a straight line.
For both metrics the dominant eigenvector corresponds to the translations along the chosen horizontal axis (cf.~\Cref{fig:fixRad1D_HKPCA_shooting} and \Cref{fig:fixRad1D_SHKPCA_shooting} where this eigenvector is visualized via the exponential map as described in \Cref{sec:NumericsPrel}). For $\HK$ this translation is combined with a change in mass, as discussed in the previous paragraph. Accordingly, in the PCA spectrum for $\HK$ the second largest eigenvalue is substantially larger than for $\SHK$ and it corresponds to the direction of pure mass changes (see \Cref{fig:fixRad1D_HKPCA_shooting} again).

\paragraph{Disks with random radii and positions in a box}
\begin{figure}[hbtp]
	\centering
	\subcaptionbox{Generated data,\newline reference measure supported on black disk.\label{fig:varRad2D_Samples}}{\includegraphics[scale=.6]{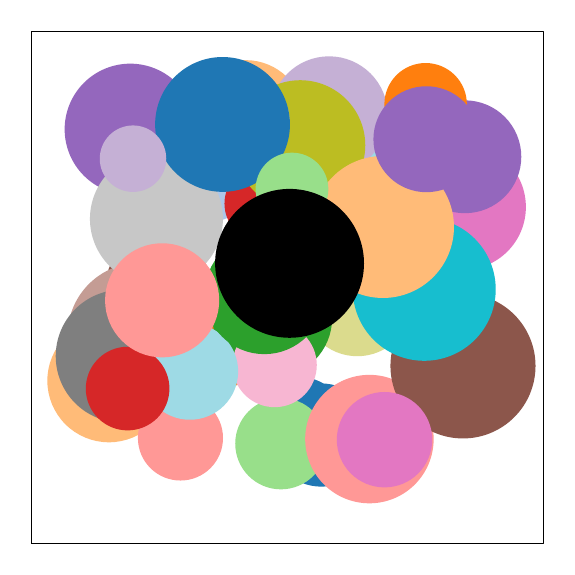}}~\vline~
	\subcaptionbox{Ratios of explained variance for first PCA modes for $\HK$ and $\SHK$ (log scale), 2d PCA embedding for $\HK$ (the 2d $\SHK$ embedding looks virtually identical), and scatter plot of the sample disk centers. \label{fig:varRad2D_HKPCA_spec}}{\includegraphics[scale=.3]{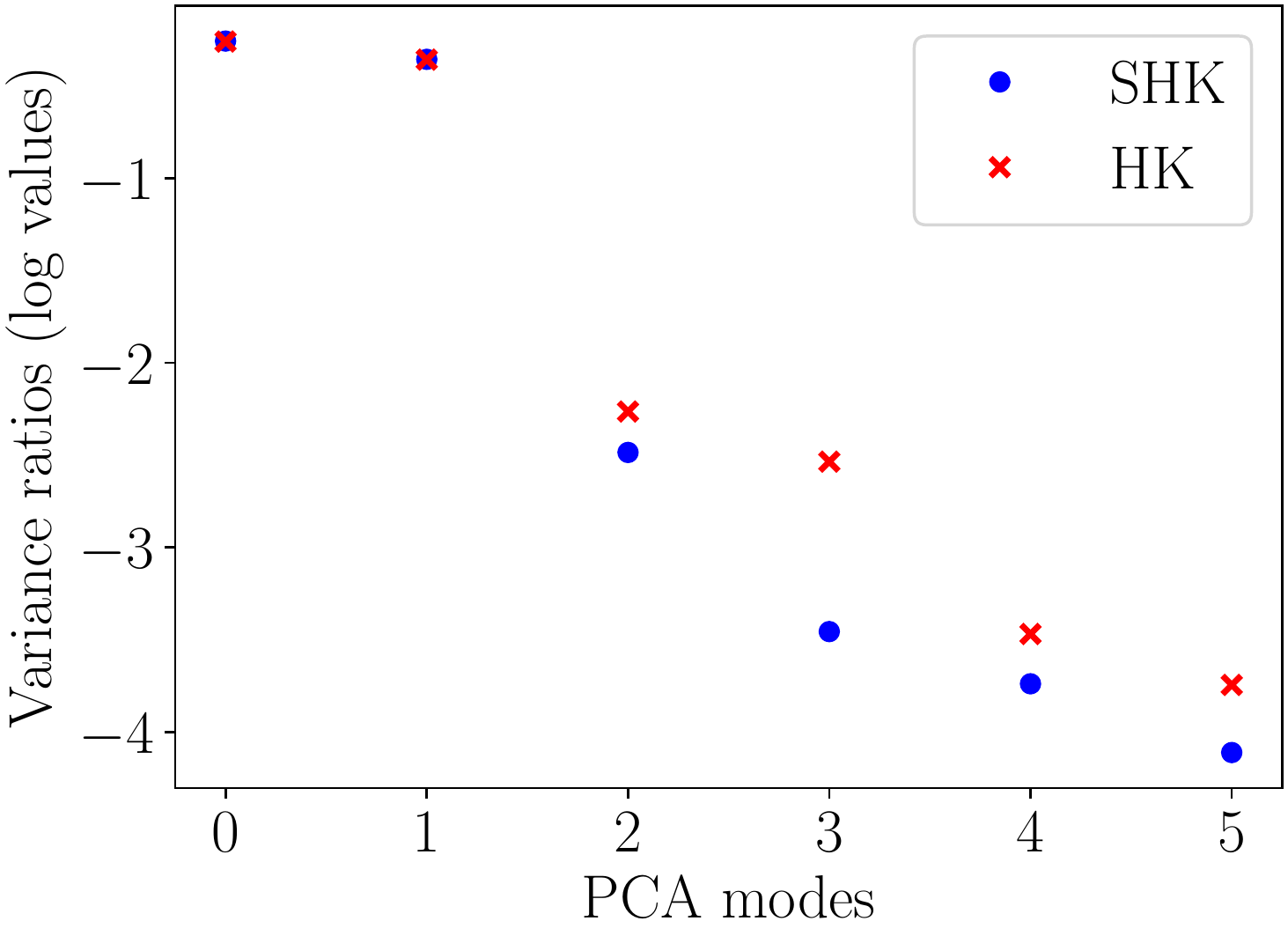}~\includegraphics[scale=.2]{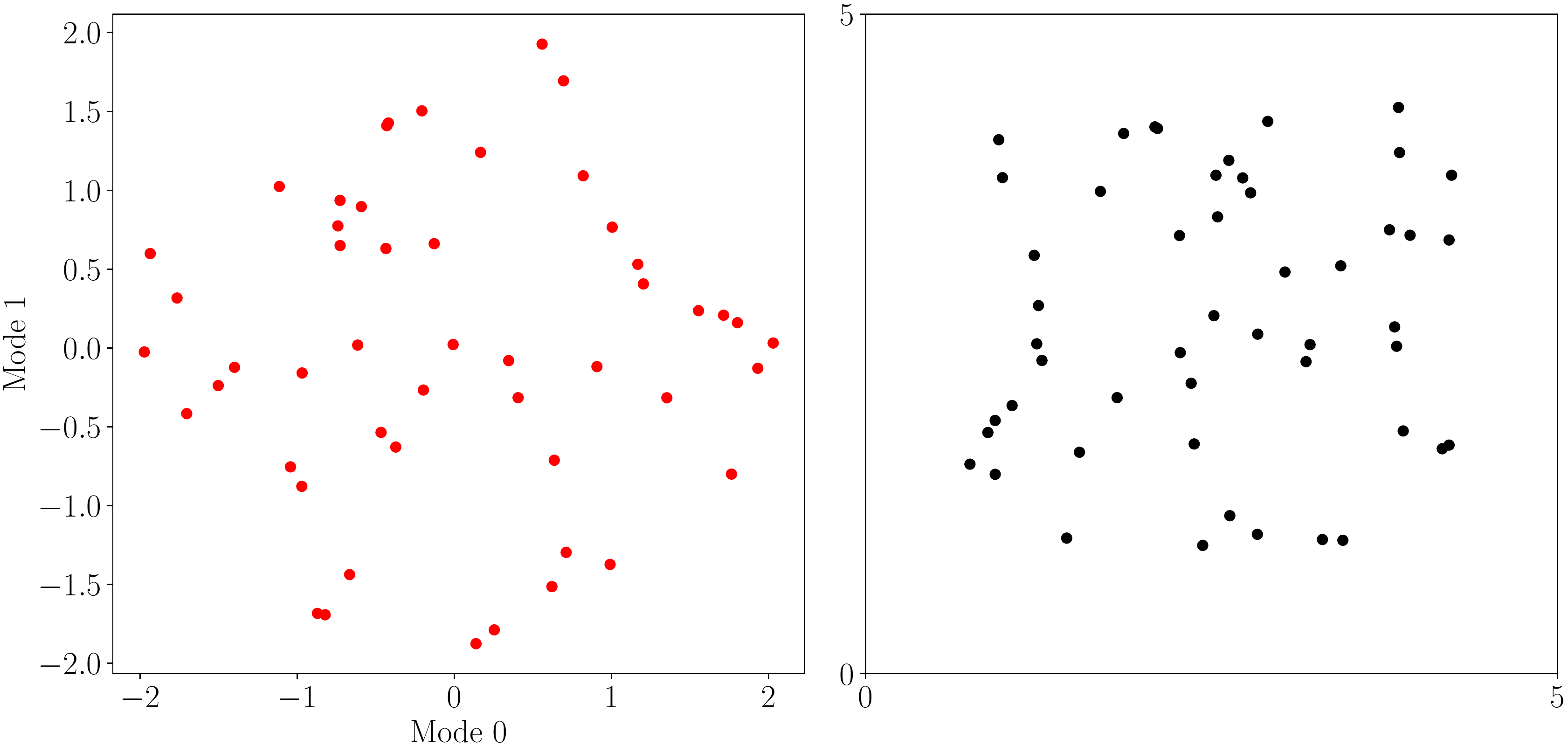}}
	\caption{Principal component analysis for the $\HK$ and $\SHK$ embeddings of the 3-dimensional disk data. Up to reflection and rotation the 2d PCA projection and the scatter plot of the sample disk centers look almost identical suggesting that the first two modes capture the translation of the samples.}
	\label{fig:varRad2D_PCA}
\end{figure}

\begin{figure}[hbtp]
	\centering
	\subcaptionbox{Exponential map along first four PCA modes for $\HK$.\label{fig:varRad2D_HKPCA_shooting}}%
	{\includegraphics[scale=.25]{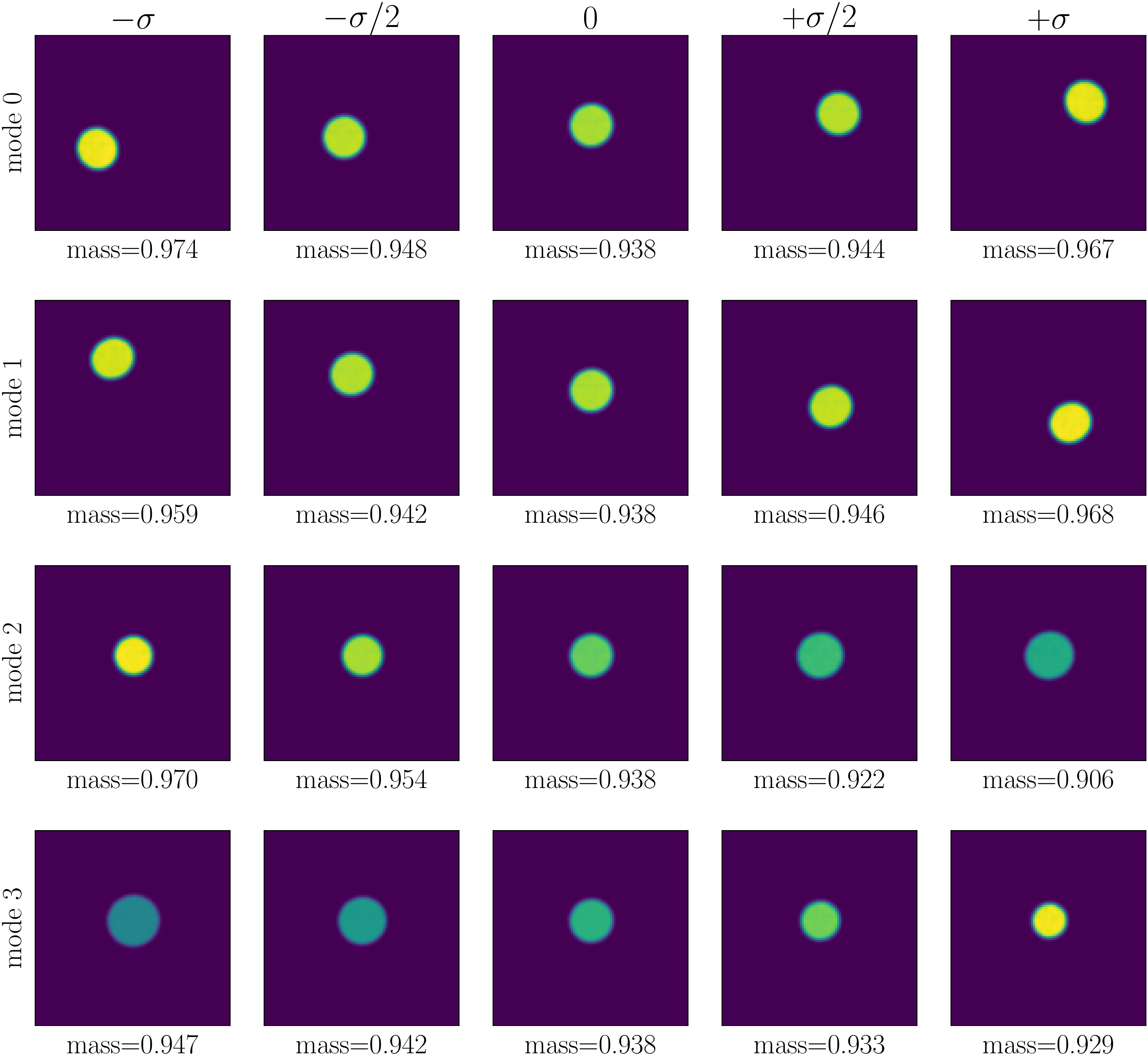}}
	
	\subcaptionbox{Exponential map along first three PCA modes for SHK.\label{fig:varRad2D_SHKPCA_shooting}}%
	{\includegraphics[scale=.25]{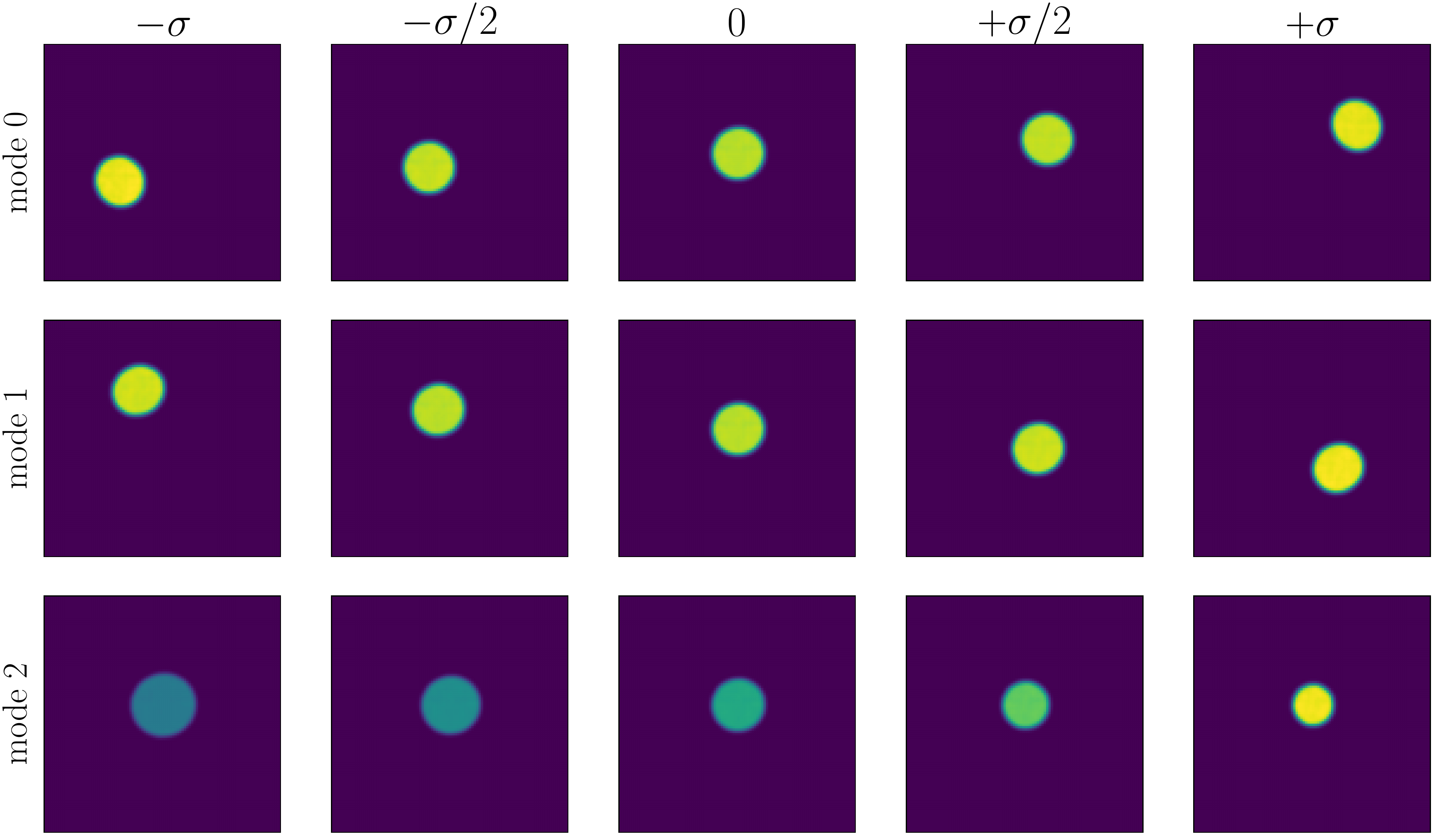}}
	\caption{Visualizing the dominant PCA eigenvectors via the exponential map for $\HK$ and $\SHK$ on the 2-dimensional disk data.}
	\label{fig:varRad2D_shooting}
\end{figure}

In the next experiment the data consists of uniform probability measures on disks of varying radii and centers, with radii and centers sampled uniformly from the interval $[R_{\min};R_{\max}]\assign[0.3;0.7])$ and $[R_{\max}, L-R_{\max}]^2$ respectively, where $L\assign 5$ and we set again $\kappa\assign6$.
The reference measure for the tangent space approximation corresponds to the disk of radius $R=0.5$, centered at $(L/2,L/2)$, see for instance \Cref{fig:varRad2D_Samples}.
Applying PCA to the tangent space embeddings of these measures for both HK and SHK yields the spectra and projection shown in \Cref{fig:varRad2D_HKPCA_spec}. Shooting along the modes via the exponential map is shown in \Cref{fig:varRad2D_shooting}.
From these Figures we conclude that for $\HK$ and $\SHK$ the first two modes encode essentially the positions of the centers of the disk (up to an exchange, rescaling  and small rotation of the axes, compare with the right-most Plot of \Cref{fig:varRad2D_HKPCA_spec}).
Moreover, mode 2 in SHK captures the radius variations. However, for $\HK$ mode 2 corresponds instead mostly to the mass bias, as in the previous experiments, whereas the radius variations are captured by mode 3.

\subsection{Linearized HK and SHK metrics on the sphere}

\begin{figure}[bt]
	\centering
	\subcaptionbox{Subset of generated data, reference measure in Black supported on the `north pole cap'.\label{fig:rotCirc3D_Samples}}%
	{\hspace{1em}\includegraphics[scale=.4]{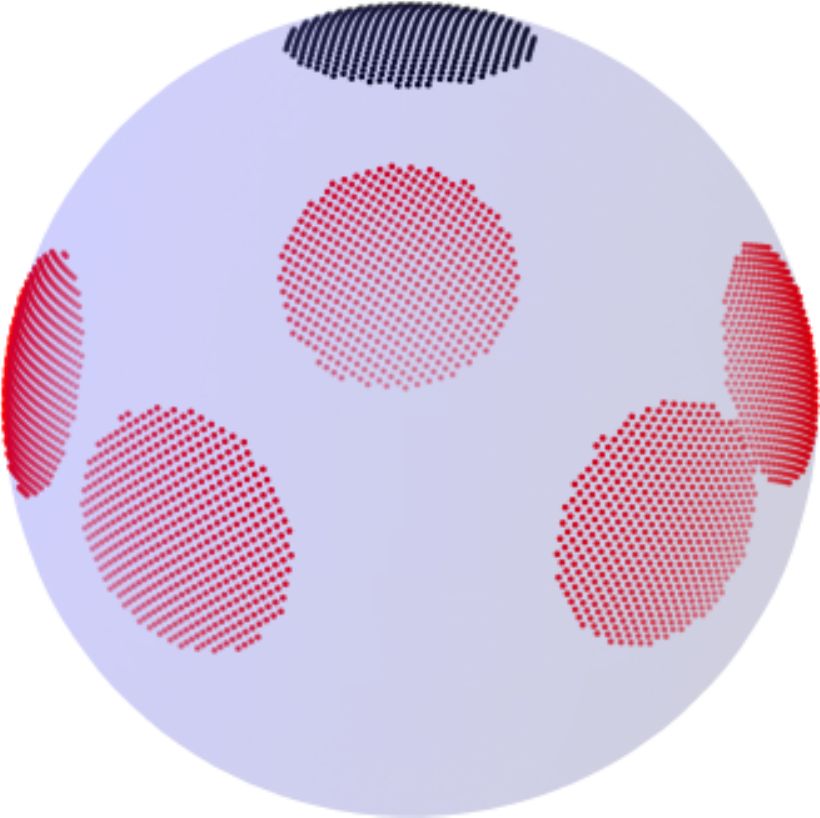}}%
	\qquad\vline\qquad
	\subcaptionbox{Ratios of explained variance for first PCA modes (log scale), and 2d PCA projection (for SHK embedding, $\HK$ embedding yields the same plot).\label{fig:rotCirc3D_HKPCA_spec}}%
	{\includegraphics[scale=.3]{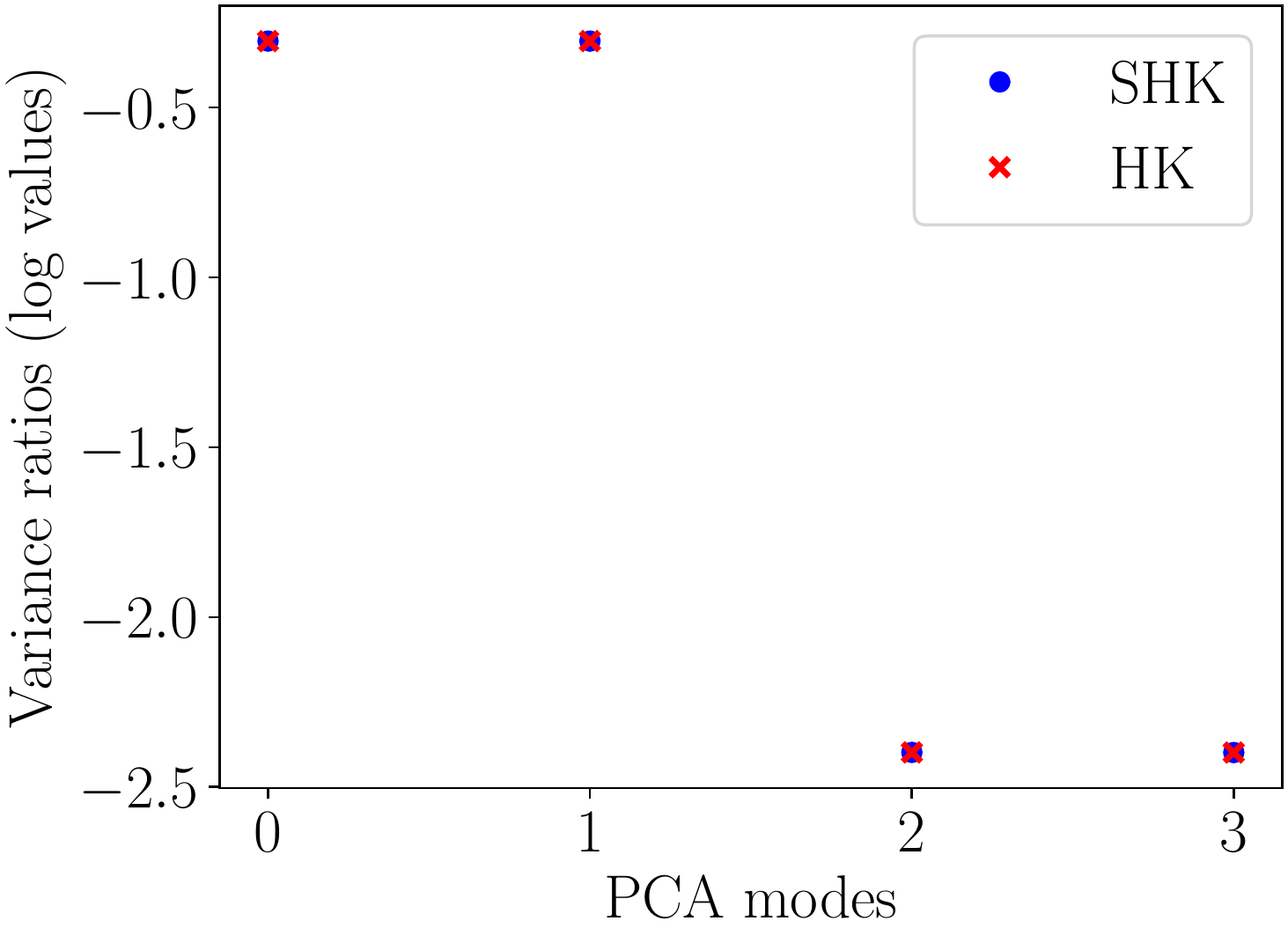}\quad \includegraphics[scale=.3]{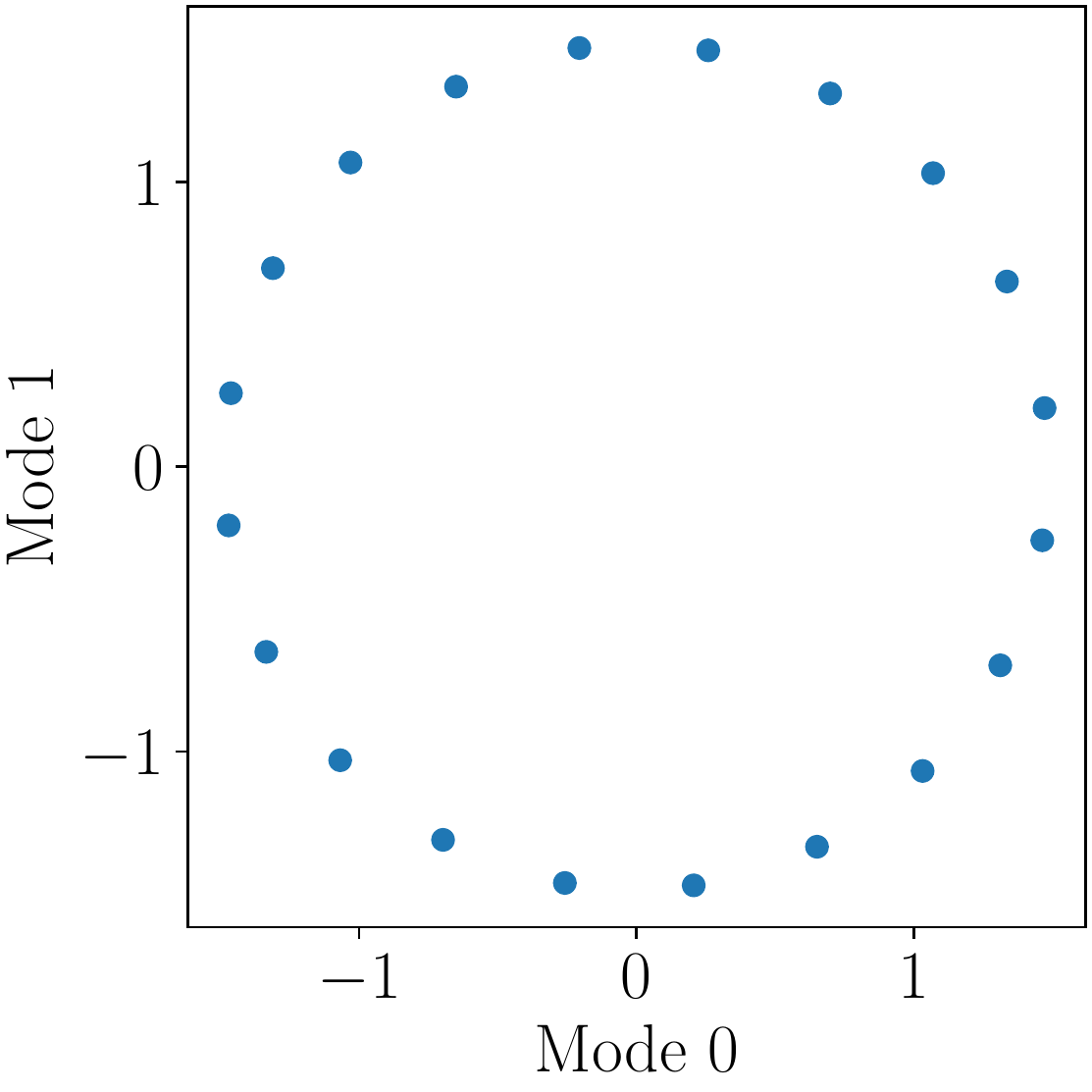}}
	\caption{Principal component analysis for the $\HK$ and $\SHK$ embeddings of the sphere example.\label{fig:rotCirc3D_PCA}}
\end{figure}
\begin{figure}[bt]
\centering
	\subcaptionbox{Exponential map along first two PCA modes for $\HK$.}%
	{\includegraphics[scale=.2]{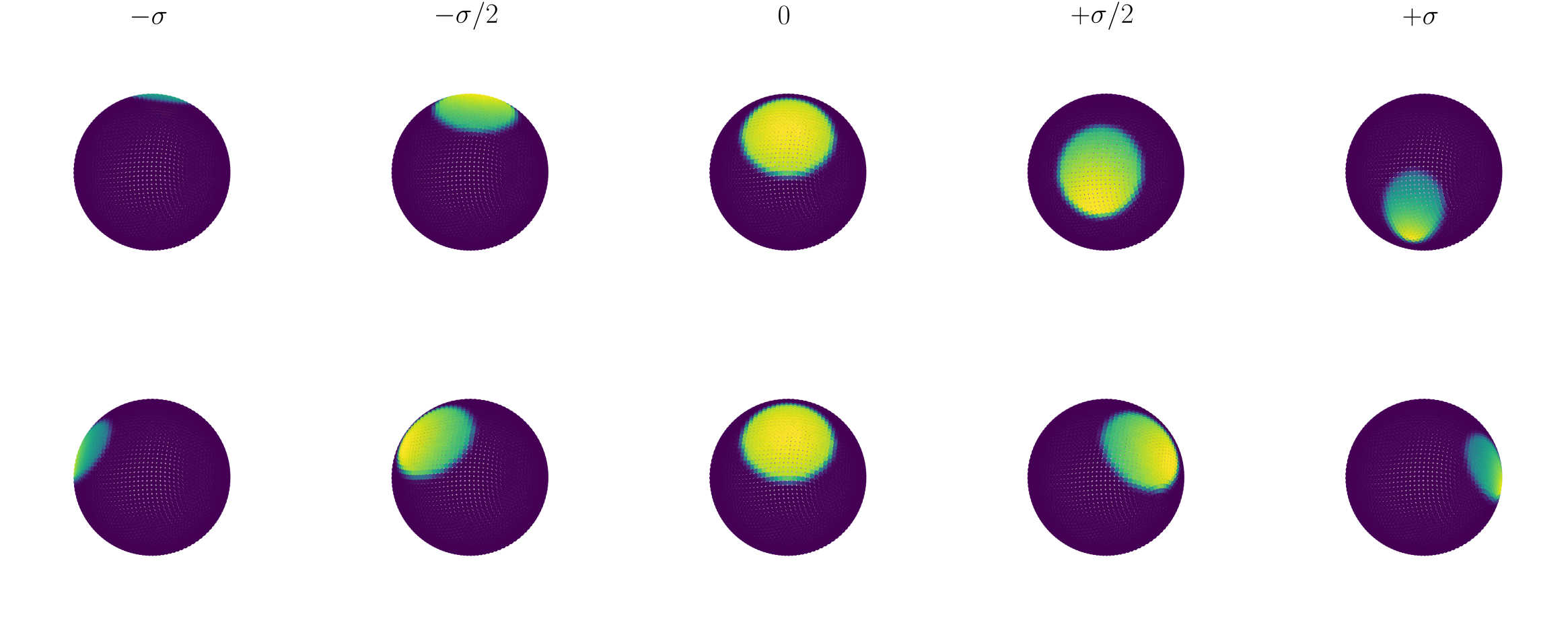}}
	~
	\subcaptionbox{Mass variations during the shooting for each mode}%
	{\;\;\includegraphics[scale=.22]{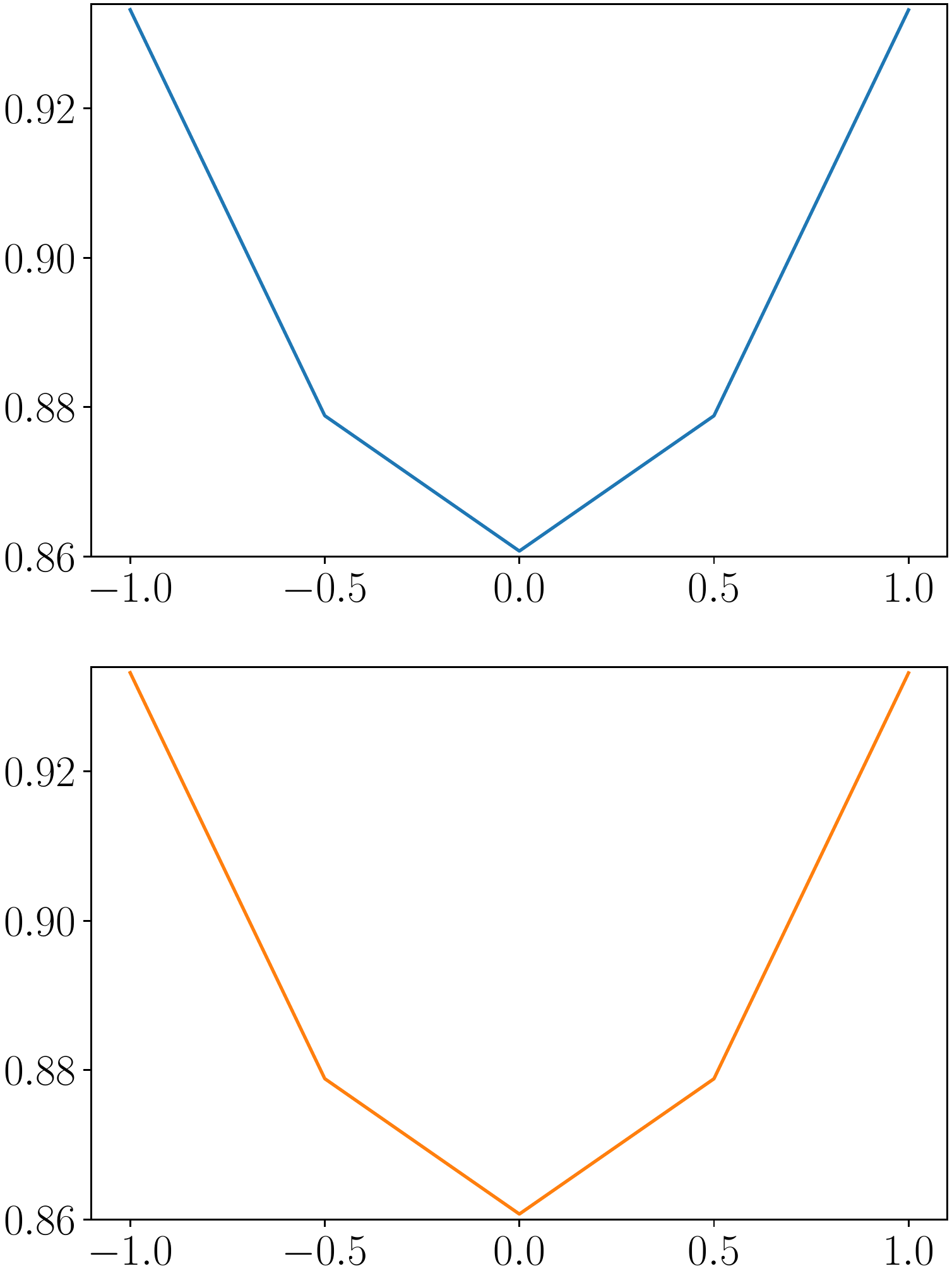}\;\;}
	\caption{Visualizing the dominant PCA eigenvectors via the exponential map for $\HK$ for the sphere example.}
	\label{fig:rotCirc3D_shooting}
\end{figure}

In this section we give an example for linearized optimal transport when the base space $X$ is the 2-dimensional unit sphere. We set $\kappa=4$.
As samples we pick uniform probability measures on balls (for the sphere metric) centered on points along the equator, whereas the reference measure is a ball centered on the north pole, see \Cref{fig:rotCirc3D_Samples}.
In this experiment the first four PCA eigenvalues for $\HK$ and $\SHK$ are essentially identical and
the first two components capture approximately the position of the samples from the perspective of the north pole, resulting in a circular structure, see \Cref{fig:rotCirc3D_HKPCA_spec}.
Exponential shooting confirms this interpretation (see \cref{fig:rotCirc3D_shooting}) showing the reference measure being pushed along two orthogonal directions relative to the north pole.
This results in a strong deformation of the balls which has to be corrected by the two subsequent modes, albeit carrying a much lower variance.

\subsection{Convexity of the range of the log map}
\label{sec:convexity_range_num}
We now revisit the question of convexity of the range of the logarithmic map in linearized optimal transport of Section \ref{sec:convexity_range}.

\begin{figure}[hbt]
	\centering
	\includegraphics[scale=.4]{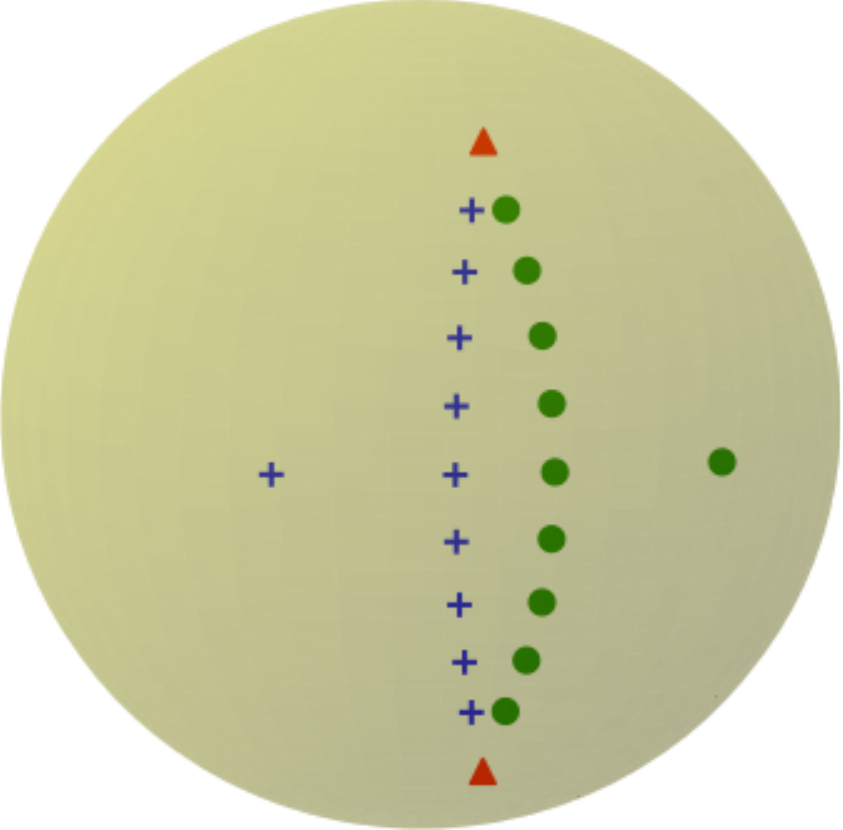}\hskip4cm
	\includegraphics[scale=.4]{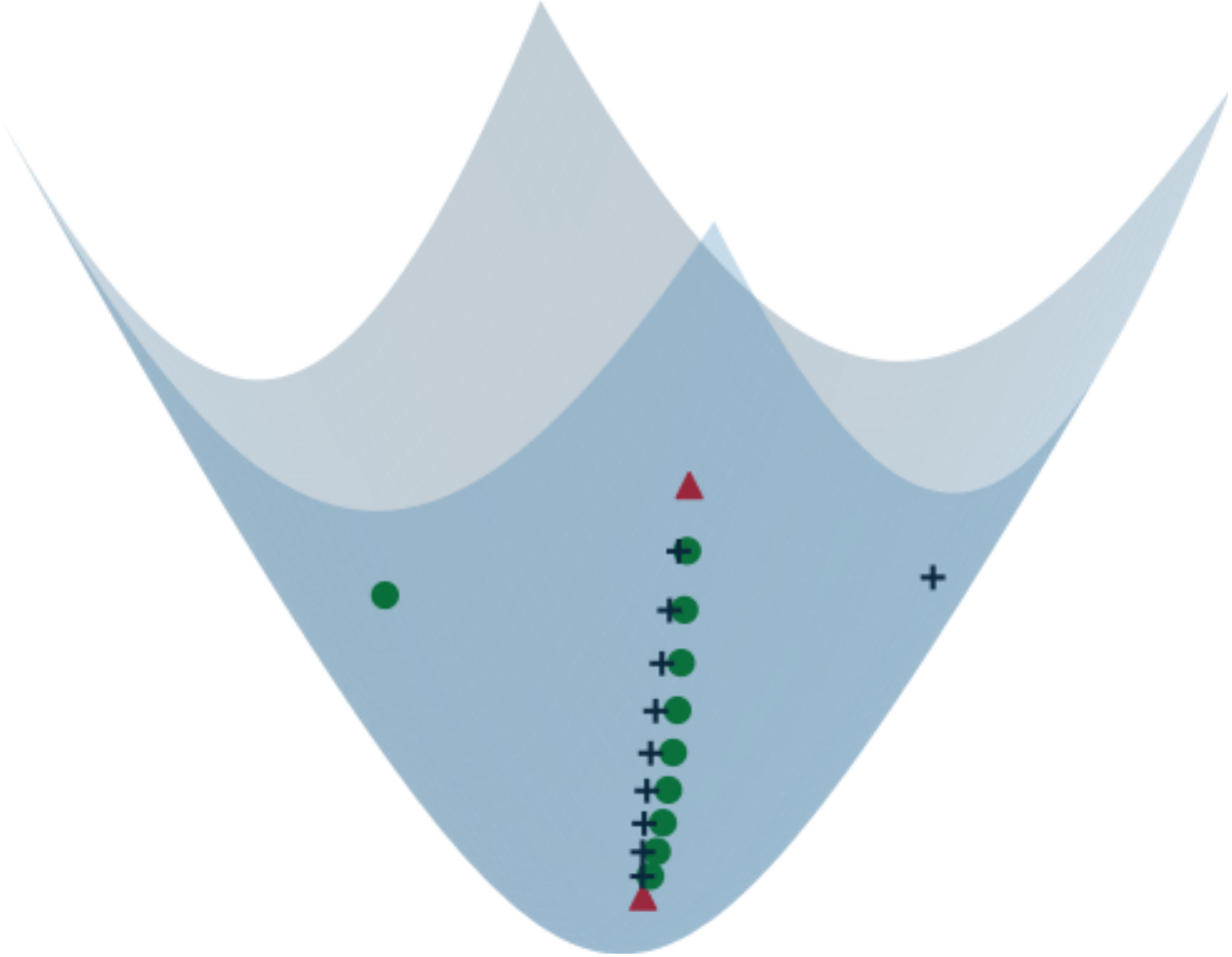}
	\caption{Interpolation between two discrete measures by mapping to tangent space, linear interpolation, and application of the exponential map for $\W_2$ on a sphere (left), and the paraboloid model of the hyperbolic space (right).
		In each case, $\nu_0\assign\delta_{y_0}$ and $\nu_1\assign\delta_{y_1}$ are indicated by red triangles. $\mu\assign\tfrac12 (\delta_{x_0}+\delta_{x_1})$ is indicated by the isolated cross ($i=0$) and dot ($i=1$). The interpolated measures $\nu_t$ (for some $t \in (0;1)$) are indicated by $T_t(x_i)$ (see text) by crosses ($i=0$) and dots ($i=1$) respectively, indicating from which $x_i$ the corresponding mass was transported.
		One can see that in the case of the sphere, the interpolated transport is optimal between the initial measure and the target one, \eqref{eq:LogdTCondition} is satisfied. For the hyperbolic space, it is not as the order of the Dirac masses (from left to right) is reversed, thus \eqref{eq:LogdTCondition} is violated.
		\label{fig:W2_examples}}
\end{figure}

\paragraph{$\W_2$ distance}
We chose as base measure $\mu\assign\frac{1}{2}\delta_{x_0}+\frac{1}{2}\delta_{x_1}$ and two sample measures $\nu_0\assign\delta_{y_0}$ and $\nu_1\assign\delta_{y_1}$, for four points $x_0,x_1,y_0,y_1 \in X$.
Then for $s \in \{0,1\}$ there are optimal transport maps, simply given by $T_s(x_i)=y_s$ for $i \in \{0,1\}$ and therefore for $v_s \assign \Log_{\mu}^{\W_2}(\nu_s)$, we find $v_s(x_i)=\Log_{x_i}^X(y_s)$.
For $t \in (0;1)$ then define $v_t \assign (1-t) \cdot v_0 + t \cdot v_1$ and $\nu_t \assign \Exp_{\mu}^{\W_2}(v_t)=T_{t\#} \mu$ where $T_t(x_i)=\Exp_{x_i}^X(v_t(x_i))$. $\nu_t$ takes the form $\nu_t = \tfrac12 (\delta_{T_t(x_0)}+\delta_{T_t(x_1)})$.
Thus the curve of measures obtained by logarithmic map, linear interpolation in the tangent space, and re-applying the exponential map consists of two Dirac masses.
If one has for all $t \in (0;1)$ that
\begin{equation}
\label{eq:LogdTCondition}
d(x_0,T_t(x_0))^2 + d(x_1,T_t(x_1))^2 \leq d(x_0,T_t(x_1))^2 + d(x_1,T_t(x_0))^2,
\end{equation}
then $T_t$ is indeed the optimal transport map between $\mu$ and $\nu_t$ and therefore $\Log_\mu^{\W_2}(\nu_t)=v_t$, and finally that the whole line segment between $v_0$ and $v_1$ in the tangent space of $\mu$ lies in the range of the logarithmic map.

For $X \subset \R^n$ one has $\Log_x^X(y)=y-x$ and $\Exp_x^X(v)=x+v$. In this case, $v_s(x_i)=y_s-x_i$ for $s \in \{0,1\}$, $v_t(x_i)=(1-t) \cdot y_0 + t \cdot y_1 - x_i$, and therefore $T_t(x_0)=T_t(x_1)=(1-t) \cdot y_0 + t \cdot y_1$. Thus, both Dirac masses travel on together on the straight line from $y_0$ to $y_1$, and \eqref{eq:LogdTCondition} is satisfied (since it becomes an equality).

For other manifolds, we can repeat this construction numerically and try to verify or falsify \eqref{eq:LogdTCondition} by visual inspection.
Figure \ref{fig:W2_examples} shows $x_i$, $y_s$, and $T_t(x_i)$ for highly symmetric examples on the unit sphere and for hyperbolic space (represented by its parabolo\"id model $Z=\sqrt{1+X^2+Y^2}$).
The sphere is known to satisfy the MTW(0,0) condition, thus the range of the logarithmic map is convex, and from the figure we can clearly see that \eqref{eq:LogdTCondition} is satisfied. 
Hyperbolic space has negative sectional curvature and therefore MTW(0,0) cannot be satisfied (this is a consequence of Loeper's identity for the Ma-Trudinger-Wang tensor \cite{Loeper2009}). Indeed, Figure \ref{fig:W2_examples} suggests that \eqref{eq:LogdTCondition} is violated.

\begin{figure}[hbt]
	\centering
	\includegraphics[scale=.4]{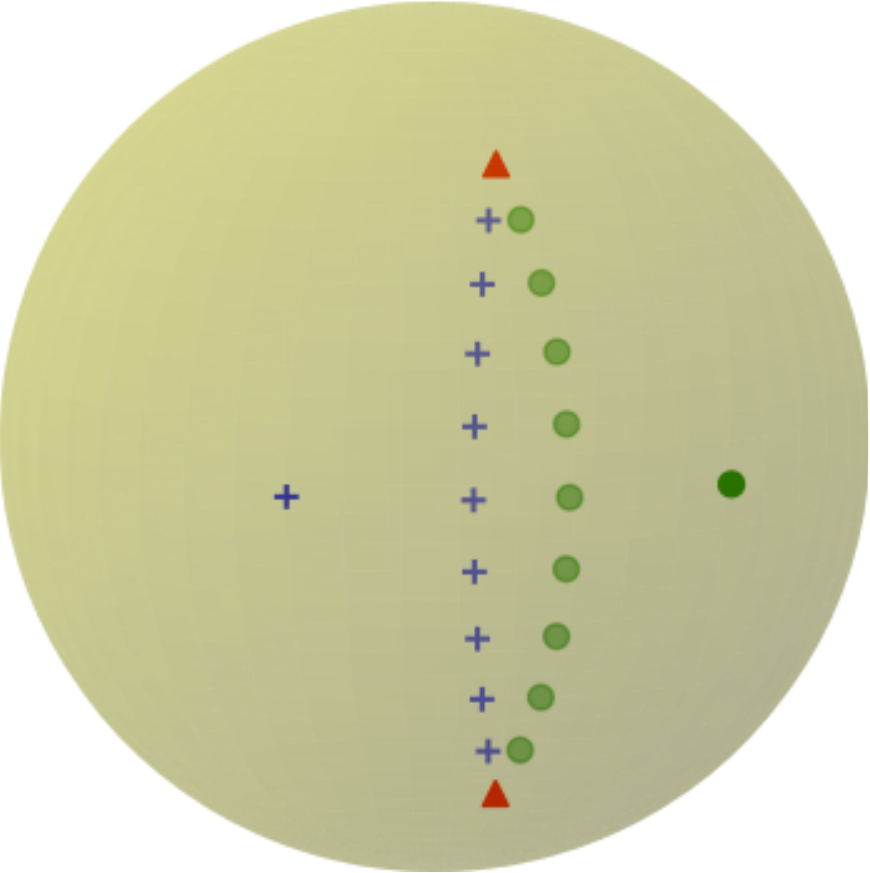}\hskip1cm
	\includegraphics[scale=.4]{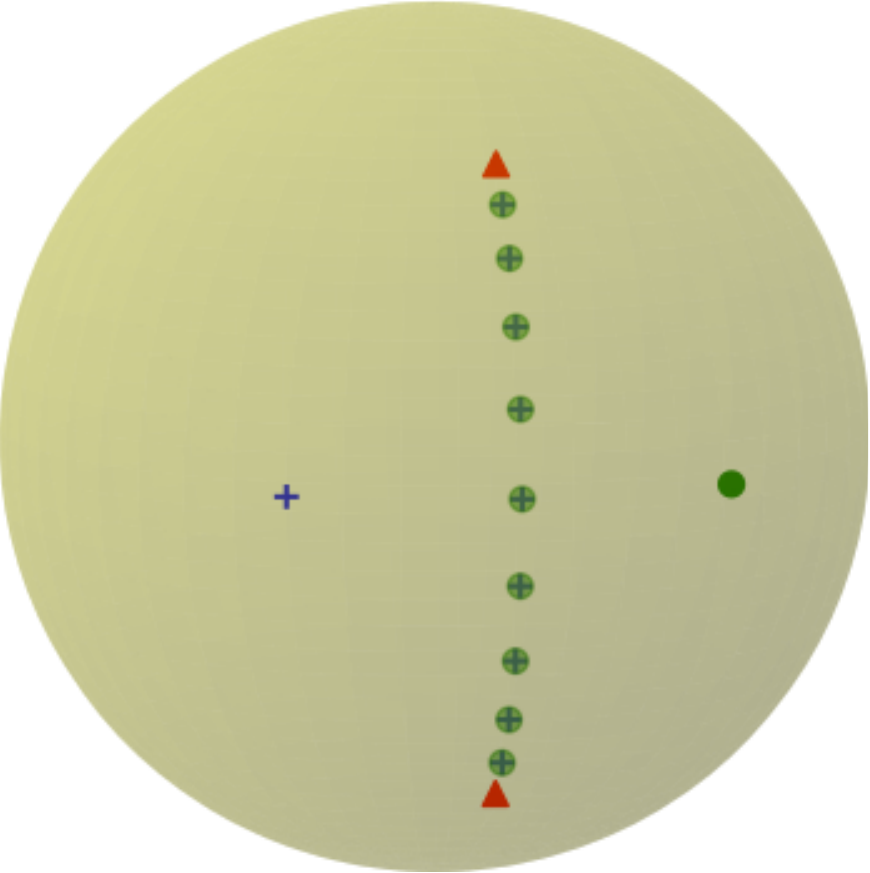}\hskip1cm
	\includegraphics[scale=.4]{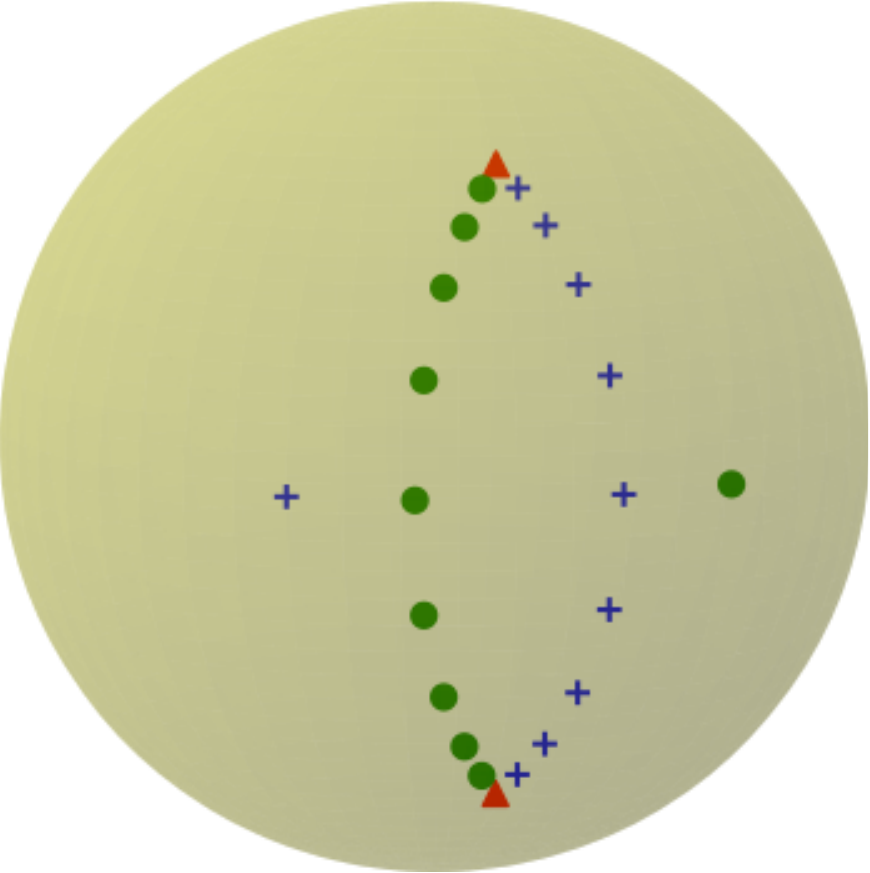}
	\caption{Same setup as in \Cref{fig:W2_examples} for $\HK_1$ on spheres with radii $r=0.1$,$1$ and $1.5$ (left to right). For $r\leq 1$, the convex combination of the logarithmic components lies in the range of the logarithmic map, for $r>1$ it does not.
		\label{fig:HK_examples}}
\end{figure}

\paragraph{HK and SHK distances.}
The above experiment can readily be adapted to the two unbalanced transport metrics.
Figure \ref{fig:HK_examples} visualizes the experiment for $\HK_1$ with $X$ being spheres of different radii.
For the unit sphere, convexity seems to hold (which can be related to the fact that the cone over the unit sphere is locally isometric to flat Euclidean space, see \cite[Section 7]{liero2018optimal} for more details). For radii less than one numerical examples suggest that convexity is also satisfied. For radii strictly greater than 1 convexity seem to fail. The same holds for flat space (not shown in Figure).
These observations seem to be consistent with the results on $c^{\HK}_\kappa$ for the sphere obtained in \cite{gallouet2021regularity}, albeit the precise relation is as of now unclear due to the aforementioned non-linear change of variables (\Cref{rem:logcconcave}).

These cases of non-convexity are not observed for the SHK metric which exhibits a behaviour much more similar to that of the $W_2$ distance (see \Cref{fig:SHK_examples}).

\begin{figure}[hbt]
	\centering
	\raisebox{0.25\height}{\includegraphics[scale=.4]{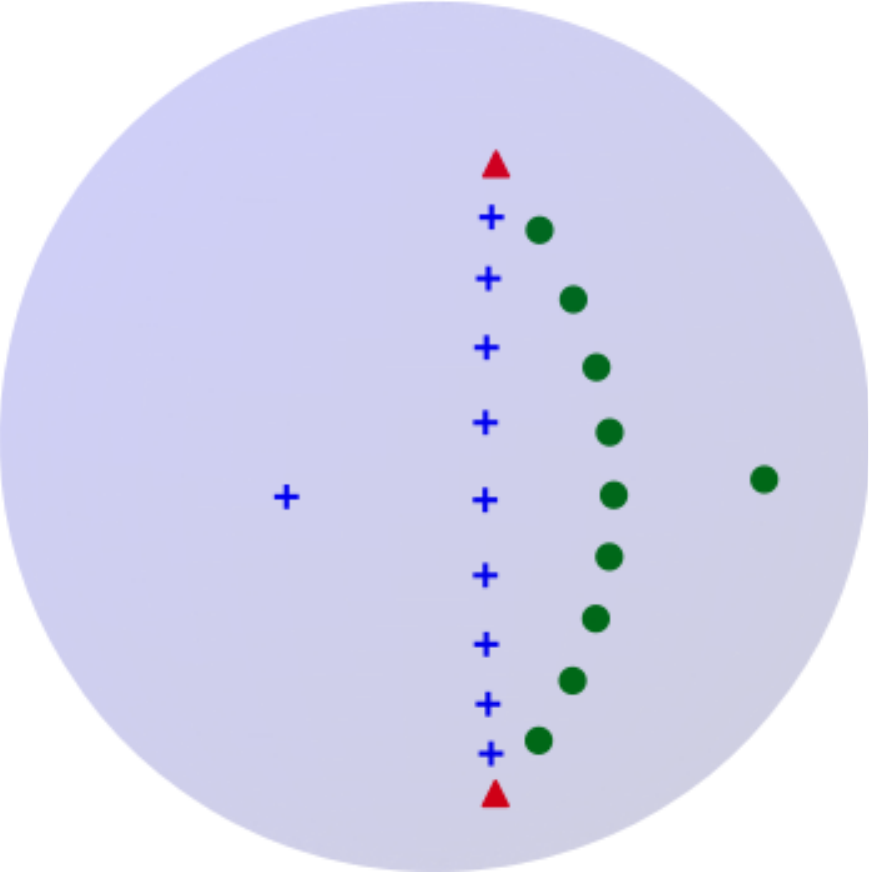}}
	\hskip1.5cm
	\includegraphics[scale=.25]{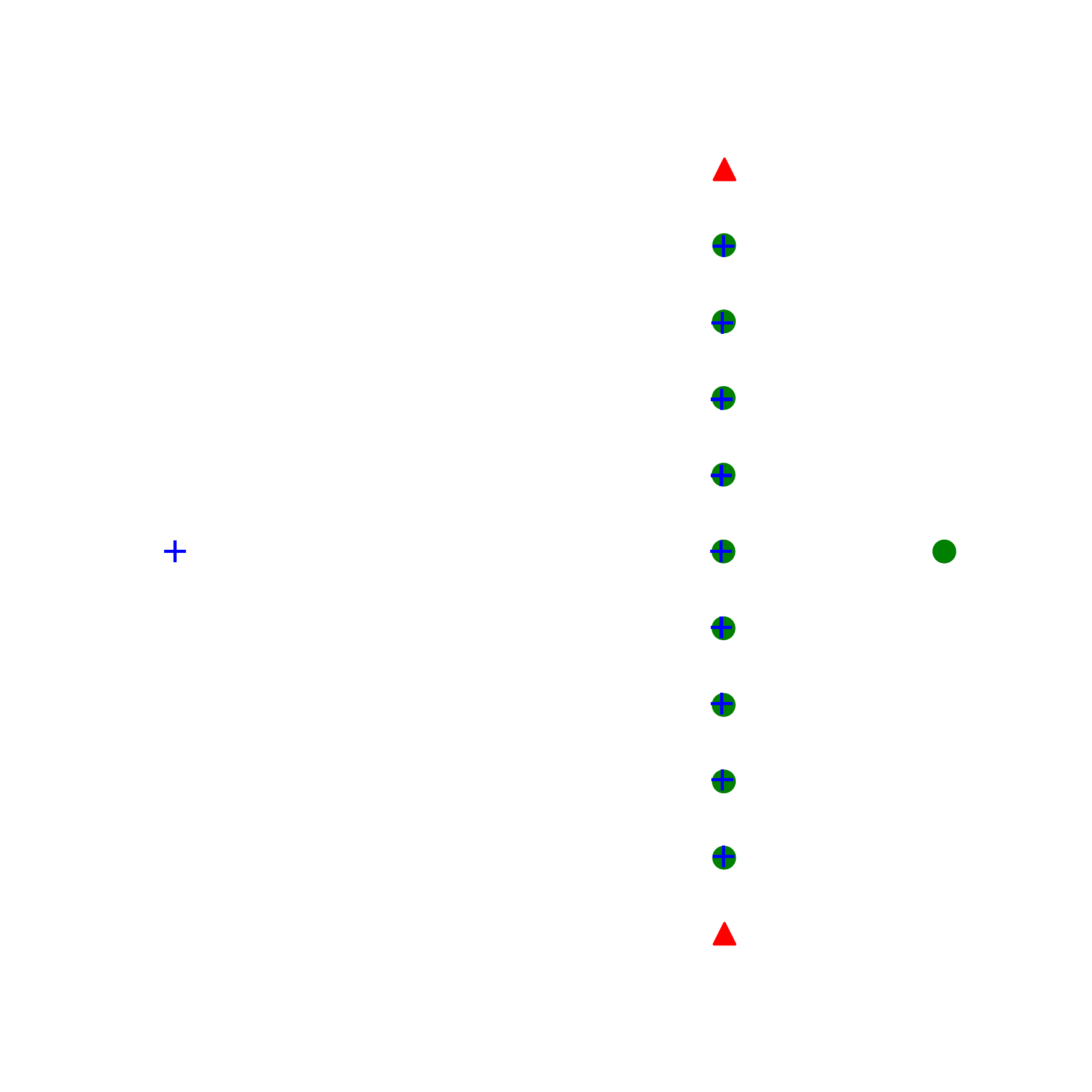}
	\caption{Same setup as in \Cref{fig:W2_examples,fig:HK_examples} for $\SHK_1$ on the sphere of radius 1.5 (left) and flat $\R^2$ (right). In both cases convexity of the range of the log map appears to be satisfied, with the interpolating measures being given by single Dirac masses in the flat case.
		\label{fig:SHK_examples}}
\end{figure}

\paragraph{Acknowledgements} The authors gratefully acknowledge support from the DFG CRC 1456 Mathematics of the Experiment A03.

\bibliographystyle{siamplain}
\bibliography{bib_HK.bib}

\appendix
\section{Approximate gradients and HK transport maps}
\label{ap:approx_grad}
In the case of quadratic optimal transport, solutions to the dual formulation are (locally) Lipschitz-continuous, allowing to use their gradient almost-everywhere as transport directions in \Cref{thm:Brenier,thm:McCann}. In the case of unbalanced transport the potentially infinite values of the cost functions prevent one from having such a high regularity and, even close to almost every point, there may not be a first order expansion in every direction. A way to define a similar notion in a measure theoretical setting is the use of \emph{approximate gradients} which capture the idea that at most points, one can find a vector describing the variations of the function, at least in most directions:

\begin{definition}[Approximate limit]
	For a Borel function $f:X\to \R$, we say that $f$ admits an approximate limit $l\in\R$ at $x\in X$ when Lebesgue-almost any sequence converges to $l$ through $f$ in the following sense:
	$$\tn{For any $\epsilon>0$, the set $\Delta_\epsilon=\{y~|~\abs{f(y)-l}\leq \epsilon\}$ has Lebesgue-density 1.}$$
\end{definition}

A corresponding notion of approximate differentiability for a function can then be defined as admitting both a weak limit and a weak limit for its growth rate:

\begin{definition}[Approximate gradient]
	\label{def:approx_diff}
	For a Borel function $f:\Surf\to\R$, we say that $f$ admits an approximate gradient $\tilde{\nabla} f(x)=u$ at $x\in X$ when 
	\begin{itemize}
		\item $f$ admits an approximate limit $l$ at $x$.
		\item for any $\epsilon>0$, the set
		\begin{equation}
			\Delta_\epsilon=\left\{y\in \Surf~\middle|~\Log_x^\Surf\text{ defined at y, }~\frac{\abs{f(y)-l-u\cdot \Log^X_x(y)}}{\norm{\Log^X_x(y)}}\leq\epsilon\right\}
		\end{equation}
		has Lebesgue-density $1$ at $x$.
	\end{itemize}
\end{definition}

One can show that this last definition is equivalent to finding a function $g$ differentiable at the chosen point $x\in \Surf$ such that the set $\{f=g\}$ has density 1 at $x$, and we shall use this reformulation for convenience in the upcoming proof, but we still give \cref{def:approx_diff} as it is more concrete in our opinion. For a more extensive introductions to these weaker notions of limit and gradient (albeit only in $\R^d$), we refer the reader to \cite[Section 5.5]{ambrosio2005gradient}. 

As mentioned in \Cref{sec:HK}, this notion of gradient is the natural one allowing to `differentiate' the primal dual optimality conditions $(i)$ and $(iii)$ in \Cref{prop:HK_L1_existence} and then solve the resulting equations to get an expression (valid $\mu_{0}-$almost everywhere) for $\HK$ optimal transport maps. This was \Cref{thm:McCann_HK}:
\begin{theoremnon}
	Let $\mu_0 \in \measpl(\Surf)$, $\mu_1 \in \measp(X)$ and $(\Phi_0,\Phi_1)$ be a maximizing pair for \eqref{eq:HK_L1_dual} with these measures. Then:
	\begin{enumerate}[(i)]
		\item The minimizer $\pi$ of \eqref{eq:HK_soft_primal} is unique and can be written as $\pi=(\id,T)_\# \pi_0$ for a corresponding optimal transport map $T$ and $\pi_0$ is the first marginal of $\pi$.
		\item The function $\Phi_0$ is $\mu_0$-almost everywhere approximately differentiable and the map $T$ satisfies
		\begin{equation*}
			T(x)=\Exp^X_x\left(-\kappa\arctan\left(\frac{1}{2\kappa}\frac{\norm{\tilde{\nabla}\Phi_0(x)}}{1-\Phi_0(x)/\kappa^2}\right)\frac{\tilde{\nabla}\Phi_0(x)}{\norm{\tilde{\nabla}\Phi_0(x)}}\right)
			\quad \text{for $\pi_0-$almost every $x$,}
		\end{equation*} 
		while $\tilde{\nabla}\Phi_0(x_0)=0$ for $\mu_0^\perp$-a.e. $x_0\in \Surf$, where $\mu_0^\perp$ is the part of $\mu_0$ that is singular with respect to $\pi_0$.
	\end{enumerate}	
\end{theoremnon}

\begin{proof}[Proof of \Cref{thm:McCann_HK}]
	The result and its proof are essentially minor adaptations of \cite[Proposition 16]{gallouet2021regularity} or \cite[Theorem 6.6 (iii)]{liero2018optimal}. For transparency we still give the full presentation here.
	
	Let us consider a dual optimizing pair $(\Phi_0,\Phi_1)$ for \eqref{eq:HK_L1_dual}. By the optimality conditions of \cref{prop:HK_L1_existence}, there exist two sets $A_0\subset \spt(\mu_0)$, $A_1\subset \spt(\mu_1)$ such that $\mu_i(\Surf\setminus A_i)=0$, $i=0,1$ and 
	\begin{subequations}
		\begin{align}
			& (1-\frac{\Phi_0(x_0)}{\kappa^2})(1-\frac{\Phi_1(x_1)}{\kappa^2})\geq \Cos^2(d(x_0,x_1)/\kappa)
			& & \text{for every $(x_0,x_1)\in A_0\times A_1$,}\label{eq:global_ineq}\\
			& (1-\frac{\Phi_0(x_0)}{\kappa^2})(1-\frac{\Phi_1(x_1)}{\kappa^2})=\cos^2(d(x_0,x_1)/\kappa)
			& & \text{for $\pi$-a.e.~$(x_0,x_1)\in A_0\times A_1$,}\label{eq:support_eq}\\
			& \Phi_i=\kappa^2
			& & \text{on $\spt(\mu_i^\perp)\cap A_i,~i=0,1$.}\label{eq:outside_sat}
		\end{align}
	\end{subequations}
	
	Let us first quickly notice that since $\Phi_0$ is constant equal to $\kappa^2$ $\mu_0^\perp$-almost-everywhere on $\spt(\mu_0^\perp)\cap A_0$, it has approximate gradient $0$ at each point of density 1 of this set, which is $\mu_0$-almost every point since $\mu_0\ll \vol$.
	
	We define the sub-level sets for $n\in\N\setminus\{0\}$,
	$$A_{1,n}=\left\{x_1\in A_1~\middle|~1-\frac{\Phi_1(x_1)}{\kappa^2}\geq\frac{1}{n}\right\}$$
	as well as the (Lipschitz-continuous) function on $\Surf$, 
	$$f_{0,n}:x_0\mapsto\sup_{x_1\in A_{1,n}}\frac{\Cos^2(d(x_0,x_1)/\kappa)}{1-\frac{\Phi_1(x_1)}{\kappa^2}}.$$
	It is immediate that $(1-\frac{\Phi_0}{\kappa^2})\geq f_{0,n}$ on $A_0$. Furthermore, on $\spt(\mu_0^\perp)\cap A_0$, both of these functions are identically $0$ (because any $x_0$ in that set is at distance at least $\kappa\frac{\pi}{2}$ from any point in $A_{1,n}$). On the other hand, $\pi_0$-almost any $x_0$ is associated to some $x_1\in A_{1,n}$ for some $n$, in the sense that $1-\Phi_0(x_0)/\kappa^2=\frac{\cos^2(d(x_0,x_1)/\kappa)}{1-\Phi_1(x_1)/\kappa^2}$ (by \eqref{eq:support_eq}). Therefore: 
	$$\pi_0\left(\Surf\setminus\bigcup_{n=1}^\infty\left\{x_0\in A_0~\middle|~1-\frac{\Phi_0(x_0)}{\kappa^2}=f_{0,n}(x_0)\right\}\right)=0.$$
	Furthermore, since $\mu_0$ and therefore $\pi_0$ are absolutely continuous w.r.t. the Lebesgue measure, denoting by $A'_0$ the subset of all points $x\in A_0$ with Lebesgue density 1 and such that $f_{0,n}$ be differentiable at $x$, Rademacher's theorem guarantees that $\pi_0$-almost any $x\in A_0$ is in $A'_0$ and therefore also
	$$\pi_0\left(\Surf\setminus\bigcup_{n=1}^\infty A_{0,n}\right)=0$$
	with the definition
	$$A_{0,n}\assign\left\{x_0\in A'_0~\middle|~1-\frac{\Phi_0(x_0)}{\kappa^2}=\frac{\Cos^2(d(x_0,x_1)/\kappa)}{1-\frac{\Phi_1(x_1)}{\kappa^2}}(=f_{0,n}(x_0))\tn{ for some }x_1\in A_{1,n}\right\}$$
	Now, any $x_0\in A_{0,n}$ is a point of strong differentiability for $f_{0,n}$ and therefore of approximate differentiability of $\Phi_0$. It follows that, for $\pi_0$-almost every $x_0\in \Surf$, there exists $n\in\N\setminus\{0\}$ s.t. $\tilde{\nabla}\Phi_0(x_0)=\nabla f_{0,n}(x_0)$.
	
	Let us now compute the gradient of the Lipschitz function $f_{0,n}$ (where it is defined). Consider $x_0\in A_{0,n}$ and $x_1\in A_{1,n}$, such that 
	$$f_{0,n}=1-\frac{\Phi_0(x_0)}{\kappa^2}=\frac{\cos^2(d(x_0,x_1)/\kappa)}{1-\frac{\Phi_1(x_1)}{\kappa^2}}.$$
	
	But for any other $x\in \Surf$, $\kappa^2-\Phi_1(x_1)\geq\frac{\cos^2(d(x,x_1))}{f_{0,n}(x)}=:g_n(x)$ (by definition of $f_{0,n}$) and therefore, $x_0$ is a maximum of $g_n$ i.e. $0\in\nabla^+g_n(x_0)$ where $\nabla^+g_n$ denotes the supergradient of the function $g_n$. Let us recall that this function is also sub-differentiable at $x_0$ giving us that it actually is differentiable at that point, with gradient $0$. By hypothesis, $f_{0,n}$ is differentiable and non-zero at $x_0$ and $\Cos^2(d(.,x_1))$ is sub-differentiable  (this is straightforward by super-differentiability of the squared distance on $X$, see \cite[Lemma 15]{gallouet2021regularity}), therefore $g_n$ is indeed sub-differentiable and therefore differentiable at $x_0$ (and its gradient is its only super-gradient, $0$). This implies, using the chain rule, that $d(.,x_1)$ and therefore $d^2(.,x_1)$ is also differentiable at $x_0$, with gradient $\nabla_xd^2(x_0,x_1)=-2\Log^X_{x_0}(x_1)$ (see \eqref{eq:LogMapBasic}) and:
	\begin{equation*}
		\begin{split}
			0= &  \frac{2}{\kappa}\frac{\sin\left(\frac{d(x_0,x_1)}{\kappa}\right)\cos\left(\frac{d(x_0,x_1)}{\kappa}\right)}{f_{0,n}(x_0)}\frac{\Log^X_{x_0}(x_1)}{\norm{\Log^X_{x_0}(x_1)}}-\frac{\cos^2\left(\frac{d(x_0,x_1)}{\kappa}\right)}{f_{0,n}(x_0)^2}\nabla f_{0,n}(x_0)
		\end{split}
	\end{equation*}
	But since $x_0\in A_{0,n}$,
	\begin{equation*}
		\begin{split}
			\tilde{\nabla}\Phi_0(x_0)=-\kappa^2\nabla f_{0,n}(x_0)= &  -2\kappa\tan\left(\frac{d(x_0,x_1)}{\kappa}\right)f_{0,n}(x_0)\frac{\Log^X_{x_0}(x_1)}{\norm{\Log^X_{x_0}(x_1)}}
		\end{split}
	\end{equation*}
	and using the fact that $d(x_0,x_1)<\kappa\frac{\pi}{2}$, we can conclude
	\begin{equation*}
		\begin{split}
			\Log_{x_0}(x_1)= & -\kappa\arctan\left(\frac{1}{2\kappa}\frac{\norm{\tilde{\nabla}\Phi_0(x_0)}}{1-\Phi_0(x_0)/\kappa^2}\right)\frac{\tilde{\nabla}\Phi_0(x_0)}{\norm{\tilde{\nabla}\Phi_0(x_0)}}
		\end{split}
	\end{equation*}
	and $x_1$ is in fact uniquely determined by $x_0$ (and the dual solution $\Phi_0$). Any optimal transport plan $\pi$ is therefore supported on the graph of the map $T$ given in the statement and, since the marginal densities $1-\Phi_0/\kappa^2$, $1-\Phi_1/\kappa^2$ (Proposition \ref{prop:HK_L1_existence} \eqref{item:HK_L1_marginals}) do not depend on the chosen optimal transport plan, by strict convexity of the KL divergence w.r.t.~its first argument (see Proposition \ref{prop:HKBasic} \eqref{item:HKBasicMinimizer}), the map $T$ and therefore also the optimal plan $\pi$ are unique as well.
\end{proof}

\end{document}